\documentclass[a4,12pt]{amsart}
\usepackage{amssymb,amstext,amsmath,amscd,amsthm,amsfonts,enumerate,latexsym}
\usepackage{color}
\usepackage[dvipdfmx]{graphicx}
\usepackage[all]{xy}
\theoremstyle{plain}
\newtheorem{thm}{Theorem}[section]
\newtheorem*{thm*}{Theorem}
\newtheorem*{cor*}{Corollary}

\newtheorem{prop}[thm]{Proposition}
\newtheorem{lem}[thm]{Lemma}
\newtheorem{cor}[thm]{Corollary}
\newtheorem{claim}{Claim}
\newtheorem*{claim*}{Claim}

\theoremstyle{definition}
\newtheorem{defn}[thm]{Definition}
\newtheorem{ex}[thm]{Example}
\newtheorem{rem}[thm]{Remark}

\newtheorem{ques}[thm]{Question}

\newtheorem{setting}[thm]{Setting}
\newtheorem{notation}[thm]{Notation}

\theoremstyle{remark}

\numberwithin{equation}{thm}
\def\soc{\operatorname{Soc}}
\def\xx{\text{{\boldmath$x$}}}

\def\Z{\mathbb{Z}}
\def\D{\mathcal{D}}

\def\Gdim{\operatorname{Gdim}}
\def\pd{\operatorname{pd}}
\def\Ext{\operatorname{Ext}}

\def\Ker{\operatorname{Ker}}
\def\KK{\mathbb{K}}
\def\LL{\mathbb{U}}
\def\E{\operatorname{E}}

\def\Hom{\operatorname{Hom}}
\def\RHom{\mathrm{{\bf R}Hom}}

\def\Tor{\operatorname{Tor}}

\def\a{\mathrm a}

\def\e{\mathrm{e}}
\def\m{\mathfrak m}
\def\n{\mathfrak n}

\def\p{\mathfrak p}
\def\q{\mathfrak q}

\def\K{\mathrm{K}}
\def\H{\mathrm{H}}

\newcommand{\rma}{\mathrm{a}}

\newcommand{\rme}{\mathrm{e}}

\newcommand{\rmi}{\mathrm{i}}

\newcommand{\rmr}{\mathrm{r}}

\newcommand{\rmv}{\mathrm{v}}

\newcommand{\rmA}{\mathrm{A}}

\newcommand{\rmD}{\mathrm{D}}
\newcommand{\rmE}{\mathrm{E}}

\newcommand{\rmH}{\mathrm{H}}
\newcommand{\rmI}{\mathrm{I}}

\newcommand{\rmK}{\mathrm{K}}

\newcommand{\rmQ}{\mathrm{Q}}

\newcommand{\rmU}{\mathrm{U}}

\newcommand{\calD}{\mathcal{D}}

\newcommand{\calR}{\mathcal{R}}

\newcommand{\calX}{\mathcal{X}}

\newcommand{\fka}{\mathfrak{a}}

\newcommand{\fkm}{\mathfrak{m}}
\newcommand{\fkn}{\mathfrak{n}}

\newcommand{\fkp}{\mathfrak{p}}
\newcommand{\fkq}{\mathfrak{q}}

\newcommand{\fkM}{\mathfrak{M}}
\newcommand{\fkN}{\mathfrak{N}}

\newcommand{\mapright}[1]{%
\smash{\mathop{%
\hbox to 1cm{\rightarrowfill}}\limits^{#1}}}

\newcommand{\mapleft}[1]{%
\smash{\mathop{%
\hbox to 1cm{\leftarrowfill}}\limits_{#1}}}

\def\depth{\operatorname{depth}}
\def\Supp{\operatorname{Supp}}

\def\Ass{\operatorname{Ass}}
\def\height{\mathrm{ht}}
\def\Spec{\operatorname{Spec}}

\def\gr{\mbox{\rm gr}}

\def\R{{\mathcal R}}

\def\Y{{\mathcal Y}}

\def\R{{R}}

\def\yy{\text{\boldmath $y$}}

\def\PP{\mathbb{P}}
\def\II{\mathbb{I}}

\begin{document}
\title[Almost Gorenstein rings -- towards a theory of higher dimension --]{Almost Gorenstein rings\\
-- towards a theory of higher dimension --}
\author{Shiro Goto}
\address{Department of Mathematics, School of Science and Technology, Meiji University, 1-1-1 Higashi-mita, Tama-ku, Kawasaki 214-8571, Japan}
\email{goto@math.meiji.ac.jp}
\author{Ryo Takahashi}
\address{Graduate School of Mathematics, Nagoya University, Furocho, Chikusa-ku, Nagoya 464-8602, Japan}
\email{takahashi@math.nagoya-u.ac.jp}
\author{Naoki Taniguchi}
\address{Department of Mathematics, School of Science and Technology, Meiji University, 1-1-1 Higashi-mita, Tama-ku, Kawasaki 214-8571, Japan}
\email{taniguti@math.meiji.ac.jp}
\thanks{SG was partly supported by JSPS Grant-in-Aid for Scientific Research 22540054; RT was partly supported by JSPS Grant-in-Aid for Young Scientists 22740008 and by JSPS Postdoctoral Fellowships for Research Abroad.}

\begin{abstract} The notion of almost Gorenstein local ring introduced by V. Barucci and R. Fr\"oberg for one-dimensional Noetherian local rings which are analytically unramified has been generalized by S. Goto, N. Matsuoka and T. T. Phuong to one-dimensional Cohen-Macaulay local rings, possessing canonical ideals. The present purpose is to propose a higher-dimensional notion and develop the basic theory. The graded version is also posed and explored.
\end{abstract}

\maketitle
{\footnotesize \tableofcontents}
\section{Introduction}\label{intro}
For the last fifty years, commutative algebra has been concentrated mainly in the study of Cohen-Macaulay rings/modules and has performed huge achievements (\cite{BH}). While tracking the development, the authors often encounter non-Gorenstein Cohen-Macaulay rings in divers branches of (and related to) commutative algebra.  On all such occasions, they have a query why there are so many Cohen-Macaulay rings which are not Gorenstein rings. Gorenstein rings are, of course,  defined by local finiteness of self-injective dimension (\cite{Bass}), enjoying beautiful symmetry. However as a view from the very spot, there is a substantial estrangement between two conditions of finiteness and infiniteness of self-injective dimension, and researches for the fifty years also show that Gorenstein rings turn over some part of their roles to canonical modules (\cite{HK}). It seems, nevertheless, still reasonable to ask for a new class of non-Gorenstein Cohen-Macaulay rings that could be called \textit{almost Gorenstein} and  are good next to Gorenstein rings. This observation has already motivated the research \cite{GMP} of one-dimensional case. The second step should be  to detect the notion of almost Gorenstein local/graded ring of higher dimension and develop the theory.

Almost Gorenstein local rings of dimension one were originally introduced in 1997 by Barucci and Fr\"oberg \cite{BF} in the case where the local rings are analytically unramified. As was mentioned by \cite{B} as for the proof of \cite[Proposition 25]{BF}, their framework was not sufficiently flexible  for the analysis of one-dimensional analytically ramified local rings. This observation has inspired  Goto, Matsuoka, and Phuong \cite{GMP}, where they posed a modified definition of one-dimensional almost Gorenstein local rings, which works well also in the case where the rings are analytically ramified. The present research aims to go beyond \cite{GMP} towards a theory of higher dimensional cases, asking for possible extensions of results known by \cite{B, BF, BDF, GMP}.

To explain our aim and motivation more precisely, let us start on our definition.

\begin{defn}\label{1.1}
Let $R$ be a Noetherian local ring with maximal ideal $\fkm$. Then $R$ is said to be an almost Gorenstein local ring, if the following conditions are satisfied.
\begin{enumerate}[(1)]
\item
$R$ is a Cohen-Macaulay local ring, which possesses the canonical module $\K_R$ and 
\item
there exists an exact sequence
$$0 \to R \to \mathrm{K}_R \to C \to 0$$
of $R$-modules such that $\mu_R(C) = \e_\fkm^0(C)$.
\end{enumerate}
Here $\mu_R(C)$ $($resp. $\e_\fkm^0(C)$$)$ denotes the number of elements in a minimal system of generators for $C$ $($resp. the multiplicity of $C$ with respect to $\fkm$$)$.
\end{defn}

With this definition every Gorenstein local ring is almost Gorenstein (take $C = (0)$) and the converse is also true, if $\dim R = 0$. In the exact sequence quoted in Definition \ref{1.1} (2), if $C \ne (0)$, then $C$ is a Cohen-Macaulay $R$-module with $\dim_RC = \dim R - 1$ and one has the equality $\fkm C = (f_2, f_3, \ldots, f_d)C$ for some elements $f_2, f_3, \ldots, f_d \in \fkm$ ($d = \dim R$), provided the residue class field $R/\fkm$ of $R$ is infinite. Hence $C$ is a \textit{maximally generated Cohen-Macaulay} $R$-module in the sense of \cite{BHU}, which is called in the present paper an \textit{Ulrich} $R$-module. Therefore, roughly speaking, our Definition \ref{1.1} requires  that if $R$ is an almost Gorenstein local ring, then $R$ might be a non-Gorenstein local ring but the ring $R$ can be embedded into its canonical module $\K_R$ so that the difference $\K_R/R$ should be tame and well-behaved.

In the case where $\dim R = 1$, if $R$ is an almost Gorenstein local ring, then $\fkm C = (0)$ and $R$ is an almost Gorenstein local ring exactly in the sense of \cite[Definition 3.1]{GMP}. The converse is also true, if $R/\fkm$ is infinite. (When the field $R/\fkm$ is too small, the converse is not true in general; see Remark \ref{3.4'} and \cite[Remark 2.10]{GMP}.) With Definition \ref{1.1}, as  we will later show, many results of \cite{GMP} of dimension one are extendable over higher-dimensional local rings, which supports the appropriateness of our definition.

Let us now state our results, explaining how this paper is organized. In Section \ref{ulmd} we give a brief survey on Ulrich modules, which we will need throughout this paper. In Section \ref{defag} we explore basic properties of almost Gorenstein local rings, including the so-called non-zerodivisor characterization. 
In Section \ref{chex}, we will give a characterization of almost Gorenstein local rings in terms of the existence of certain exact sequences of $R$-modules.
Let $M$ be an $R$-module.
For a sequence $\xx=x_1, x_2, \ldots, x_n$ of elements in $R$, the Koszul complex of $M$ associated to $\xx$ is denoted by $\KK_\bullet(\xx,M)$.
For each $z\in M$, we define a complex
$$
\LL(z,M)=(\cdots \to \underset{\begin{smallmatrix}\phantom{o}\\2\end{smallmatrix}}{0}\to\underset{\begin{smallmatrix}\phantom{o}\\1\end{smallmatrix}}{R}\xrightarrow{\varphi}\underset{\begin{smallmatrix}\phantom{o}\\0\end{smallmatrix}}{M}\to\underset{\begin{smallmatrix}\phantom{o}\\-1\end{smallmatrix}}{0} \to \cdots),
$$
where the map $\varphi$ is given by $a\mapsto az$. Let us say that an $R$-complex $C=(\cdots\to C_2\to C_1\to C_0\to0)$ is called an acyclic complex over $M$, if $\H_0(C)\cong M$ and $\H_i(C)=(0)$ for all $i>0$. With this notation the main result of Section \ref{chex} is stated as follows. 

\begin{thm}
Let $(R,\m)$ be a Cohen-Macaulay local ring with $\dim R = d \ge 1$ and the Cohen-Macaulay type $r$.
Suppose that $R$ admits the canonical module $\K_R$ and that the residue class field $R/\m$ of $R$ is infinite.
Then the following conditions are equivalent.
\begin{enumerate}[\rm(1)]
\item
$R$ is an almost Gorenstein local ring.
\item
There exist an $R$-sequence $\xx=x_1, x_2, \ldots,x_{d-1}$ and an element $y\in\K_R$ such that
$
\KK_\bullet(\xx,R)\otimes_R\LL(y,\K_R)
$
is an acyclic complex over $k^{r-1}$.
\item
There exist an $R$-sequence $\xx$ $($not necessarily of length $d-1$ $)$ and an element $y\in\K_R$ such that
$
\KK_\bullet(\xx,R)\otimes_R\LL(y,\K_R)
$
is an acyclic complex over an $R$-module annihilated by $\m$.
\end{enumerate}
\end{thm}

In Section \ref{chcan} we give the following characterization of almost Gorenstein local rings in terms of canonical ideals. When $\dim R = 1$, this result corresponds to \cite[Theorem 3.11]{GMP}.

\begin{thm}\label{1.2}
Let $(R,\fkm)$ be a Cohen-Macaulay local ring with $d = \dim R \ge 1$ and infinite residue class field. Let $I ~(\ne R)$ be an ideal of $R$ and assume that $I \cong \K_R$ as an $R$-module. Then the following conditions are equivalent.
\begin{enumerate}[\rm(1)]
\item
$R$ is an almost Gorenstein local ring.
\item
$R$ contains a parameter ideal $Q = (f_1, f_2, \ldots, f_d)$ such that $f_1 \in I$ and
$
\fkm(I + Q) = \fkm Q.
$
\end{enumerate}
\end{thm}

With the same notation as Theorem \ref{1.2}, if $R$ is not a Gorenstein ring, we then have  $\e_1(I + Q) = \rmr (R)$ (here $\e_1(I+Q)$ (resp. $\rmr (R)$) denotes the first Hilbert coefficient of the ideal $I + Q$ of $R$ (resp. the Cohen-Macaulay type of $R$)). A structure theorem of the Sally module ${\mathcal S}_Q(I+Q)$ of $I+Q$ with respect to the reduction $Q$ shall be described. These results reasonably extend the corresponding ones in \cite[Theorem 3.16]{GMP} to higher-dimensional local rings.

In Section \ref{idealiz} we study the question of when the idealization  $A = R \ltimes X$ of a given $R$-module $X$ is an almost Gorenstein local ring. Our goal is the following, which extends  \cite[Theorem 6.5]{GMP} to higher-dimensional cases.

\begin{thm}\label{5.4}
Let $(R,\fkm)$ be a Cohen-Macaulay local ring of dimension $d \ge 1$, which possesses the canonical module $\K_R$. 
Suppose that $R/\fkm$ is infinite. Let $\fkp \in \Spec R$ such that  $R/\fkp$ is a regular local ring of dimension $d-1$. Then the following conditions are equivalent.
\begin{enumerate}[\rm(1)]
\item
$A = R \ltimes \fkp$ is an almost Gorenstein local ring.
\item
$R$ is an almost Gorenstein local ring. 
\end{enumerate}
\end{thm}

In Section \ref{pG} we explore a special class of almost Gorenstein local rings, which we call semi-Gorenstein. A structure theorem of minimal free resolutions of semi-Gorenstein local rings shall be given. The semi-Gorenstein property is preserved under localization while the almost Gorenstein property is not, which we will show later; see Section \ref{assgr}.

In Section \ref{aggr} we search for possible definitions of almost Gorenstein graded rings. Let $R = \bigoplus_{n \ge 0}R_n$ be a Cohen-Macaulay graded ring with $k = R_0$ a local ring. Assume that $R$ possesses the graded canonical module $\K_R$. This condition is equivalent to saying that $k$ is a homomorphic image of a Gorenstein local ring (\cite{GW1, GW2}). Let $\fkM$ denote the unique graded maximal ideal of $R$ and let $a = \rma (R)$ be the $a$-invariant of $R$. Hence $a = \operatorname{max} \{n \in \Bbb Z \mid [\rmH_\fkM^d(R)]_n \ne (0)\}$ (\cite[Definition (3.1.4)]{GW1}), where $\{[\rmH_\fkM^d(R)]_n\}_{n \in \Bbb Z}$ denotes the homogeneous components of the $d$-th graded local cohomology module $\rmH_\fkM^d(R)$ of $R$ with respect to $\fkM$.  
With this notation our definition of almost Gorenstein graded ring is stated as follows, which we discuss in Section \ref{aggr}.

\begin{defn}\label{1.3}
We say that $R$ is an almost Gorenstein graded ring, if there exists an exact sequence
$$0 \to R \to \mathrm{K}_R(-a) \to C \to 0$$
of graded $R$-modules with $\mu_R(C) = \e_\fkM^0(C)$. Here $\mathrm{K}_R(-a)$ denotes the graded $R$-module whose underlying $R$-module is the same as that of $\K_R$ and whose grading is given by $[\mathrm{K}_R(-a)]_n = [\mathrm{K}_R]_{n-a}$ for all $n \in \Bbb Z$.
\end{defn}

In Section \ref{assgr} we study almost Gorensteinness in the graded rings associated to filtrations of 
ideals. We shall prove that the almost Gorenstein property of base local rings is inherited from that of the associated graded rings with a certain condition on the Cohen-Macaulay type. In general, local rings of an almost Gorenstein local ring are not necessarily almost Gorenstein, which we will show in this section; see Remark \ref{3}.

In Section \ref{homog} we explore Cohen-Macaulay homogeneous rings $R = k[R_1]$ over an infinite field $k = R_0$. We shall prove the following, which one can directly apply, for instance, to the Stanley-Reisner rings $R=k[\Delta]$ of simplicial complexes $\Delta$ over $k$.

\begin{thm}\label{1.4}
Let $R = k[R_1]$ be a Cohen-Macaulay homogeneous ring over an infinite field $k$ and assume that $R$ is not a Gorenstein ring. Let $d = \dim R \ge 1$ and set $a = \rma (R)$. Then the following conditions are equivalent.
\begin{enumerate}[\rm(1)]
\item
$R$ is an almost Gorenstein graded ring.
\item
The total ring $\rmQ (R)$ of fractions of $R$ is a Gorenstein ring and $a = 1-d$.
\end{enumerate}
\end{thm}

In Section \ref{tancon} we study the relation between the almost Gorensteinness of Cohen-Macaulay local rings $(R,\fkm)$ and their tangent cones $\gr_\fkm (R) =\bigoplus_{n \ge 0}\fkm^n/\fkm^{n+1}$. We shall prove, provided $R/\fkm$ is infinite and $\rmv (R) = \e_\fkm^0(R) + \dim R - 1$ (here $\mathrm{v} (R)$ denotes the embedding dimension of $R$), that $R$ is an almost Gorenstein local ring if and only if  $\rmQ (\gr_\fkm(R))$ is a Gorenstein ring, which will eventually show that every two-dimensional rational singularity is an almost Gorenstein local ring (Corollary \ref{8.5}).

In the final section we shall prove that every one-dimensional Cohen-Macaulay complete local ring of finite Cohen-Macaulay representation type is an almost Gorenstein local ring, if it possesses a coefficient field of characteristic $0$.

As is confirmed in Sections \ref{aggr}, \ref{assgr}, \ref{homog}, our definition of almost Gorenstein graded rings works well to analyze divers graded rings. We, however, note here the following. By definition, the ring $R_\fkM$ is an almost Gorenstein local ring, if $R$ is an almost Gorenstein graded ring with unique graded maximal ideal $\fkM$, but as Example \ref{6.5} shows, the converse is not true in general. In fact, for the example, one has $\rma (R) = -2$ and there is no exact sequence $0 \to R \to \mathrm{K}_R(2) \to C \to 0$ of graded $R$-modules such that $\mu_R(C) = \e_\fkM^0(C)$, while there exists an exact sequence $0 \to R \to \mathrm{K}_R(3) \to D \to 0$ such that $\mu_R(D) = \e_\fkM^0(D)$. The example seems to suggest  the existence of alternative and more flexible definitions of almost Gorensteinness for graded rings. We would like to leave the quest to forthcoming researches.

In what follows, unless otherwise specified, let $R$ denote a Noetherian local ring with maximal ideal $\fkm$. For each finitely generated $R$-module $M$, let $\mu_R(M)$ (resp. $\ell_R(M)$) denote the number of elements in a minimal system of generators of $M$ (resp. the length of $M$). We denote by $\e_\fkm^0(M)$ the multiplicity of $M$ with respect to $\fkm$.


\section{Survey on Ulrich modules}\label{ulmd}

Let $R$ be a Noetherian local ring with maximal ideal $\fkm$. The purpose of this section is to summarize some preliminaries on Ulrich modules, which we will use throughout this paper.
We begin with the following.

\begin{defn}\label{2.1}
Let $M~(\ne (0))$ be a finitely generated $R$-module. Then $M$ is said to be an Ulrich $R$-module, if $M$ is a Cohen-Macaulay $R$-module and $\mu_R(M) = \rme_\fkm^0(M)$.
\end{defn}


\begin{prop}\label{2.2}
Let $M$ be a finitely generated $R$-module of dimension $s \ge 0$. Then the following assertions hold true.
\begin{enumerate}[\rm(1)]
\item
Suppose $s = 0$. Then $M$ is an Ulrich $R$-module if and only if $\fkm M = (0)$, that is $M$ is a vector space over the field $R/\fkm$.
\item
Suppose that $M$ is a Cohen-Macaulay $R$-module. If $\fkm M = (f_1, f_2, \ldots, f_s)M$ for some $f_1, f_2, \ldots, f_s \in \fkm$, then $M$ is an Ulrich $R$-module. The converse is also true, if $R/\fkm$ is infinite. $($We actually have $\fkm M = (f_1, f_2, \ldots, f_s)M$ for any elements $f_1, f_2, \ldots, f_s \in \fkm$ whose images in $R/[(0):_RM]$ generate a minimal reduction of the maximal ideal of $R/[(0):_RM]$.$)$  When this is the case, the elements $f_1, f_2, \ldots, f_s$ form a part of a minimal system of generators for $\fkm$.
\item
Let $\varphi : R \to S$ be a flat local homomorphism of Noetherian local rings such that $S/\fkm S$ is a regular local ring. Then $M$ is an Ulrich $R$-module if and only if $S \otimes_RM$ is an Ulrich $S$-module.  
\item
Let $M$ be an Ulrich $R$-module with $s = \dim_RM \ge 1$. Let $f \in \fkm$ and assume that $f$ is  superficial for $M$ with respect to $\fkm$. Then $M/fM$ is an Ulrich $R$-module of dimension $s - 1$.
\item
Let $f \in \fkm$ and assume that $f$ is  $M$-regular. If $M/fM$ is an Ulrich $R$-module, then $M$ is an Ulrich $R$-module and $f \not\in \fkm^2$.
\end{enumerate}
\end{prop}

\begin{proof}

(1) This follows from the facts that $\mu_R(M) = \ell_R(M/\fkm M)$ and $\rme_\fkm^0(M) = \ell_R(M)$.

(2) Suppose that $\fkm M = (f_1, f_2, \ldots, f_s)M$ for some $f_1, f_2, \ldots, f_s \in \fkm$. Then $f_1, f_2, \ldots, f_s$ is a system of parameters of $M$ and $\fkm^{n+1} M = (f_1, f_2, \ldots, f_s)^{n+1}M$ for all $n \ge 0$. Hence $\ell_R(M/\fkm^{n+1}M) = \ell_R(M/(f_1, f_2, \ldots, f_s)M){\cdot}\binom{n+s}{s}$ and therefore $\rme_\fkm^0(M) = \ell_R(M/(f_1, f_2, \ldots, f_s)M) = \ell_R(M/\fkm M)$, so that   $M$ is an Ulrich $R$-module. Let $\overline{R} = R/[(0):_RM]$  and let $\overline{f_i}$ denote the image of $f_i$ in $\overline{R}$. We then have $\overline{\fkm}M = (\overline{f_1}, \overline{f_2}, \ldots, \overline{f_s})M$, where $\overline{\fkm} = \fkm \overline{R}$. Hence $(\overline{f_1}, \overline{f_2}, \ldots, \overline{f_s})$ is a minimal reduction of $\overline{\fkm}$, because $M$ is a faithful $\overline{R}$-module, so that $f_1, f_2, \ldots, f_s$ form a part of a minimal system of generators for the maximal ideal $\fkm$.

Conversely, suppose that $R/\fkm$ is infinite and that $M$ is an Ulrich $R$-module. Let us choose elements $f_1, f_2, \ldots, f_s \in \fkm$ so that $(\overline{f_1}, \overline{f_2}, \ldots, \overline{f_s})$ is a minimal reduction of $\overline{\fkm}$. Then $\rme_\fkm^0(M) = \rme_{\overline{\fkm}}^0(M) = \rme_{(\overline{f_1}, \overline{f_2}, \ldots, \overline{f_s})}^0(M) = \ell_R(M/(f_1, f_2, \ldots, f_s)M)$. Hence $\fkm M = (f_1, f_2, \ldots, f_s)M$ as $\ell_R(M/\fkm M) = \rme_\fkm^0(M)$.

(3)  Choose a regular system $g_1, g_2, \ldots, g_n \in \fkn$ of parameters  for the regular local ring $S/\fkm S$~(here $n= \dim S/\fkm S$) and set $\overline{S} = S/(g_1, g_2, \ldots, g_n)S$. Then the composite map $\psi : R \to S \to \overline{S}$ is flat (\cite[Lemma 1.23]{HK}) and  
$$
(S\otimes_RM)/(g_1, g_2, \ldots, g_n)(S\otimes_RM) \cong \overline{S}\otimes_RM,
$$
so that passing to the homomorphism $\psi$, we may assume $\fkm S = \fkn$. We then have
$$
\mu_S(S\otimes_RM) = \mu_R(M), \ \ \e^0_{\fkn}(S\otimes_RM) = \e^0_{\fkm}(M).
$$
Hence $M$ is an Ulrich $R$-module if and only if $S\otimes_RM$ is an Ulrich $S$-module.

(4) Since $f$ is superficial for $M$ with respect to $\fkm$ and $s > 0$, $f$ is $M$-regular and $\e_\fkm^0(M/fM) = \e_\fkm^0(M)$. Therefore  $M/fM$ is a Cohen-Macaulay $R$-module, and consequently $M/fM$ is an Ulrich $R$-module, because $\mu_R(M/fM) = \mu_R(M) = \e_\fkm^0(M) = \e_{\fkm}^0(M/fM)$.

(5) We put $R(X) = R[X]_{\fkm R[X]}$ and $S(X) = S[X]_{\fkn S[X]}$, where $X$ is an indeterminate. Then, since $\fkm R[X] = \fkn  S[X] \cap R[X]$, we get a flat local homomorphism $\psi : R(X) \to S(X)$,  extending  $\varphi:R\to S$.
Because $\mu_{R(X)}(R(X) \otimes_RM) = \mu_R(M)$ and $\rme_{\fkm R(X)}^0(R(X) \otimes_RM) = \e_\fkm^0(M)$, $R(X) \otimes_RM$ is an Ulrich $R(X)$-module. For the same reason, $S(X)\otimes_S(S\otimes_RM)$ is an Ulrich $S(X)$-module if and only if $S\otimes_RM$ is an Ulrich $S$-module. Therefore, since $S(X)/\fkm S(X) = (S/\fkm S)(X)$ is a regular local ring, passing to the homomorphism $\psi : R(X) \to S(X)$, without loss of generality we may assume that the residue class field $R/\fkm$ of $R$ is infinite. We now choose elements $f_2, f_3, \ldots, f_s \in \fkm$ so that $\fkm{\cdot}(M/fM) = (f_2, f_3, \ldots, f_s){\cdot}(M/fM)$. Then $\fkm M = (f_1, f_2, \ldots, f_s)M$ with $f_1 = f$. Therefore by assertion (2), $M$ is an Ulrich $R$-module and $f \not\in \fkm^2$.
\end{proof}

\section{Almost Gorenstein local rings}\label{defag}

Let $R$ be a Cohen-Macaulay local ring with maximal ideal $\fkm$ and $d = \dim R \ge 0$, possessing the canonical module $\K_R$. Hence $R$ is a homomorphic image of a Gorenstein ring (\cite{R}). The purpose of this section is to define almost Gorenstein local rings and explore their basic properties.

We begin with the following.

\begin{lem}\label{3.1}
Let $R \overset{\varphi}{\longrightarrow} \K_R \to C \to 0$ be an exact sequence of $R$-modules. Then the following assertions hold true.
\begin{enumerate}[\rm(1)]
\item
If $\dim_RC < d$, then $\varphi$ is injective and the total ring  $\rmQ (R)$ of fractions of  $R$ is a Gorenstein ring.
\item
Suppose that  $\varphi$ is injective. If $C \ne (0)$, then $C$ is a Cohen-Macaulay $R$-module of dimension $d-1$.
\item
If $\varphi$ is injective and $d = 0$, then $\varphi$ is an isomorphism.
\end{enumerate}
\end{lem}

\begin{proof}
(1) Let $L = \operatorname{Ker} \varphi$ and assume that $L \ne (0)$. Choose $\fkp \in \Ass_RL$ and we have the exact sequence $0 \rightarrow L_\fkp \rightarrow R_\fkp \xrightarrow{\varphi_\fkp} (\K_R)_\fkp \rightarrow C_\fkp \rightarrow 0$ of $R_\fkp$-modules. Since $\fkp \in \Ass R$ and $\dim_RC < d$, we get $C_\fkp = (0)$, whence $\varphi_\fkp$ is an epimorphism. Therefore, because  $(\K_R)_\fkp \cong \K_{R_\fkp}$ (\cite[Korollar 6.2]{HK}) and $\ell_{R_\fkp}(\K_{R_\fkp}) = \ell_{R_\fkp}(R_\fkp)$, $\varphi_\fkp$ is necessarily an isomorphism and hence $L_\fkp = (0)$, which is impossible. Thus $L= (0)$ and $\varphi$ is injective. The second assertion is clear, because $R_\fkp \cong (\K_R)_\fkp \cong \K_{R_\fkp}$ for every $\fkp \in \Ass R$.

(2) Let $\fkp \in \Supp_RC$ with $\dim R/\fkp = \dim_RC$. If $\dim_RC =d$, then $\fkp \in \Ass R$ and hence $\ell_{R_\fkp}(R_\fkp) = \ell_{R_\fkp}(\K_{R_\fkp}) = \ell_{R_\fkp}((\K_R)_\fkp)$, so that $C_\fkp = (0)$, because the homomorphism $\varphi_\fkp : R_\fkp \to (\K_R)_\fkp$ is injective, which  is impossible. Hence $\dim_RC < d$, while we get $\depth_RC \ge d-1$, applying the depth lemma to the exact sequence $0 \to R \to \K_R \to C \to 0$. Thus $C$ is a Cohen-Macaulay $R$-module of dimension $d - 1$.

(3) This is clear. 
\end{proof}

\begin{rem}\label{3.2}
Suppose that $d > 0$ and that $\rmQ (R)$ is a Gorenstein ring. Then $R$ contains an ideal $I ~(\ne R)$ such that $I \cong \K_R$ as an $R$-module. When this is the case, $R/I$ is a Gorenstein local ring of dimension $d-1$ (\cite[Satz 6.21]{HK}).
\end{rem}

We are now ready to define almost Gorenstein local rings.

\begin{defn}\label{3.3}
Let $(R,\fkm)$ be a Cohen-Macaulay local ring which possesses the canonical module $\K_R$. Then $R$ is said to be an almost Gorenstein local ring, if there is an exact sequence $0 \to R \to \K_R \to C \to 0$ of $R$-modules such that $\mu_R(C) = \e_\fkm^0(C)$.
\end{defn}

In Definition \ref{3.3}, if $C \ne (0)$, then $C$ is an Ulrich $R$-module of dimension $d-1$ (Definition \ref{2.1} and Lemma \ref{3.1} (2)). Note that every Gorenstein local ring $R$ is almost Gorenstein (take $C = (0)$) and that $R$ is a Gorenstein local ring, if $R$ is an almost Gorenstein local ring of dimension $0$ (Lemma \ref{3.1} (3)).

Almost Gorenstein local rings were defined in 1997 by Barucci and Fr\"oberg \cite{BF} in the case where $R$ is analytically unramified and $\dim R = 1$. Goto, Matsuoka and Phuong \cite{GMP} extended the notion to the case where $R$ is not necessarily analytically unramified but still of dimension one. Our definition \ref{3.3} is a higher-dimensional proposal. In fact we have the following.

\begin{prop}\label{3.4}
Let $(R,\fkm)$ be a one-dimensional Cohen-Macaulay local ring. If $R$ is an almost Gorenstein local ring in the sense of Definition $\ref{3.3}$, then $R$ is an almost Gorenstein local ring in the sense of \cite[Definition 3.1]{GMP}.
The converse also holds, when $R/\fkm$ is infinite.
\end{prop}

\begin{proof}
Firstly, assume that $R$ is an almost Gorenstein local ring in the sense of Definition \ref{3.3} and choose an exact sequence $0 \to R \overset{\varphi}{\longrightarrow} I \to C \to 0$ of $R$-modules so that $\mu_R(C) = \e_\fkm^0(C)$, where $I ~(\ne R)$ is an ideal of $R$ such that $I \cong \K_R$ as an $R$-module (Lemma \ref{3.1} (1) and Remark \ref{3.2}). Then, because $\fkm C = (0)$ by Proposition \ref{2.2} (1), we get $\fkm I \subseteq (f)$, where $f = \varphi (1)$. We set  $Q = (f)$. Then, since $\fkm Q \subseteq \fkm I \subseteq Q$, we have either $\fkm Q = \fkm I$ or $\fkm I = Q$. If $\fkm I = \fkm Q$, then $Q$ is a reduction of $I$, so that $R$ is an almost Gorenstein local ring in the sense of \cite[Definition 3.1]{GMP} (see \cite[Theorem 3.11]{GMP} also). If $\fkm I = Q$, then the maximal ideal $\fkm$ of $R$ is invertible, so that $R$ is a discrete valuation ring. Hence in any case, $R$ is an almost Gorenstein local ring in the sense of \cite{GMP}.

Conversely, assume that $R$ is an almost Gorenstein local ring in the sense of \cite{GMP} and that $R/\fkm$ is infinite. Let us choose an $R$-submodule $K$ of $\rmQ (R)$ such that $R \subseteq K \subseteq \overline{R}$  and $K \cong \K_R$ as an $R$-module, where $\overline{R}$ denotes the integral closure of $R$ in $\rmQ (R)$. Then by \cite[Theorem 3.11]{GMP}, we get $\fkm K \subseteq R$, and therefore $R$ is an almost Gorenstein local ring in the sense of Definition \ref{3.3} (use Proposition \ref{2.2} (1)).
\end{proof}

\begin{rem}\label{3.4'}
When the field $R/\fkm$ is finite, $R$ is not necessarily an almost Gorenstein local ring in the sense of Definition \ref{3.3}, even though $R$ is an almost Gorenstein local ring in the sense of \cite{GMP}. The ring $$R = k[[X, Y, Z]]/[(X,Y)\cap (Y, Z) \cap (Z,X)]$$ is a typical example, where $k[[X,Y, Z]]$ is the formal power series ring over $k = \Bbb Z/(2)$ (\cite[Remark 2.10]{GMP}). 
This example also shows that $R$ is not necessarily an almost Gorenstein local ring in the sense of Definition \ref{3.3}, even if it becomes an almost Gorenstein local ring in the sense of Definition \ref{3.3}, after enlarging the residue class field $R/\fkm$ of $R$. 
\end{rem}

We note the following.

\begin{prop}\label{6131104}
Suppose that $R$ is not a Gorenstein ring and consider the following two conditions.
\begin{enumerate}[\rm(1)]
\item
$R$ is an almost Gorenstein local ring.
\item
There exist an exact sequence
$
0 \to R \to \rmK_R \to C \to 0
$
of $R$-modules, a non-zerodivisor $f \in (0):_RC$, and a parameter ideal $\fkq ~(\subsetneq R)$ for $R/(f)$ such that $\m C=\fkq C$.
\end{enumerate}
Then the implication {\rm(2)} $\Rightarrow$ {\rm(1)} holds. If $R/\fkm$ is infinite, the reverse implication {\rm(1)} $\Rightarrow$ {\rm(2)} is also true.
\end{prop}

\begin{proof}
(2) $\Rightarrow$ (1)  We have $C \ne (0)$, so that by Lemma \ref{3.1} (2) $C$ is a Cohen-Macaulay $R$-module of dimension $d-1$. Hence Proposition \ref{2.2} (2) shows $C$ is an Ulrich $R$-module, because $\fkm C = \fkq C$.

(1) $\Rightarrow$ (2)  We take  an exact sequence $0 \to R \to \rmK_R \to C \to 0$ of $R$-modules such that $C$ is an Ulrich $R$-module of dimension $d-1$. Hence $[(0):_RC]$ contains a non-zerodivisor $f$ of $R$. We choose elements $\{f_i\}_{2 \le i \le d}$ of $\fkm$ so that their images in $R/(f)$ generate a minimal reduction of the maximal ideal of $R/(f)$. We then have $\fkm C = \fkq C$ by Proposition \ref{2.2} (2), because the images of $\{f_i\}_{2 \le i \le d}$ in $R/[(0):_RC]$ also generate a minimal reduction of the maximal ideal of $R/[(0):_RC]$. 
\end{proof}

Let us now explore basic properties of almost Gorenstein local rings. We begin with the non-zerodivisor characterization.

\begin{thm}\label{3.9}
Let $f \in \fkm$ and assume that $f$ is $R$-regular.
\begin{enumerate}[\rm(1)]
\item If $R/(f)$ is an almost Gorenstein local ring, then $R$ is an almost Gorenstein local ring. If $R$ is moreover not a Gorenstein ring, then $f\not\in \fkm^2$.
\item Conversely, suppose that $R$ is an almost Gorenstein local ring which is  not a Gorenstein ring. Consider the exact sequence
$$0 \to R \to \rmK_R \to C \to 0$$
of $R$-modules such that  $\mu_R(C) = \e_\fkm^0(C)$. If $f$ is  superficial for $C$ with respect to $\fkm$ and $d \ge 2$, then $R/(f)$ is an almost Gorenstein local ring.
\end{enumerate}
\end{thm}

\begin{proof} We set $\overline{R} = R/(f)$. Remember that $\rmK_R/f\rmK_R = \rmK_{\overline{R}}$ (\cite[Korollar 6.3]{HK}), because $f$ is $R$-regular.  

(1) We choose an exact sequence
$0 \to \overline{R} \overset{\psi}{\longrightarrow} \K_{\overline{R}}\to D \to 0$
of $\overline{R}$-modules so that $D$ is an Ulrich $\overline{R}$-module of dimension $d -2$. Let $\xi \in \K_R$ such that $\psi (1) = \overline{\xi}$, where $\overline{\xi}$ denotes the image of $\xi$ in $\K_{\overline{R}} = \K_R/f\K_R$. We now consider the exact sequence $$R \overset{\varphi}{\longrightarrow} \K_R \to C \to 0$$
of $R$-modules with $\varphi (1) = \xi$. Then, because $\psi = \overline{R} \otimes_R\varphi$, we get $D = C/fC$, whence  $\dim_RC < d$, because $\dim_RD = d-2$. Consequently, by Lemma \ref{3.1} (1) the homomorphism $\varphi$ is injective, and hence by Lemma \ref{3.1} (2), $C$ is a Cohen-Macaulay $R$-module of dimension $d-1$.  Therefore, $f$ is $C$-regular, so that  by Proposition \ref{2.2} (5), $C$ is an Ulrich $R$-module and $f \not\in \fkm^2$. Hence $R$ is almost Gorenstein. 

(2) The element $f$ is $C$-regular, because $f$ is superficial for $C$ with respect to $\fkm$ and $\dim_RC = d-1 > 0$. Therefore the exact sequence $0 \to R \to \rmK_R \to C \to 0$ gives rise to the exact sequence of $\overline{R}$-modules
$$0 \to \overline{R} \to \rmK_{\overline{R}} \to C/fC \to 0,$$
where $C/fC$ is an Ulrich $\overline{R}$-module by Proposition \ref{2.2} (4). Hence $\overline{R}$ is almost Gorenstein.
\end{proof}

The following is a direct consequence of Theorem \ref{3.9} (1).

\begin{cor}\label{3.10}
Suppose that $d > 0$. If $R/(f)$ is an almost Gorenstein local ring for every non-zerodivisor $f \in \fkm$, then $R$ is a Gorenstein local ring.
\end{cor}

We are now interested in the question of how the almost Gorenstein property is inherited under flat local homomorphisms. Let us begin with the following. Notice that the converse of the first assertion of Theorem \ref{3.5} is not true in general, unless $R/\fkm$ is infinite (Remark \ref{3.4'}).

\begin{thm}\label{3.5}
Let $(S, \fkn)$ be a Noetherian local ring and let $\varphi : R \to S$ be a flat local homomorphism such that $S/\fkm S$ is a regular local ring. Then $S$ is an almost Gorenstein local ring, if $R$ is an almost Gorenstein local ring. The converse also holds, when $R/\fkm$ is infinite.
\end{thm}

\begin{proof} Suppose that $R$ is an almost Gorenstein local ring and consider an exact sequence $0 \to R \to \K_R \to C \to 0$ of $R$-modules such that $\mu_R(C) = \rme_\fkm^0(C)$. If $C = (0)$, then $R$ is a Gorenstein local ring, so that $S$ is a Gorenstein local ring. Suppose $C \ne (0)$. Then $S \otimes_RC$ is an Ulrich $S$-module by Proposition \ref{2.2} (2), since $C$ is an Ulrich $R$-module. Besides, $\K_S \cong S \otimes_R\K_R$ as an $S$-module (\cite[Satz 6.14]{HK}), since $S/\fkm S$ is a Gorenstein local ring. Thus $S$ is almost Gorenstein, thanks to the exact sequence  of $S$-modules
$$
0 \to S \to \K_S \to S \otimes_RC \to 0.
$$

Suppose now that $R/\fkm$ is infinite and $S$ is an almost Gorenstein local ring.  Let $n =\dim S/\fkm S$. We have to  show that $R$ is an almost Gorenstein local ring. Assume the contrary and choose a counterexample $S$ so that $\dim S = n+ d$ is as small as possible. Then $S$ is not a Gorenstein ring, so that $\dim S = n+d > 0$. Choose an exact sequence
$$0 \to S \to \rmK_S \to D \to 0$$
of $S$-modules with $\mu_S(D) = \rme^0_\fkn(D)$. Suppose $n > 0$. If $d > 0$, then we take an element $g \in \fkn$ so that $g$ is superficial for $D$ with respect to $\fkn$ and $\overline{g}$ is a part of a regular system of parameters of $S/\fkm S$, where $\overline{g}$ denotes the image of $g$ in $S/\fkm S$.  Then $g$ is $S$-regular and the composite homomorphism $R \overset{\varphi}{\to} S \to S/gS$ is flat. Therefore the minimality of $n+d$ forces $R$ to be an almost Gorenstein local ring, because $S/gS$ is an almost Gorenstein local ring by Theorem \ref{3.9} (2). Thus $d=0$ and $\fkp = \fkm S$ is a minimal prime ideal of $S$. Hence the induced  flat local homomorphism $R \overset{\varphi}{\to} S \to S_\fkp$ shows that $R$ is a Gorenstein ring, because $S_\fkp$ is a Gorenstein ring (Lemma \ref{3.1} (1)). Consequently $n = 0$ and $\fkn = \fkm S$.

Suppose now that $d \ge 2$. Then because $\fkn = \fkm S$, we may choose an element $f \in \fkm$ so that $f$ is $R$-regular and $\varphi(f)$ is superficial for $D$ with respect to $\fkn$.  Then by Theorem \ref{3.9} (2) $S/fS$ is an almost Gorenstein local ring, while the homomorphism $R/fR {\to} S/fS$ is flat. Consequently, $R/fR$ is an almost Gorenstein local ring, so that  by Theorem \ref{3.9} (1) $R$ is an almost Gorenstein local ring.

Thus $d=1$ and $\fkn = \fkm S$, so that $R$ is an almost Gorenstein local ring by \cite[Proposition 3.3]{GMP}, which is the required contradiction.
\end{proof}

Let $\rmr (R)= \ell_R(\Ext_R^d(R/\fkm, R))$ denote the Cohen-Macaulay type of $R$.




\begin{cor}\label{3.8}
Let $R$ be an almost Gorenstein local ring and choose an exact sequence  $0 \to R \overset{\varphi}{\longrightarrow} \K_R \to C \to 0$
of $R$-modules so that $\mu_R(C) = \e_\fkm^0(C)$. If $\varphi (1) \in \fkm\K_R$, then $R$ is a regular local ring. Therefore,  $\mu_R(C) = \rmr (R) - 1$, if $R$ is not a regular local ring.
\end{cor}

\begin{proof}
Enlarging the residue class field of $R$, by Theorem  \ref{3.5} we may assume that $R/\fkm$ is infinite. Suppose $\varphi (1) \in \fkm \K_R$. Then $C \ne (0)$ and therefore $ d > 0$ (Lemma \ref{3.1} (3)). Assume $d = 1$. Then by Lemma \ref{3.1} (1) $\rmQ (R)$ is a Gorenstein ring. Therefore  by Remark \ref{3.2} we get an exact sequence 
$0 \to R \overset{\psi}{\longrightarrow} I \to C \to 0$
of $R$-modules with $\psi (1) \in \fkm I$, where $I ~(\subsetneq R)$ is an ideal of $R$ such that $I \cong \K_R$ as an $R$-module. Let $a= \psi (1)$. Then $\fkm I = (a)$, because $\fkm C = (0)$ and $a \in \fkm I$. Hence $R$ is a discrete valuation ring, because the maximal ideal $\fkm$ of $R$ is invertible.

Let $d > 1$ and assume that our assertion holds true for $d-1$.  Let $f \in \fkm$ be a non-zerodivisor of $R$ such that $f$ is superficial for $C$ with respect to $\fkm$. We set $\overline{R} = R/(f)$ and $\overline{C} = C/fC$. Then by Theorem \ref{3.9} (2) (and its proof) $\overline{R}$ is an almost Gorenstein local ring with the exact sequence
$0 \to \overline{R} \overset{\overline{\varphi}}{\longrightarrow} \K_{\overline{R}} \to \overline{C} \to 0$
of $\overline{R}$-modules, 
where $\overline{\varphi}= \overline{R}\otimes_R\varphi$ and $\rmK_{\overline{R}} = \rmK_R/f\rmK_R$. Therefore, because $\overline{\varphi} (1) \in \fkm \K_{\overline{R}}$, the hypothesis of induction on $d$ shows $\overline{R}$ is regular and hence so is $R$.

The second assertion follows from the fact that $\mu_R(C) = \mu_R(\K_R) -1 = \rmr (R) -1$ (\cite[Korollar 6.11]{HK}), because $\varphi (1) \not\in \fkm \K_R$.
\end{proof}

The following is an direct consequence of Theorem \ref{3.5}. See \cite[Satz 6.14]{HK} for the equality $\rmr (S) = \rmr (R)$.

\begin{cor}\label{3.6}
Suppose that $R$ is an almost Gorenstein local ring. Then  for every $n \ge 1$ the formal power series ring $S=R[[X_1, X_2, \ldots, X_n]]$ is also an almost Gorenstein local ring with $\rmr (S) = \rmr (R)$.
\end{cor}

\begin{prop}\label{11b}
Let $(S,\fkn)$ be a Noetherian local ring and let $\varphi : R \to S$ be a flat local homomorphism such that $S/\fkm S$ is a Gorenstein ring. Assume the following three conditions are satisfied. 
\begin{enumerate}
\item[$(1)$] The field $R/\fkm$ is infinite.
\item[$(2)$] $R$ and $S$ are  almost Gorenstein local rings.
\item[$(3)$] $S$ is not a  Gorenstein ring.
\end{enumerate}
If $\dim R = 1$, then $S/\fkm S$ is a regular local ring.
\end{prop}

\begin{proof}
When $\dim S/\fkm S > 0$, we pass to the flat local homomorphism $R \to S/gS$, choosing $g \in \fkn$ so that $g$ is $S/\fkm S$-regular and $S/gS$ is an almost Gorenstein local ring. Therefore we may assume that $\dim S/\fkm S = 0$. Choose an ideal $I \subsetneq R$ of $R$ so that $I \cong \rmK_R$ as an $R$-module. Therefore $IS \cong \rmK_S$ as an $S$-module, since $S/\fkm S$ is a Gorenstein local ring. Let $\e_1(I)$ (resp. $\e_1(IS))$ be the first Hilbert coefficient of $R$ (resp. $S$) with respect to $I$ (resp. $IS$). Then by \cite[Theorem 3.16]{GMP} and \cite[Satz 6.14]{HK} we have 
$$\rme_1(I) = \rmr(R) = \rmr (S) = \e_1(IS) = \ell_S(S/\fkm S){\cdot}\e_1(I),$$
because $R$ and $S$ are almost Gorenstein local rings and  both of them  are not Gorenstein rings. Thus $\ell_S(S/\fkm S) = 1$, so that  $S/\fkm S$ is a field.
\end{proof}

Unless $\dim R = 1$, Proposition \ref{11b} does not hold true in general, as  we show in the following.

\begin{ex}\label{badex1}
Let $T$ be an almost Gorenstein local ring with maximal ideal $\fkm_0$, $\dim T = 1$, and $\rmr (T) = 2$. Let $R = T[[X]]$ be the formal power series ring and let $R[Y]$ be the polynomial ring. We set $S =R[Y]/(Y^2 - X)$. Then the following assertions hold true.
\begin{enumerate}[\rm(1)]
\item $R$ and $S$ are two-dimensional almost Gorenstein local rings with $\rmr (R) = \rmr (S) = 2$.
\item $S$ is a finitely generated free $R$-module of rank two but $S/\fkm S$ is not a field, where $\fkm= \fkm_0 R + XR$ denotes the maximal ideal of $R$.
\end{enumerate}
\end{ex}

\begin{proof} We set $k = T/\fkm_0$.
Notice that $S$ is a local ring with maximal ideal $\fkm S + yS$, where $y$ denotes the image of $Y$ in $S$. The $R$-module $S$ is free of rank two and $S/\fkm S =(R/\fkm)[Y]/(Y^2)= k[Y]/(Y^2)$. The $T$-algebra $S$ is flat with $$S/\fkm_0 S = (k[[X]])[Y]/(Y^2 - X),$$ which is a discrete valuation ring. By Corollary  \ref{3.6} $R$ is an almost Gorenstein local ring with $\rmr (R) = 2$. We now consider the exact sequence $$(\sharp) \ \ \ 0 \to T \to \rmK_T \to T/\fkm_0 \to 0$$ of $T$-modules. Then since $\rmK_S \cong S \otimes_T\rmK_T$, tensoring exact sequence $(\sharp)$ by $S$, we get the exact sequence $$0 \to S \to \rmK_S \to S/\fkm_0 S \to 0$$ of $S$-modules. Therefore $S$ is an almost Gorenstein local ring by definition, since  $S/\fkm_0 S$ is a discrete valuation ring, while $\rmr (S) = 2$ by \cite[Satz 6.14]{HK}, since $S/\fkm S$ is a Gorenstein ring.
\end{proof}

Let us note a few basic examples of almost Gorenstein local rings.

\begin{ex}\label{3.11}
Let $U = k[[X_1, X_2, \ldots, X_n,Y_1, Y_2,, \ldots, Y_n]]~(n \ge 2)$ be the formal power series ring over a field $k$ and put $R = U/\rmI_2({\Bbb M})$, where $\rmI_2({\Bbb M})$ denotes the ideal of $U$ generated by $2 \times 2$ minors of the matrix ${\Bbb M} =\left(\begin{smallmatrix}
X_1&X_2& \cdots&X_n\\
Y_1&Y_2& \cdots&Y_n
\end{smallmatrix}\right)$.
Then $R$ is almost Gorenstein with $\dim R = n + 1$ and $\rmr (R) = n - 1$.
\end{ex}

\begin{proof}
It is well-known that $R$ is a Cohen-Macaulay normal local ring with $\dim R = n + 1$ and $\rmr (R) =n - 1$ (\cite{BV}). The sequence $\{X_i - Y_{i-1}\}_{1 \le i \le n}$~(here $Y_0 = Y_n$ for convention) forms a regular sequence in $R$ and 
$$R/(X_i - Y_{i-1} \mid 1 \le i \le n)R \cong  k[[X_1, X_2, \ldots, X_n]]/\rmI_2({\Bbb N}),$$ where ${\Bbb N} = \left(\begin{smallmatrix}
X_1&X_2& \cdots&X_{n-1}&X_n\\
X_2&X_3& \cdots&X_{n}&X_1
\end{smallmatrix}\right)$.
Let $S = k[[X_1, X_2, \ldots, X_n]]/\rmI_2({\Bbb N})$.
Then $S$ is a Cohen-Macaulay local ring of dimension one, such that $\fkn^2 = x_1\fkn$ and $\rmK_S \cong  (x_1, x_2, \ldots, x_{n-1})$, where $\fkn$ is the maximal ideal of $S$ and $x_i$ is the image of $X_i$ in $S$. Hence $S$ is an almost Gorenstein local ring, because $\fkn(x_1, x_2, \ldots, x_{n-1}) \subseteq (x_1)$. Thus $R$ is an almost Gorenstein local ring by Theorem \ref{3.9} (1).
\end{proof}

\begin{ex}\label{3.12}
Let $S=k[[X, Y, Z]]$ be a formal power series ring over a field $k$ and let ${\Bbb M} =\left(
\begin{smallmatrix}
f_{11}&f_{12}&f_{13}\\
f_{21}&f_{22}&f_{23}
\end{smallmatrix}
\right)$ be a matrix such that  $f_{ij} \in kX+kY+kZ$ for each $1 \le i \le 2$ and $1 \le j \le 3$. Assume that $\mathrm{ht}_S\rmI_2({\Bbb M}) = 2$ and set $R = S/\rmI_2({\Bbb M})$. Then $R$ is an almost Gorenstein local ring if and only if $R \not\cong S/(Y,Z)^2$.
\end{ex}

\begin{proof}
Since  $\rmQ (S/(Y,Z)^2)$ is not a Gorenstein ring, the \textit{only if} part follows from Lemma \ref{3.1} (1). Suppose that $R \not\cong S/(Y,Z)^2$. Then, thanks to \cite[Classification Table 6.5]{GS}, without loss of generality we may assume that our matrix ${\Bbb M}$ has the form ${\Bbb M} = \left(\begin{smallmatrix}
g_1&g_2&g_3\\
X&Y&Z
\end{smallmatrix}\right)$, where $g_i \in kX + kY + kZ $ for every $1 \le i \le 3$. Let $d_1 =Yg_3 - Zg_2$, $d_2 = Zg_1 - Xg_3$, and $d_3 = Xg_2 - Yg_1$. Then the $S$-module $R = S/(d_1, d_2, d_3)$ has a minimal free resolution 
$$0 \to S^{2} \overset{{}^t{\Bbb M}}{\longrightarrow} S^{3} \xrightarrow{[
d_1 d_2 d_3]} S \to R \to 0$$
and, taking the $S$-dual, we get the presentation $S^{3}  \overset{{\Bbb M}}{\longrightarrow} S^{2} \overset{\varepsilon}{\longrightarrow} \K_R \to 0$ of the canonical module $\K_R = \Ext_R^2(R,S)$. Hence $\K_R/R\xi \cong S^{2}/L \cong S/(X,Y,Z)$, where $\xi = \varepsilon(\binom{0}{1})$, and $L$ denotes the $S$-submodule of $S^{2}$ generated by $\binom{X}{g_1}$, $\binom{Y}{g_2}$, $\binom{Z}{g_3}$, and $\binom{0}{1}$.
We therefore have an exact sequence
$R \to \K_R \to R/\fkm \to 0$ of $R$-modules, where $\fkm$ is the maximal ideal of $R$. Hence $R$ is almost Gorenstein by Lemma \ref{3.1} (1).
\end{proof}

\begin{ex}\label{3.13}
Let $a, \ell \in \Bbb Z$ such that $a \ge 4$, $\ell \ge 2$ and let
$$R = k[[t^a, t^{a\ell - 1}, \{t^{a\ell + i}\}_{1 \le i \le a - 3}\}]] ~\subseteq k[[t]]
$$
be the semigroup ring of the numerical semigroup $H =\left<a, a\ell - 1, \{a\ell + i\}_{1 \le i \le a - 3}\right>$, where $k[[t]]$ denotes the formal power series ring over a field $k$. Then $R$ is an almost Gorenstein local ring with $\rmr (R) = a-2$. Therefore the formal power series rings $R[[X_1, X_2, \ldots, X_n]]$~($n \ge 1)$ are also almost Gorenstein.

\end{ex}

\begin{proof}
Let $I = (t^{2a\ell - a -1}, \{t^{3a\ell - 2a -i -1}\}_{1 \le i \le a-3})$. Then $I \cong \K_R$ and $\fkm I = \fkm t^{2a\ell - a - 1}$, where $\fkm$ denotes the maximal ideal of $R$. Hence $R$ is an almost Gorenstein local ring (see \cite[Example 2.13]{GhGHV} for details).
\end{proof}

\begin{rem}
The local rings $R_\fkp~(\fkp \in \Spec R \setminus \{\fkm\})$ of an almost Gorenstein local ring $(R, \fkm)$ are not necessarily Gorenstein rings (Example \ref{7.9}). Also, the local rings $R_\fkp$ of an almost Gorenstein local ring $R$ are not necessarily almost Gorenstein (Example \ref{p7.5}). 
\end{rem}

\section{Characterization in terms of existence of certain exact sequences}\label{chex}
In this section we investigate the almost Gorenstein property of local rings in terms of the existence of certain exact sequences.

Throughout this section, let $(R,\m,k)$ be a Cohen-Macaulay local ring of dimension $d$ and Cohen-Macaulay type $r$, admitting the canonical module $\rmK_R$. In what follows, all $R$-modules assumed to be finitely generated. For each sequence $\xx=x_1, x_2, \ldots,x_n$ of elements in $R$ and an $R$-module $M$, let $\KK_\bullet(\xx,M)$ be the Koszul complex of $M$ associated to  $\xx$. Hence
$$
\KK_\bullet(\xx,M)=\KK_\bullet(x_1,R)\otimes_R\cdots\otimes_R\KK_\bullet(x_n,R)\otimes_RM.
$$
For each $z\in M$ let $R\xrightarrow{z}M$ stand for the homomorphism $a\mapsto az$.
Denote by $\LL_R(z,M)$ the complex
$$
\LL_R(z,M)=(\cdots\to\underset{\begin{smallmatrix}\phantom{o}\\2\end{smallmatrix}}{0}\to\underset{\begin{smallmatrix}\phantom{o}\\1\end{smallmatrix}}{R}\xrightarrow{z}\underset{\begin{smallmatrix}\phantom{o}\\0\end{smallmatrix}}{M}\to\underset{\begin{smallmatrix}\phantom{o}\\-1\end{smallmatrix}}{0}\to\cdots).
$$
When there is no danger of confusion, we simply write $\LL(z,M)$ as $\LL_R(z,M)$. Let $\D(R)$ denote the derived category of $R$. Hence for two complexes $X,Y$ of $R$-modules,  one has $\H_i(X)\cong\H_i(Y)$ for all $i\in\Z$, if $X\cong Y$ in $\D(R)$.

Let us  begin with  the following.

\begin{lem}\label{6131107}
Let $\xx=x_1,x_2, \ldots,x_n$ be an $R$-sequence and let $y\in\rmK_R$.
Then 
$$
R/(\xx)\otimes_R^{\bf L}\LL_R(y,\rmK_R)\cong\LL_{R/(\xx)}(\overline{y},\rmK_R/\xx\rmK_R)
$$
in $\D(R)$, where $\overline y$ denotes the image of $y$ in $\rmK_R/\xx\rmK_R$.
\end{lem}

\begin{proof}
Since $\LL_R(y,\rmK_R)$ is the mapping cone of the map $R\xrightarrow{y}\rmK_R$, 
we get an exact triangle
$
R \xrightarrow{y} \rmK_R \to \LL_R(y,\rmK_R) \rightsquigarrow
$
in $\D(R)$, which gives rise to, applying the triangle functor $R/(\xx)\otimes_R^{\bf L}-$, an exact triangle
$$
R/(\xx)\otimes_R^{\bf L}R \xrightarrow{R/(\xx)\otimes_R^{\bf L}y} R/(\xx)\otimes_R^{\bf L}\rmK_R \to R/(\xx)\otimes_R^{\bf L}\LL_R(y,\rmK_R) \rightsquigarrow.
$$
Notice that $R/(\xx)\otimes_R^{\bf L}R\cong R/(\xx)$ and that $\Tor_i^R(R/(\xx),\rmK_R)\cong\H_i(\KK_\bullet(\xx,\rmK_R))=(0)$ for all $i>0$, since $\xx$ is also an $\rmK_R$-sequence; hence $R/(\xx)\otimes_R^{\bf L}\rmK_R\cong\rmK_R/\xx\rmK_R$. Observe that $R/(\xx)\otimes_R^{\bf L}\LL_R(y,\rmK_R)$ is isomorphic to the mapping cone of the map $R/(\xx) \xrightarrow{\overline y} \rmK_R/\xx\rmK_R$, which is nothing but $\LL_{R/(\xx)}(\overline{y},\rmK_R/\xx\rmK_R)$.
\end{proof}

We firstly give a characterization of Gorenstein local rings.

\begin{prop}\label{6131722}
The following conditions are equivalent.
\begin{enumerate}[\rm(1)]
\item
$R$ is a Gorenstein ring.
\item
There exist an $R$-sequence $\xx$ and an element $y\in\rmK_R$ such that $\KK_\bullet(\xx,R)\otimes_R\LL(y,\rmK_R)$ is an exact sequence.
\end{enumerate}
\end{prop}

\begin{proof}
(1) $\Rightarrow$ (2) Choose $y \in \rmK_R$ so that $\rmK_R =Ry$ and notice that $\LL(y,\rmK_R)\cong\KK_\bullet(1,R)\cong(0)$ in $\D(R)$. Therefore for each $R$-sequence $\xx$ we get
$$
\KK_\bullet(\xx,R)\otimes_R\LL(y,\rmK_R)\cong\KK_\bullet(\xx,R)\otimes_R^{\bf L}\LL(y,\rmK_R)\cong\KK_\bullet(\xx,R)\otimes_R^{\bf L}(0)\cong(0)
$$
in $\D(R)$, where the first isomorphism comes from the fact that $\KK_\bullet(\xx,R)$ is a complex of free $R$-modules. Thus the complex $\KK_\bullet(\xx,R)\otimes_R\LL(y,\rmK_R)$ is exact.

(2) $\Rightarrow$ (1)
By Lemma \ref{6131107}
$$
(0)\cong\KK_\bullet(\xx,R)\otimes_R\LL(y,\rmK_R)\cong R/(\xx)\otimes_R^{\bf L}\LL_R(y,\rmK_R)\cong\LL_{R/(\xx)}(\overline y,\rmK_R/\xx\rmK_R)
$$
in $\D(R)$. Hence the map $R/(\xx)\xrightarrow{\overline y}\rmK_R/\xx\rmK_R$ is an isomorphism, which shows $R/(\xx)\cong\K_{R/(\xx)}$. Therefore $R/(\xx)$ is a Gorenstein ring and hence so is $R$.
\end{proof}

The following theorem \ref{6131106} is the main result of this section, characterizing almost Gorenstein local rings.  For an $R$-module $M$ and a complex $C=(\cdots\to C_2\to C_1\to C_0\to0)$ of $R$-modules, we say that $C$ is acyclic over $M$, if $\H_0(C)\cong M$ and $\H_i(C)=(0)$ for all $i>0$.

\begin{thm}\label{6131106}
Assume that $d = \dim R \ge 1$ and the field $k=R/\fkm$ is infinite. Then the following conditions are equivalent.
\begin{enumerate}[\rm(1)]
\item
$R$ is an almost Gorenstein local ring.
\item
There exist an $R$-sequence $\xx=x_1,x_2, \ldots,x_{d-1}$ and an element $y\in\rmK_R$ such that $\KK_\bullet(\xx,R)\otimes_R\LL(y,\rmK_R)$ is an acyclic complex over $k^{r-1}$.
\item
There exist an $R$-sequence $\xx$ and an element $y\in\rmK_R$ such that $\KK_\bullet(\xx,R)\otimes_R\LL(y,\rmK_R)$ is an acyclic complex over an $R$-module $M$ such that $\fkm M = (0)$. 
\end{enumerate}
\end{thm}

\begin{proof}
(1) $\Rightarrow$ (2)
By Theorem \ref{3.9} (2) applied repeatedly, we get an $R$-sequence $\xx=x_1,x_2, \ldots,x_{d-1}$ such that $R/(\xx)$ is an almost Gorenstein local ring of dimension one. Choose an exact sequence
$$
0 \to R/(\xx) \xrightarrow{\varphi} \rmK_R/\xx\rmK_R \to k^{r-1} \to 0.
$$
Setting $\overline{y}=\varphi(1)$ with $y\in\rmK_R$, we see that $\LL(\overline y,\rmK_R/\xx\rmK_R)$ is quasi-isomorphic to $k^{r-1}$. By Lemma \ref{6131107} 
$$
\KK_\bullet(\xx,R)\otimes_R\LL(y,\rmK_R)\cong R/(\xx)\otimes_R^{\bf L}\LL(y,\rmK_R)\cong\LL(\overline y,\rmK_R/\xx\rmK_R)\cong k^{r-1}
$$
in $\D(R)$. Hence $\KK_\bullet(\xx,R)\otimes_R\LL(y,\rmK_R)$ is an acyclic complex over $k^{r-1}$.

(2) $\Rightarrow$ (3)
This is obvious.

(3) $\Rightarrow$ (1)
By Lemma \ref{6131107} $$
\LL(\overline y,\rmK_R/\xx\rmK_R)\cong R/(\xx)\otimes_R^{\bf L}\LL(y,\rmK_R)\cong \KK_\bullet(\xx,R)\otimes_R\LL(y,\rmK_R)\cong k^n
$$ for some $n\ge0$. Hence there exists an exact triangle
$
R/(\xx) \xrightarrow{\overline y} \rmK_R/\xx\rmK_R \to k^n \rightsquigarrow
$
in $\D(R)$, which gives the exact sequence
$
0 \to R/(\xx) \xrightarrow{\overline y} \rmK_R/\xx\rmK_R \to k^n \to 0
$
of homology. Thus $R/(\xx)$ is an almost Gorenstein local ring of dimension one, whence by Theorem \ref{3.9} (1) $R$ is an almost Gorenstein local ring.
\end{proof}

Applying Theorem \ref{6131106} to local rings of lower dimension, we readily get the following.

\begin{cor}\label{711824}
Assume that $k=R/\fkm$ is infinite.
\begin{enumerate}[\rm(1)]
\item
Let $d=1$. Then $R$ is an almost Gorenstein local ring if and only if there exist an element $y\in\rmK_R$ and an exact sequence
$$
0 \to R \xrightarrow{y} \rmK_R \to k^{r-1} \to 0.
$$
\item
Let $d=2$.
Then $R$ is an almost Gorenstein local ring if and only if there exist an $R$-regular element $x$, an element $y\in\rmK_R$, and an exact sequence
$$
0 \to R \xrightarrow{\binom{y}{-x}} \rmK_R\oplus R \xrightarrow{(x,y)} \rmK_R \to k^{r-1} \to 0.
$$
\item
Let $d=3$.
Then $R$ is an almost Gorenstein local ring if and only if there exist an $R$-sequence $x_1,x_2$, an element $y\in\rmK_R$, and an exact sequence
$$
0 \to R \xrightarrow{\left(\begin{smallmatrix}y\\x_2\\-x_1\end{smallmatrix}\right)} \rmK_R\oplus R^2 \xrightarrow{\left(\begin{smallmatrix}x_2&-y&0\\-x_1&0&-y\\0&x_1&x_2\end{smallmatrix}\right)} \rmK_R^2\oplus R \xrightarrow{(x_1,x_2,y)} \rmK_R \to k^{r-1} \to 0.
$$
\end{enumerate}
\end{cor}

For an $R$-module $M$ let $\pd_RM$ and $\Gdim_RM$ denote the projective dimension and the G-dimension of $M$, respectively (we refer the reader to \cite{C} for details of G-dimension).

\begin{cor}\label{6160013}
Assume that $R$ is an almost Gorenstein local ring of dimension $d\ge 1$. Then the following assertions hold true.
\begin{enumerate}[\rm(1)]
\item
The exact sequence
$$
0 \to R \to \rmK_R\oplus R^{d-1} \to \rmK_R^{d-1}\oplus R^{\binom{d-1}{2}} \to \cdots \to \rmK_R^{\binom{d-1}{2}}\oplus R^{d-1} \to \rmK_R^{d-1}\oplus R \to \rmK_R \to k^{r-1} \to 0
$$
arising from Theorem $\ref{6131106}~(2)$ is self-dual with respect to $\rmK_R$, that is, after dualizing this exact sequence by $\rmK_R$, one obtains the same exact sequence $($up to isomorphisms$)$.
\item
Suppose that  $R$ is not a Gorenstein ring. Then $R$ is G-regular in the sense of \cite{greg}, that is $\Gdim_RM=\pd_RM$ for every finitely generated $R$-module $M$.
\end{enumerate}
\end{cor}

\begin{proof}
(1) Let $X[n]$ denote, for a complex $X$ of $R$-modules and $n \in \Bbb Z$, the complex $X$ shifted by $n$ (to the left). Then with the same notation as in Theorem \ref{6131106} (2), $\Bbb K=\KK_\bullet(\xx,R)$ and $U=\LL(y,\rmK_R)$. Therefore $\Hom_R(\Bbb K,R)\cong \Bbb K[1-d]$ and $\Hom_R(U,\rmK_R)\cong U[-1]$, which show 
\begin{align*}
\Hom_R(\Bbb K\otimes_RU,\rmK_R)
&\cong\Hom_R(\Bbb K,\Hom_R(U,\rmK_R))
\cong\Hom_R(\Bbb K,U[-1])\\
&\cong\Hom_R(\Bbb K,R)\otimes_RU[-1]
\cong \Bbb K[1-d]\otimes_RU[-1]
\cong (\Bbb K\otimes_RU)[-d]
\end{align*}
(for the third isomorphism, remember that $\Bbb K$ is a bounded complex of free $R$-modules). Hence we get the assertion, because $\H_0(\Bbb K\otimes_RU)\cong k^{r-1}$ and $\H_i(\Bbb K\otimes_RU)=(0)$ for $i>0$ by Theorem \ref{6131106} (2).

(2) It suffices to show that every $R$-module $M$ of finite G-dimension is of finite projective dimension. Let $N$ be a high syzygy of $M$. Then since $N$ is totally reflexive and maximal Cohen-Macaulay, we have $$\Ext_R^i(N,R)=(0)=\Ext_R^i(N,\rmK_R)$$ for all $i>0$. Apply the functor $\Hom_R(N,-)$ to the exact sequence in assertion (1) and we get $$\Ext_R^i(N,k^{r-1})=(0)$$ for $i\gg0$. Since $r-1> 0$ (as $R$ is not Gorenstein), $N$ has finite projective dimension, and so does $M$.
\end{proof}

Let us consider the Poincar\'{e} and Bass series over almost Gorenstein local rings. First of all let us fix some terminology. Let $X$ (respectively, $Y$) be a homologically right (respectively, left) bounded complex of $R$-modules, possessing finitely generated homology modules.
The {\em Poincar\'{e} series} of $X$ and the {\em Bass series} of $Y$ are defined as the following formal Laurent series with coefficients among non-negative integers:
$$
\PP_X(t)=\sum_{n\in\Z}\dim_k\Tor_n^R(X,k){\cdot}t^n,\ \II^Y(t)=\sum_{n\in\Z}\dim_k\Ext_R^n(k,Y){\cdot}t^n.
$$
We then have the following,  in which the Poincar\'{e} and Bass series of $C=\operatorname{Coker} \varphi$ are described in terms of the Bass series of $R$.

\begin{thm}
Let $(R,\fkm,k)$ be an almost Gorenstein local ring of dimension $d\ge 1$ and assume that $R$ is not a Gorenstein ring. Consider an exact sequence 
$$(\sharp) \ \ \ 0\to R\overset{\varphi}{\to}\rmK_R\to C\to0$$ of $R$-modules such that  $C$ is an Ulrich $R$-module.  We then have the following.
\begin{enumerate}[\rm(1)]
\item $
\RHom_R(C,\rmK_R)\cong C[-1]
$
in  $\calD(R)$.
\item $
\PP_C(t)=t^{1-d}\II^R(t)-1$ and $\II^C(t)=\II^R(t)-t^{d-1}$.
\end{enumerate}
\end{thm}

\begin{proof}
(1) Since $C$ is a Cohen-Macaulay $R$-module of dimension $d-1$, $\Ext_R^i(C,\rmK_R)=(0)$ for all $i\ne1$ (\cite[Satz 6.1]{HK}). To see $C \cong \Ext_R^1(C,\rmK_R)$, take the $\rmK_R$-dual of exact sequence $(\sharp)$ and we get the following commutative diagram
$$
\begin{CD}
0 @>>> \Hom_R(\rmK_R,\rmK_R) @>{\Hom_R(\varphi,\rmK_R)}>> \Hom_R(R,\rmK_R) @>>> \Ext_R^1(C,\rmK_R) @>>> 0 \\
@. @A{\cong}AA @A{\cong}AA \\
0 @>>> R @>\varphi>> \rmK_R @>>> C @>>> 0,
\end{CD}
$$
where the vertical isomorphisms  are canonical ones. Hence $C \cong \Ext_R^1(C,\rmK_R)$, so that $\rmr_R(C)=\mu_R(C)=r-1$ by Corollary \ref{3.8}  and \cite[Satz 6.10]{HK}.

(2)   By \cite[(A.7.7)]{C}
$$
t\II^C(t)=\II^{C[-1]}(t)=\II^{\RHom(C,\rmK_R)}(t)=\PP_C(t)\II^\rmK_R(t),
$$
while $$
\II^C(t)=t^{d-1}\PP_C(t),
$$
as $\II^\rmK_R(t)=t^d$. Therefore, since $\rmr_R(C)=\mu_R(C)=r-1$, applying $\Hom_R(k,-)$ to exact sequence $(\sharp)$ and writing the long exact sequence, we get
$$
\Ext_R^i(k,C)\cong
\begin{cases}
(0) & (i\le d-2),\\
k^{r-1} & (i=d-1),\\
\Ext_R^{i+1}(k,R) & (i\ge d).
\end{cases}
$$
Hence $\II^C(t)=\II^R(t)-t^{d-1}$.
\end{proof}

\section{Characterization in terms of canonical ideals}\label{chcan}

Let $(R,\fkm)$ be a Cohen-Macaulay local ring of dimension $d > 0$, which possesses the canonical module $\K_R$. The main result of this section is the following characterization of almost Gorenstein local rings in terms of canonical ideals, which is a natural generalization of \cite[Theorem 3.11]{GMP}.

\begin{thm}\label{4.1}
Suppose that $\rmQ (R)$ is a Gorenstein ring and take an ideal $I ~(\ne R)$ such that $I \cong \K_R$ as an $R$-module. Consider the following two conditions$\mathrm{:}$
\begin{enumerate}[\rm(1)]
\item
$R$ is an almost Gorenstein local ring.
\item
$R$ contains a parameter ideal $Q = (f_1, f_2, \ldots, f_d)$ such that $f_1 \in I$ and $\fkm (I+Q) = \fkm Q$.
\end{enumerate}
Then the implication $(2) \Rightarrow (1)$ holds. If $R/\fkm$ is infinite, the implication $(1) \Rightarrow (2)$ is also true.
\end{thm}

\begin{proof}
$(2) \Rightarrow (1)$ Let $\q = (f_2, f_3, \ldots, f_d)$. Then $\q$ is a parameter ideal for the Cohen-Macaulay local ring $R/I$, because $I + Q = I + \q$. We set $\overline{R} = R/\q$, $\overline{\fkm} = \fkm \overline{R}$, and $\overline{I} = I \overline{R}$. Notice that $\overline{I} \cong I/\q I \cong \K_{\overline{R}}$, since $\q \cap I = \q I$. Let $\overline{f_1}$ be the image of $f_1$ in $\overline{R}$.  Then since $\overline{\fkm}{\cdot}\overline{I} = \overline{\fkm}{\cdot}\overline{f_1}$, by \cite[Theorem 3.11]{GMP} $\overline{R}$ is an almost Gorenstein local ring, so that  $R$ is an almost Gorenstein local ring  by Theorem \ref{3.9}.

$(1) \Rightarrow (2)$ Suppose that $R/\fkm$ is infinite. We may assume that $R$ is not a Gorenstein ring (because $I$ is a principal ideal, if $R$ is a Gorenstein ring). We consider the  exact sequence
\begin{equation}
0 \to R \overset{\varphi}{\longrightarrow} I \to C \to 0 \tag{$\sharp$}
\end{equation}
of $R$-modules such that $C$ is an Ulrich $R$-module. Let $f_1= \varphi (1) \in I$. Choose an $R$-regular sequence $f_2, f_3, \ldots, f_d \in \fkm$ so that (1) $f_1, f_2, \ldots, f_d$ is a system of parameters of $R$, (2) $f_2, f_3, \ldots, f_d$ is a system of parameters for the ring $R/I$, and (3) $\fkm C = (f_2, f_3, \ldots, f_d)C$ (this choice is, of course, possible; see Proposition \ref{2.2} (2)). Let $\q = (f_2, f_3, \ldots, f_d)$ and set $\overline{R} = R/\q$, $\overline{\fkm} = \fkm \overline{R}$, and $\overline{I} = I \overline{R}$. Then exact sequence $(\sharp)$ gives rise to  the exact sequence
$$0 \to \overline{R} \overset{\overline{\varphi}}{\longrightarrow} \overline{I} \to \overline{C} \to 0$$
of $\overline{R}$-modules where $\overline{C} = C/\q C$, because $\overline{I} \cong I/\q I$ and $f_2, f_3, \ldots, f_d$ form a $C$-regular sequence. Therefore, since $\overline{I} \cong \K_{\overline{R}}$ and $\overline{\fkm}\overline{C} = (0)$, $\overline{R}$ is an almost Gorenstein local ring. We furthermore have that $\overline{\fkm}{\cdot}\overline{I} = \overline{\fkm}{\cdot}\overline{f_1}$
 (here $\overline{f_1}$ denotes the image of $f_1$ in $\overline{R}$), because $\overline{R}$ is not a discrete valuation ring (see Corollary \ref{3.8}; remember that $\overline{f_1} = \overline{\varphi} (1))$.
Consequently, since $\fkm I \subseteq \fkm f_1 + \q$, we get $$\fkm I \subseteq (\fkm f_1 + \q) \cap I = \fkm f_1 + (\q \cap I) = \fkm f_1 + \q I \subseteq \fkm Q,$$ where $Q = (f_1, f_2, \ldots, f_d)$. Hence $\fkm(I + Q ) = \fkm Q$ as wanted.
\end{proof}

Let $R$ be an almost Gorenstein local ring of dimension $d \ge 2$. Let $I$ be an ideal of $R$ such that $I \cong \K_R$ as an $R$-module. Suppose that $R$ is not a Gorenstein ring but contains a parameter ideal $Q = (f_1, f_2, \ldots, f_d)$  such that $f_1 \in I$ and $\fkm(I+Q) = \fkm Q$. Let $\q = (f_2, f_3, \ldots, f_d)$. We set $\overline{R} = R/\q$, $\overline{\fkm} = \fkm \overline{R}$ and $\overline{I} = I \overline{R}$. Then $\overline{R}$ is an almost Gorenstein local ring with $\overline{I}=\rmK_{\overline{R}}$ and $(\overline{f})$ is a reduction of $\overline{I}$ with $\overline{\fkm}\overline{I} = \overline{\fkm} \overline{f}$, where $\overline{f}$ denotes the image of $f$ in $\overline{R}$ (see Proof of Theorem \ref{4.1}).

We explore what kind of properties the ideal $J = I + Q$ enjoys. To do this, we fix the following notation, which we maintain throughout this section.

\begin{notation}
Let ${\mathcal T} = R[Qt] \subseteq {\mathcal R} = R[It] \subseteq R[t]$ where $t$ is an indeterminate and set $\textstyle\gr_J(R) = {\mathcal R}/J{\mathcal R} ~(= \bigoplus_{n \ge 0}J^n/J^{n+1}$). We  set $${\mathcal S} = {\mathcal S}_Q(J) = J{\mathcal R}/J{\mathcal T}$$
(the Sally module of $J$ with respect to $Q$; \cite{V}). Let $${\mathcal B} = {\mathcal T}/\fkm {\mathcal T}$$
($= (R/\fkm)[T_1, T_2, \ldots, T_d],$ the polynomial ring) and $$\mathrm{red}_Q(J) = \operatorname{min} \{n \ge 0 \mid J^{n+1} = QJ^n\}.$$ We denote by $\{\rme_i(J)\}_{0 \le i \le d}$ the Hilbert coefficients of $R$ with respect to $J$. 
\end{notation}

Let us begin with the following. We set $f = f_1$.

\begin{cor}\label{4.2} The following assertions hold true.
\begin{enumerate}[\rm(1)]
\item
$\mathrm{red}_Q(J) = 2$.
\item
${\mathcal S}_Q(J) \cong {\mathcal B}(-1)$ as a graded ${\mathcal T}$-module. 
\item[$(3)$] $\ell_R(R/J^{n+1}) = \ell_R(R/Q){\cdot}\binom{n+d}{d} - \rmr (R){\cdot}\binom{n+d-1}{d-1} + \binom{n+d - 2}{d-2}$ for all $n \ge 0$. Hence $\rme_1(J) = \rmr (R)$, $\rme_2(J) = 1$, and $\rme_i(J) = 0$ for $3 \le i \le d$.
\item[$(4)$] Let $G = \gr_J(R)$. Then the elements $f_2t, f_3t, \ldots, f_dt ~(\in {\mathcal T}_1)$ form a regular sequence in $G$ but $ft$ is a zero-divisor in $G$. Hence $\depth G = d - 1$ and the graded local cohomology module $\rmH_{\fkM}^d(G)$ of $G$ is not finitely generated, where $\fkM = \fkm G + G_+$.
\end{enumerate}
\end{cor}

\begin{proof}
Let $K = Q:_R\fkm$. Then $Q \subseteq J \subseteq K$. Notice that $\ell_R(J/Q) = \mu_R(J/Q) = \mu_R(\overline{I}/(\overline{f})) = \rmr (\overline{R}) -1 = \rmr (R) -1$, because $\overline{f} \not\in \overline{\fkm}\overline{I}$ and $\overline{I} = \K_{\overline{R}}$. Therefore $\ell_R(K/J) = 1$ since $\ell_R(J/Q) = \rmr (R)$, so that $K = J + (x)$ for some $x \in K$, while $K^2 = QK$ (\cite{CP}), as $R$ is not a regular local ring. Consequently, $J^3 = QJ^2$ by \cite[Proposition 2.6]{GNO}. Thus $\mathrm{red}_Q(J) = 2$, since $\overline{I}^2 \ne \overline{f}{\cdot}\overline{I}$ (\cite[Theorem 3.7]{GMP}).

Let us show $\ell_R(J^2/QJ) = 1$. We have $\ell_R(\overline{I}^2/\overline{f}{\cdot}\overline{I}) = 1$ by \cite[Theorem 3.16]{GMP}. Choose $g \in I^2$ so that $I^2 \subseteq fI + (g) + \q$. Then $$I^2 = (fI + (g) + \q) \cap I^2 \subseteq fI + (g) + \q I,$$ since $\q \cap I = \q I$. Hence $J^2 = QJ + (g)$, because
$J^2 = QJ + I^2$. Consequently $\ell_R(J^2/QJ) = 1$, since $\fkm J^2 = \fkm Q^2$ (remember that $\fkm J = \fkm Q$). Therefore, thanks to \cite{S2, V}, we have ${\mathcal S}_Q(J) \cong {\mathcal B}(-1)$ as a graded ${\mathcal T}$-module, $\e_1(J) = \e_0(J) - \ell_R(R/J) + 1$,  and $$\textstyle\ell_R(R/J^{n+1}) = \e_0(J){\cdot}\binom{n+d}{d} - \e_1(J){\cdot}\binom{n+d-1}{d-1} + \binom{n+d - 2}{d-2}$$ for all $n \ge 0$. Hence
$$
\e_1(J)
=\e_0(J) - \ell_R(R/J) + 1
=\ell_R(R/Q) - \ell_R(R/J) + 1
=\ell_R(J/Q) +1
=\rmr (R).
$$
Thus assertions (1), (2), and (3) follow.

To see assertion (4), we claim the following, which shows the sequence $f_2t, f_3t, \ldots, f_dt$ is $\gr_J(R)$-regular.

\begin{claim*} $\q \cap J^n = \q J^{n-1}$ for all $n \in \Bbb Z$.
\end{claim*}

\begin{proof}[Proof of Claim]
As $J^2 = QJ + (g) = fJ + (g) + \q J,$ we have $$\q \cap J^2 = \q J + \q \cap [fJ + (g)] \subseteq \q J + (\q \cap I) = \q J.$$
Suppose that $n \ge 3$ and that our assertion holds true for $n-1$. Then 
$$
\q \cap J^n 
= \q \cap Q J^{n-1}
= \q J^{n-1} + (\q \cap f J^{n-1})\\
= \q J^{n-1} + f{\cdot}(\q \cap J^{n-1})
= \q J^{n-1} + f{\cdot}\q J^{n-2}
= \q J^{n-1}.
$$
Hence $\q \cap J^n = \q J^{n-1}$ for all $n \in \Bbb Z$.
\end{proof}

To show that $ft$ is a zero-divisor in $\gr_J(R)$, remember that $g \not\in QJ$, because $J^2 \ne QJ$. Since $J^2 \subseteq Q$, we may write $g = fy + h$ with $y \in R$ and $h \in \q$. Then because $fy = g -h \in I^2 + \q$, we see $$fy \in (I^2 + \q) \cap I = I^2 + \q I \subseteq J^2,$$ while $y \not\in J$. In fact, if $y \in J$, then $$h = g -fy \in \q \cap J^2  = \q J \subseteq QJ,$$ so that $g = fy + h \in QJ$, which is impossible. Thus $ft$ is a zero-divisor in $\gr_J(R)$.
\end{proof}

Let $\rho : \gr_I(R) \overset{\varphi}{\longrightarrow} \gr_J(R) \overset{\psi}{\longrightarrow} \gr_{\overline{I}}(\overline{R})$ be the composite of canonical homomorphisms of associated graded rings and set $\mathcal A = \operatorname{Im} \varphi$. We then have $\gr_J(R) = \mathcal A[\xi_2, \xi_3, \ldots, \xi_d]$, where $\xi_i = \overline{f_it}$ denotes the image of $f_it$ in $\gr_J(R)$. We are now interested in the question of when $\{\xi_i\}_{2 \le i \le d}$ are algebraically independent over $\mathcal A$. Our goal is Theorem \ref{4.6} below.

We begin with the following, which readily follows from the fact that $\operatorname{Ker} \rho = \bigoplus_{n \ge 0}[I^{n+1} + (\q \cap I^n)]/I^{n+1}$.

\begin{lem}\label{4.3}
$\operatorname{Ker} \varphi = \operatorname{Ker} \rho$ if and only if $\q \cap I^n \subseteq J^{n+1}$ for all $n \ge 2$.
\end{lem}

\begin{lem}\label{4.4}
Let $n \ge 2$.
Then $\q \cap I^n = \q I^n$ if and only if $R/I^n$ is a Cohen-Macaulay ring.
\end{lem}

\begin{proof}
If $R/I^n$ is a Cohen-Macaulay ring, then $\q \cap I^n = \q I^n$, because $f_2, f_3, \ldots, f_d$ form a regular sequence in $R/I^n$. Conversely, suppose that $\q \cap I^n = \q I^n$. Then the descending induction on $i$ readily yields that $$(f_2, f_2, f_3, \ldots, f_i) \cap I^n = (f_2, f_3, \ldots, f_i)I^n$$ for all $2 \le i \le d$, from which it follows that the sequence $f_2, f_3, \ldots, f_d$ is $R/I^n$-regular. 
\end{proof}

\begin{prop}\label{4.5}
The following assertions hold true.
\begin{enumerate}[\rm(1)]
\item
If $R/I^3$ is a Cohen-Macaulay ring, then $I^3 = fI^2$ and therefore the ideal $I$ has analytic spread one  and $\mathrm{red}_{(f)}(I) = 2$.
\item
If $R/I^2$ is a Cohen-Macaulay ring and $I^3 = fI^2$, then $R/I^n$ is a Cohen-Macaulay ring for all $n \ge 1$
\end{enumerate}
\end{prop}

\begin{proof}
(1) We have $\q \cap I^3 = \q I^3$ by Lemma \ref{4.4}, while $\overline{I}^3 = \overline{f}{\cdot}\overline{I}^2$ by Corollary \ref{4.2} (1). Therefore $I^3 \subseteq (fI^2 + \q) \cap I^3 = fI^2 + \q I^3$, so that $I^3 = fI^2$ by Nakayama's lemma. Hence $I$ is of analytic spread one and $\mathrm{red}_{(f)}(I) = 2$, because $\overline{I}^2 \ne \overline{f}{\cdot}\overline{I}$ (\cite[Theorem 3.7]{GMP}). 

(2) We show that $\q \cap I^n = \q I^n$ for all $n \in \Bbb Z$. By Lemma \ref{4.4} we may assume that $n \ge 3$ and that our assertion holds true for $n -1$. Then
$$
\q \cap I^n = \q \cap fI^{n-1}= f(\q \cap I^{n-1})= f{\cdot}\q I^{n-1}\subseteq \q I^n.
$$
Hence $\q \cap I^n = \q I^n$ for all $n \in \Bbb Z$, whence $R/I^n$ is a Cohen-Macaulay ring by Lemma \ref{4.4}.
\end{proof}

We are now ready to prove the following.

\begin{thm}\label{4.6}
Suppose that $R/I^2$ is a Cohen-Macaulay ring and $I^3 = fI^2$. Then $\mathcal A$ is a Buchsbaum ring and $\xi_2, \xi_3, \ldots, \xi_d$ are algebraically independent over $\mathcal A$, whence $\gr_J(R) =\mathcal A[\xi_2, \xi_3, \ldots, \xi_d]$ is the polynomial ring.
\end{thm}

\begin{proof}
We have $\operatorname{Ker} \varphi = \operatorname{Ker} \rho$ by Lemma \ref{4.3}, \ref{4.4}, and Proposition \ref{4.5}, which shows that the composite homomorphism $\mathcal A \overset{\iota}{\hookrightarrow} \gr_J(R) \overset{\psi}{\longrightarrow} \gr_{\overline{I}}(\overline{R})$ is an isomorphism, where $\iota : \mathcal A \to \gr_J(R)$ denotes the embedding.
Hence $\mathcal A$ is a Buchsbaum ring (\cite[Theorem 3.16]{GMP}). Let $k = \mathcal A_0$ and let $\mathcal C = k[X_2, X_3, \ldots, X_d]$ denote the polynomial ring. We regard $\mathcal C$ to be a $\Bbb Z$-graded ring so that $\mathcal C_0 = k$ and $\deg X_i = 1$. Let $B = \mathcal A \otimes_k\mathcal C$. Then $B$ is a $\Bbb Z$-graded ring whose grading is given by $B_n = \sum_{i + j = n}\mathcal A_i \otimes_k\mathcal C_j$ for all $n \in \Bbb Z$. We put $Y_i = 1 \otimes X_i$ and consider  the homomorphism $\Psi : B=\mathcal A[Y_2, Y_3, \ldots, Y_d] \to \gr_J(R)$ of $\mathcal A$-algebras defined by $\Psi (Y_i) = \xi_i$ for all $2 \le i \le d$. Let ${\mathcal K} = \operatorname{Ker} \Psi$. We then have the exact sequence
$ 
0 \to {\mathcal K} \to B \to \gr_J(R) \to 0,
$
which gives rise to the exact sequence
$$
0 \to {\mathcal K}/(Y_2, Y_3, \ldots, Y_d){\mathcal K} \to B/(Y_2, Y_3, \ldots, Y_d) \to \gr_J(R)/(\xi_2, \xi_3, \ldots, \xi_d) \to 0,
$$
since the sequence $\xi_2, \xi_3, \ldots, \xi_d$ is $\gr_J(R)$-regular (Corollary \ref{4.2} (4)). Because $$B/(Y_2, Y_3, \ldots, Y_d) = \mathcal A \cong \gr_{\overline{I}}(\overline{R}) = \gr_J{R}/(\xi_2, \xi_3, \ldots, \xi_d),$$
 we have ${\mathcal K}/(Y_2, Y_3, \ldots, Y_d){\mathcal K}
=(0)$ and hence ${\mathcal K} = (0)$ by graded Nakayama's lemma. Thus  $\Psi : B \to \gr_J(R)$ is an isomorphism of $\mathcal A$-algebras.
\end{proof}

The ring $R/I^n$ is not necessarily a Cohen-Macaulay ring. Let us explore one example.

\begin{ex}
Let $S = k[[s,t]]$ be the formal power series ring over a field $k$ and set $R = k[[s^3, s^2t, st^2, t^3]]$ in $S$. Then $R$ is a Cohen-Macaulay local ring of dimension $2$. Setting $x = s^3$, $y = s^2t$, $z = st^2,$ and $w = t^3$, we have $I = (y,z) \cong \K_R$ as an $R$-module and $\fkm{\cdot}(I + Q) = \fkm{\cdot}Q$, where $Q =(y, x-w)$. Hence $R$ is an almost Gorenstein local ring with $\rmr (R) = 2$. Let $\q = (x-w)$. Then $\q \cap I^2 = \q I^2$, because 
$$
(x-w) \cap I^2 \subseteq (x-w){\cdot}(s^2t^2S \cap R) \subseteq (x-w){\cdot}I^2.
$$
However, $\q \cap I^3 \ne \q I^3$. In fact, if $\q \cap I^3 = \q I^3$, by Proposition \ref{4.5} (1) $I$ is of analytic spread one, which is however impossible, because $(R/\fkm) \otimes_R\gr_I(R) \cong k[y,z]$. Hence $R/I^2$ is a Cohen-Macaulay ring but $R/I^3$ is not a Cohen-Macaulay ring (Lemma \ref{4.4}). We have $\e_1(J) = \rmr (R) = 2$ and $\depth~\gr_J(R) = 1$, where $J = I + Q = (x-w, y, z)$.
\end{ex}

\begin{ques}
Let $T$ be an almost Gorenstein but non-Gorenstein local ring of dimension $1$ and let $K$ be an ideal of $T$ with $K \cong \K_{T}$ as a $T$-module. Let $R = T[[X_2, X_3, \ldots, X_d]]~(d \ge 2)$ be a formal power series ring and set $I = KR$. Then $I \cong \K_R$ as an $R$-module and $R/I^n$ is Cohen-Macaulay for all $n \ge 1$. We suspect that this is the unique  case for $\gr_J(R)$ to be the polynomial ring over $\mathcal A$.
\end{ques}

\section{Almost Gorenstein local rings obtained by idealization}\label{idealiz}

Throughout this section let $(R,\fkm)$ be a Cohen-Macaulay local ring, which possesses the canonical module $\K_R$. For each $R$-module $M$ let $M^\vee = \Hom_R(M,\K_R)$. We study the question of when the idealization $R \ltimes M^\vee$ is an almost Gorenstein local ring.

Let us begin with the following, which is based on \cite[Proposition 6.1]{GMP} and gives an extension of the result to higher dimensional local rings.

\begin{thm}\label{5.1} Let $(R,\fkm)$ be a Cohen-Macaulay local ring of dimension $d \ge 1$, which possesses the canonical module $\K_R$.
Let $I ~(\ne R)$ be an ideal of $R$ and assume that $R/I$ is a Cohen-Macaulay ring of dimension $d-1$. We consider the following two conditions$\mathrm{:}$
\begin{enumerate}[\rm(1)]
\item
$A = R \ltimes I^\vee$ is an almost Gorenstein local ring.
\item
$R$ contains a parameter ideal $Q = (f_1, f_2, \ldots, f_d)$ such that $f_1 \in I$, $\fkm (I+Q) = \fkm Q$, and $(I+Q)^2 = Q(I+Q)$.
\end{enumerate}
Then one has the implication $(2) \Rightarrow (1)$. If $R/\fkm$ is infinite, the reverse implication $(1) \Rightarrow (2)$ is also true. 
\end{thm}

\begin{proof}
$(2) \Rightarrow (1)$
Let $\q = (f_2, f_3, \ldots, f_d)$ and set $\overline{R} = R/\q$, $\overline{\fkm} = \fkm \overline{R}$, and $\overline{I} = I \overline{R}$. Then $\overline{I} \cong I/\q I$, since $f_2, f_3, \ldots, f_d$ is a regular sequence in $R/I$, while $\Hom_{\overline{R}}(\overline{I},\K_{\overline{R}}) \cong I^\vee/\q I^\vee$, because $f_2, f_3, \ldots, f_d$ is an $R$-sequence  and $I$ is a maximal Cohen-Macaulay $R$-module  (\cite[Lemma 6.5]{HK}). Therefore since $\overline{I}^2 = \overline{f_1}{\cdot}\overline{I}$ and $\overline{\fkm}{\cdot}\overline{I} = \overline{\fkm}{\cdot}\overline{f_1}$ (here $\overline{f_1}$ denotes the image of $f_1$ in $\overline{R}$), by \cite[Proposition 6.1]{GMP} the idealization $A/\q A = \overline{R} \ltimes \Hom_{\overline{R}}(\overline{I}, \K_{\overline{R}})$ is an almost Gorenstein local ring. Hence $A = R \ltimes I^\vee$ is an almost Gorenstein local ring by Theorem \ref{3.9}, because $f_2, f_3, \ldots, f_d$ form a regular sequence in $A$.

$(1) \Rightarrow (2)$
Suppose that $R/\fkm$ is infinite and that $A = R \ltimes I^\vee$ is an almost Gorenstein local ring. Choose an exact sequence
\begin{equation}
0 \to A \to \K_A \to C \to 0 \tag{$\sharp$}
\end{equation}
of $A$-modules such that $\mu_A(C) = \e_\fkn^0(C)$, where $\fkn = \fkm \times I^\vee$ is the maximal ideal of $A$. If $C = (0)$, then $A$ is a Gorenstein local ring. Hence $I^\vee \cong \K_R$ (\cite{R}), whence  $I \cong \K_R^\vee = R$ and assertion (2) is certainly true. Assume that $C \ne (0)$. Then  $C$ is an Ulrich $A$-module of dimension $d-1$. We put $\fka = (0):_AC$ and consider $R$ to be a subring of $A$ via the homomorphism $R \to A, ~r \mapsto (r,0)$. Then, since $B = A/\fka$ is a module-finite extension of $S = R/[\fka \cap R]$, $\dim S = \dim B = \dim_AC = d-1$. We set $\fkn_B = \fkn B$ and $\fkm_S = \fkm S$. Then $\fkm_SB$ is a reduction of $\fkn_B$, because $[(0) \times I^\vee]^2 = (0)$ in $A$. We choose a subsystem $f_2, f_3, \ldots, f_d$ of parameters of $R$ so that $f_2, f_3, \ldots f_d$ is a system of parameters for $R/I$ and $(f_2, f_3, \ldots, f_d)B$ is a reduction of $\fkn_B$. Then $\fkn C = (f_2, f_3, \ldots, f_d)C$ by Proposition \ref{2.2} (2). Consequently, since $f_2, f_3, \ldots, f_d$ is a $C$-regular sequence, from exact sequence $(\sharp)$ above we get the exact sequence
$$
0 \to A/\q A \to \K_A/\q \K_A \to C/\q C \to 0
$$
of $A/\q A$-modules, where $\q = (f_2, f_3, \ldots, f_d)$. Hence $A/\q A$ is an almost Gorenstein local ring of dimension one, because $\K_A/\q \K_A \cong \K_{A/\q A}$ and $\fkn (C/\q C) = (0)$. Let $\overline{R} = R/\q$, $\overline{\fkm} = \fkm \overline{R}$, and $\overline{I} = I \overline{R}$. Then since $\q \cap I = \q I$, we get $\overline{I} = [I + \q]/\q \cong I/\q I$ and therefore the ring $$\overline{R} \ltimes \Hom_{\overline{R}}(\overline{I}, \K_{\overline{R}}) = (R/\q) \ltimes (I^\vee/\q I^\vee)  = A/\q A $$
is an almost Gorenstein local ring. Consequently, by \cite[Proposition 6.1]{GMP} we may choose $f_1 \in I$ so that $\overline{\fkm}{\cdot}\overline{I} = \overline{\fkm}{\cdot}\overline{f_1}$ and $\overline{I}^2 = \overline{f_1}{\cdot}\overline{I}$, where $\overline{f_1}$ denotes the image of $f_1$ in $\overline{R}$. Let  $Q = (f_1, f_2, \ldots, f_d)$. We will show that $\fkm (I+Q) = \fkm Q$ and $(I+Q)^2 =  Q(I+Q)$.
Firstly, since $\overline{\fkm}{\cdot}\overline{I} = \overline{\fkm}{\cdot}\overline{f_1}$, we get $\fkm I \subseteq (\fkm f_1 + \q) \cap I = \fkm f_1 + (\q \cap I)$.
Hence $\fkm I \subseteq \fkm Q$, because $\q \cap I = \q I$, so that $\fkm (I+Q) = \fkm Q$. Since $\overline{I}^2 = \overline{f_1}{\cdot}\overline{I}$, we similarly have 
$$I^2 \subseteq (f_1I + \q) \cap I^2 \subseteq f_1I + \q I = QI,
$$ whence $(I+Q)^2 = Q(I+Q)$. Notice that $Q$ is a parameter ideal of $R$, because $\sqrt{Q} = \sqrt{I+Q} = \fkm$, which proves Theorem \ref{5.1}.
\end{proof}

Let us consider the case where $R$ is a Gorenstein ring.
The following result extends \cite[Corollary 6.4]{GMP} to local rings of higher-dimension.

\begin{cor}\label{5.2}
Suppose that $(R,\fkm)$ is a Gorenstein local ring of dimension $d \ge 1$. Let $M$ be a Cohen-Macaulay faithful $R$-module and consider the following two conditions$\mathrm{:}$
\begin{enumerate}[\rm(1)]
\item
$A = R \ltimes M$ is an almost Gorenstein local ring.
\item
$M \cong R$ or $M \cong \fkp$ as an $R$-module for some $\fkp \in \Spec R$ such that $R/\fkp$ is a regular local ring of dimension $d-1$. 
\end{enumerate}
Then the implication $(2) \Rightarrow (1)$ holds. If $R/\fkm$ is infinite, the reverse implication $(1) \Rightarrow (2)$ is also true. 
\end{cor}

\begin{proof}
$(2) \Rightarrow (1)$
We may assume $M \cong \fkp$, where $\fkp \in \Spec R$ such that $R/\fkp$ is a regular local ring of dimension $d-1$. We choose a subsystem $f_2, f_3, \ldots, f_d$ of parameters of $R$ so that $\fkm = \fkp +(f_2, f_3, \ldots, f_d)$ and set $\q = (f_2, f_3, \ldots, f_d)$. Then 
$$
\fkp/\q \fkp \cong [\fkp + \q]/\q = \fkm/\q
$$ and hence $A/\q A \cong R/\q \ltimes \fkm/\q$.
Therefore $A/\q A$ is an almost Gorenstein local ring by \cite[Corollary 6.4]{GMP} and hence $A$ is an almost Gorenstein local ring by Theorem \ref{3.9}.

$(1) \Rightarrow (2)$
Suppose that $R/\fkm$ is infinite and let $\fkp \in \Ass R$. Then $A_{\fkp \times M} \cong R_\fkp \ltimes M_\fkp$ and $M_\fkp \ne (0)$. Because  $\rmQ (A)$ is a Gorenstein ring, we get $M_\fkp \cong \K_{R_\fkp} \cong R_\fkp$ (\cite{R}). Hence $\rmQ (R) \otimes_RM \cong \rmQ (R)$ and therefore we have an exact sequence\begin{equation}
0 \to R \to M \to X \to 0  \tag{$\sharp_1$}
\end{equation}
of $R$-modules such that $\rmQ (R) \otimes_RX = (0)$. Notice that  $X$ is a Cohen-Macaulay $R$-module of dimension $d-1$, because $X \ne (0)$ and $\depth_RM=d$. Take the $\K_R$-dual (in fact, $\K_R = R$) of exact sequence $(\sharp_1)$ and   we get the exact sequence
$$
0 \to M^\vee \overset{\varphi}{\longrightarrow} R \to \Ext_R^1(M,\K_R) \to 0
$$
of $R$-modules. Let $I = \varphi (M^\vee)$. Then $M \cong I^\vee$ and  $R/I$ is a Cohen-Macaulay local ring of dimension $d-1$. Consequently, because $A = R \ltimes I^\vee$ is an almost Gorenstein local ring, by Theorem \ref{5.1} $R$ contains a parameter ideal $Q=(f_1, f_2, \ldots, f_d)$ such that $f_1 \in I$, $\fkm (I + Q) = \fkm Q$, and $(I+Q)^2 = Q(I+Q)$. We set $\q = (f_2, f_3, \ldots, f_d)$. Then, since $Q \subseteq I + Q \subseteq Q :_R\fkm$ and $R$ is a Gorenstein local ring, we have either $Q = I + Q$ or $I+Q = Q :_R\fkm$.

If $Q = I+Q$, then 
$$I \subseteq Q \cap I = [(f_1) + \q]\cap I = (f_1) + \q I,$$ since $\q \cap I = \q I$, so that $I = (f_1) \cong R$. Therefore $M \cong I^\vee \cong R$, which is impossible. Hence $I + Q = Q :_R \fkm$ and $I \not\subseteq Q$. Choose $x \in I \setminus Q$. We then have $I + Q = Q + (x)$ and therefore 
$I = [(Q + (x)] \cap I = (f_1, x) + \q I$, so that $I =(f_1, x)$. Notice that $\mu_R(I) = 2$, because $I \not\cong R$. Let $\fkp = (f_1):_Rx$. Then $I/(f_1) \cong R/\fkp$ and hence $\dim R/\fkp < d$. On the other hand, thanks to the depth lemma applied to the exact sequence
\begin{equation}
0 \to R \overset{{\varphi}}{\longrightarrow} I  \to R/\fkp \to 0  \tag{$\sharp_2$}
\end{equation}
of $R$-modules where $\varphi (1) = f_1$, we get $\operatorname{depth} R/\fkp \ge d-1$. Hence $\dim R/\fkp = d-1$. Set  $\overline{R} = R/\q$, $\overline{\fkm} = \fkm \overline{R}$, and $\overline{I} = I \overline{R}$. Then $\overline{\fkm}{\cdot}\overline{I} \subseteq (\overline{f_1}),$ where $\overline{f_1}$ denotes the image of $f_1$ in $\overline{R}$. Since $\overline{I} \cong \overline{R}\otimes_RI$, we see $$\overline{R}\otimes_RR/\fkp \cong \overline{R}\otimes_R[I/(f_1)] \cong \overline{I}/(\overline{f_1})$$ and therefore $\fkm = \fkp + \q$, because $\overline{\fkm}{\cdot}(\overline{I}/(\overline{f_1})) = (0)$. Thus $R/\fkp$ is a regular local ring. Now we take the $\K_R$-dual of exact sequence $(\sharp_2)$ and get the exact sequence
$$0 \to I^\vee \to R \to \Ext_R^1(R/\fkp, \K_R) \to 0$$
of $R$-modules. Because $\Ext_R^1(R/\fkp,\K_R) \cong R/\fkp$, we then have $I^\vee \cong \fkp$. Hence $M \cong I^\vee \cong \fkp$ as $R$-modules, which proves the implication $(1) \Rightarrow (2)$.
\end{proof}

When $R$ contains a prime ideal $\fkp$ such that $R/\fkp$ is a regular local ring of dimension $d-1$, we have the following characterization for $A = R \ltimes \fkp$ to be an almost Gorenstein local ring, which is an extension of \cite[Theorem 6.5]{GMP}.

\begin{thm}\label{5.4} Let $(R,\fkm)$ be a Cohen-Macaulay local ring of dimension $d \ge 1$, which possesses the canonical module $\K_R$.
Suppose that $R/\fkm$ is infinite. Let $\fkp \in \Spec R$ and assume that $R/\fkp$ is a regular local ring of dimension $d-1$. Then the following conditions are equivalent.
\begin{enumerate}[\rm(1)]
\item
$A = R \ltimes \fkp$ is an almost Gorenstein local ring.
\item
$R$ is an almost Gorenstein local ring. 
\end{enumerate}
\end{thm}

\begin{proof}
By \cite[Theorem 6.5]{GMP} we may assume that $d > 1$ and that our assertion holds true for $d-1$.

$(1) \Rightarrow (2)$
Let $0 \to A \to \K_A \to Y \to 0$ be an exact sequence of $A$-modules such that $\mu_A(Y) = \e_\fkn^0(Y)$, where $\fkn = \fkm \times \fkp$ is the maximal ideal of $A$. Let us choose a parameter $f$ of $R$ so that $f$ is superficial for $Y$ with respect to $\fkn$ and $R/[\fkp + (f)]$ is a regular local ring of dimension $d-2$. Then $A/fA$ is an almost Gorenstein local ring (see the proof of Theorem \ref{3.9} (2)) and $$A/fA = R/(f) \ltimes \fkp/f\fkp \cong R/(f) \ltimes [\fkp + (f)]/(f),$$ which shows $R/(f)$ is an almost Gorenstein local ring. Thus $R$ is almost Gorenstein by Theorem \ref{3.9}.

$(2) \Rightarrow (1)$
We consider the exact sequence $0 \to R \to \K_R \to X \to 0$ of $R$-modules with $\mu_R(X) = \e_\fkm^0(X)$ and choose 
a parameter $f$ of $R$ so that $f$ is superficial for $X$ with respect to $\fkm$ and $R/[\fkp + (f)]$ is a regular local ring of  dimension $d-2$. Then because $A/fA \cong R/(f) \ltimes [\fkp + (f)]/(f)$, the ring $A/fA$ is an almost Gorenstein local ring. Hence by Theorem \ref{3.9} $A$ is an almost Gorenstein local ring.
\end{proof}

The following example extends \cite[Example 6.10]{GMP}. 

\begin{ex}\label{5.5}
Let $(R,\fkm)$ be a Gorenstein local ring of dimension $d \ge 1$ and infinite residue class field. Let $\fkp \in \Spec R$ and assume that $R/\fkp$ is a regular local ring of dimension $d-1$. We set $A = R \ltimes \fkp$. Then, thanks to Theorem \ref{5.4},  $A$ is an almost Gorenstein local ring. Therefore because $\fkp \times \fkp \in \operatorname{Spec} A$ with $A/[\fkp\times \fkp] \cong R/\fkp$, setting 
$$
R_n = \begin{cases}
R & (n=0)\\
R_{n-1}\ltimes \fkp_{n-1}&(n > 0)
\end{cases},
\qquad
\fkp_n = \begin{cases}
\fkp & (n=0)\\
\fkp_{n-1}\ltimes \fkp_{n-1}&(n > 0)
\end{cases},
$$
we get an infinite family $\{R_n\}_{n \ge 0}$ of almost Gorenstein local rings. Notice that $R_n$ is not a Gorenstein ring, if $n \ge 2$ (\cite[Lemma 6.6]{GMP}).
\end{ex}



\section{Generalized Gorenstein local rings}\label{pG}
Throughout this section let $(R,\fkm)$ denote a Noetherian local ring of dimension $d \ge0$. We explore  a special class of almost Gorenstein local rings, which we call semi-Gorenstein.

We begin with the definition.

\begin{defn}\label{def7.1}
We say that $R$ is a semi-Gorenstein local ring, if 
$R$ is an almost Gorenstein local ring, that is $R$ is a Cohen-Macaulay local ring having a canonical module $\rmK_R$ equipped with an exact sequence 
$$(\sharp) \ \ \ 
0 \to R \to \rmK_R \to C \to 0
$$
of $R$-modules such that $\mu_R(C) = \e^0_{\fkm}(C)$, where either $C = (0)$, or $C \neq (0)$ and there exist $R$-submodules $\{C_i\}_{1 \le i \le \ell}$ of $C$ such that $C = \oplus_{i=1}^\ell C_i$ and $\mu_R(C_i) = 1$ for all $1 \le i \le \ell$.\end{defn}

Therefore, every Gorenstein local ring is a semi-Gorenstein local ring (take   $C = (0)$) and every one-dimensional almost Gorenstein local ring is semi-Gorenstein, since $\fkm C = (0)$. We notice that in exact sequence $(\sharp)$ of Definition \ref{def7.1}, if $C \ne (0)$, then each $C_i$ is a cyclic Ulrich $R$-module of dimension $d-1$, whence  $C_i \cong R/\fkp_i$ for some $\fkp_i \in \Spec R$ such that $R/\fkp_i$ is a regular local ring of dimension $d-1$.

We note the following. This  is no longer true for the almost Gorenstein property, as we will show in Section \ref{assgr} (see Remark \ref{3}).

\begin{prop}\label{7.0}
Let $R$ be a semi-Gorenstein local ring. Then for every $\fkp \in \Spec R$ the local ring $R_\fkp$ is semi-Gorenstein.
\end{prop}

\begin{proof} We may assume that $R$ is not a Gorenstein ring. Choose an exact sequence $$0 \to R \to \rmK_R \to C \to 0$$ of $R$-modules which satisfies the condition in Definition \ref{def7.1}.  Hence $C = \oplus_{i=1}^\ell R/\fkp_i$, where for each $1 \le i \le \ell$, $\fkp_i \in \Spec R$ and $R/\fkp_i$ is a regular local ring of dimension $d-1$. Let $\fkp \in \Spec R$. Then since $\rmK_{R_\fkp}= (\rmK_R)_\fkp$, we get an exact sequence
$$0 \to R_\fkp \to \rmK_{R_\fkp} \to C_\fkp \to 0$$
of $R_\fkp$-modules, where  $C_\fkp = \oplus_{\fkp_i \subseteq \fkp}R_\fkp/\fkp_iR_\fkp$ is a direct sum of finite cyclic Ulrich $R_\fkp$-modules $R_\fkp/\fkp_iR_\fkp$, so that by definition the local ring $R_\fkp$ is semi-Gorenstein.
\end{proof}

Let us define  the following.

\begin{defn}[cf. \cite{BDF}]\label{def7.3}
An almost Gorenstein local ring $R$ is said to be pseudo-Gorenstein, if $\rmr (R) \le 2$. 
\end{defn}

\begin{prop}\label{7.2}
Let $R$ be a pseudo-Gorenstein local ring. Then $R$ is semi-Gorenstein and for every $\fkp \in \Spec R$ the local ring $R_\fkp$ is pseudo-Gorenstein. 
\end{prop}

\begin{proof} We may assume that $\rmr(R) = 2$. Because $R$ is not a regular local  ring, in the exact sequence $0 \to R \overset{\varphi}{\to} \rmK_R \to C \to 0$ of  Definition \ref{3.3} we get $\varphi (1) \not\in \fkm \rmK_R$ by Corollary \ref{3.8}, whence $\mu_R(C) = \mu_R(\rmK_R) -1 = 1$. Therefore $R$ is semi-Gorenstein, and the second assertion follows from Proposition \ref{7.0}.
\end{proof}

We note one example.

\begin{ex} Let $k[[t]]$ be the formal poser series ring over a field $k$. For an integer $a \ge 4$ we  set $$R = k[[t^{a+i} \mid 0 \le i \le a-1\ \text{but}~i \ne a-2]]$$ in $k[[t]]$. Then $\rmK_R = R + Rt^{a-1}$ and $\fkm \rmK_R \subseteq R$. Hence $R$ is a pseudo-Gorenstein local ring with $\rmr (R) = 2$. 
\end{ex}

Whether $C$ is decomposed into a direct sum of cyclic $R$-modules depends on the choice of exact sequences $0 \to R \to \rmK_R \to C \to 0$ with $\mu_R(C) = \e_\fkm^0(C)$, although $R$ is semi-Gorenstein. Let us note one example.

\begin{ex}
Let $S = k[X,Y]$ be the polynomial ring over a field $k$ and consider the Veronesean subring $R=k[X^4,X^3Y,X^2Y^2,XY^3,Y^4]$ of $S$ with order $4$. Then $\rmK_R = (X^3Y, X^2Y^2, XY^3)$ is the graded canonical module of $R$. The exact sequence
$$0 \to R \overset{\varphi}{\to} \rmK_R(1) \to R/(X^3Y, X^2Y^2, XY^3, Y^4) \oplus R/(X^4, X^3Y, X^2Y^2, XY^3) \to 0$$
of graded $R$-modules with $\varphi(1) = X^2Y^2$  shows that the local ring $R_\fkM$ is semi-Gorenstein, where $\fkM = R_+$. However in the exact sequence  
$$0 \to R \overset{\psi}{\to} \rmK_R(1) \to D \to 0$$
with  $\psi (1) = XY^3$, $D_\fkM$ is an Ulrich $R_\fkm$-module of dimension one,  but $D_\fkM$  is indecomposable.  In fact, setting $A =R_\fkM$ and $C = D_\fkM$, suppose $C \cong A/\fkp_1 \oplus A/\fkp_2$ for some regular local rings $A/\fkp_i~(\fkp_i \in \Spec A)$ of dimension one. Let $\fka = XY^3A :_{A}(X^3Y, X^2Y^2, XY^3)A$. Then $\fka  = (0):_{A}C = \fkp_1 \cap \fkp_2$, so that $\fka $ should be a radical ideal of $A$, which is impossible, because $X^6Y^2 \in \fka$ but $X^3Y \not\in \fka$. 
\end{ex}

Let us examine the non-zerodivisor characterization.

\begin{thm}\label{7.3}
Suppose that $R/\fkm$ is infinite. If $R$ is a semi-Gorenstein local ring of dimension $d \ge 2$, then $R/(f)$ is a semi-Gorenstein local ring for a general non-zerodivisor $f \in \fkm \setminus \fkm^2$.
\end{thm}
\begin{proof}
We may assume $R$ is not a Gorenstein ring. We look at exact sequence $(\sharp) ~0 \to R \to \rmK_R \to C \to 0$ of Definition \ref{def7.1}, 
where $C = \oplus_{i=1}^{r-1}R/\fkp_i$~($r = \rmr (R)$) and each $R/\fkp_i$ is a regular local ring of dimension $d-1$. Then $R/(f)$ is a semi-Gorenstein local ring for every $f \in \fkm$ such that $f \not\in \bigcup_{i=1}^{\ell}[\fkm^2 + \fkp_i] \cup \bigcup_{\fkp \in \Ass R}\fkp.
$
\end{proof}

We now give a characterization of semi-Gorenstein local rings in terms of their minimal free resolutions, which is a broad generalization of \cite[Corollary 4.2]{GMP}.

\begin{thm}\label{p} Let $(S, \fkn)$ be a regular local ring and $\fka \subsetneq S$ an ideal of $S$ with $n = \height_S \fka$. Let $R = S/\fka$.
Then the following conditions are equivalent.
\begin{enumerate}[$(1)$]
\item $R$ is a semi-Gorenstein local ring but not a Gorenstein ring.
\item $R$ is Cohen-Macaulay, $n \ge 2$, $r=\rmr(R) \ge 2$, and $R$ has a minimal $S$-free resolution of the form$\mathrm{:}$
$$
0 \to F_n =S^r \overset{\Bbb M}{\to} F_{n-1} = S^q \to F_{n-2} \to \cdots \to F_1 \to F_0=S \to R \to 0
$$
where $${}^t \Bbb M = 
\begin{pmatrix}
y_{21} y_{22} \cdots y_{2\ell} & y_{31} y_{32} \cdots y_{3\ell} & \cdots & y_{r1} y_{r2} \cdots y_{r\ell} & z_1 z_2 \cdots z_m \\
x_{21} x_{22} \cdots x_{2\ell} & 0 & 0 & 0  & 0 \\
0 & x_{31} x_{32} \cdots x_{3\ell} & 0 & 0  & 0 \\
\vdots & \vdots & \ddots & \vdots & \vdots \\
0 & 0 & 0 & x_{r1} x_{r2} \cdots x_{r\ell} & 0,
\end{pmatrix},
$$
$\ell= n+1$, $q \ge (r-1)\ell$, $m= q-(r-1)\ell$, and $x_{i1}, x_{i2}, \ldots, x_{i\ell}$ is a part of a regular system of parameters of $S$ for every $2 \le i \le r$.
\end{enumerate}
When this is the case, one has the equality
$$
\fka = (z_1, z_2, \ldots, z_m) + \sum_{i=2}^r {\rm I}_2
\left(\begin{smallmatrix}
y_{i1} & y_{i2} & \cdots & y_{i\ell} \\
x_{i1} & y_{i2} & \cdots & x_{i\ell}
\end{smallmatrix}\right),
$$ where $\rmI_2(\Bbb N)$ denotes the ideal of $S$ generated by $2 \times 2$ minors of the submatrix $\Bbb N=\left(\begin{smallmatrix}
y_{i1} & y_{i2} & \cdots & y_{i\ell} \\
x_{i1} & y_{i2} & \cdots & x_{i\ell}
\end{smallmatrix}\right)$ of $\Bbb M$.
\end{thm}

\begin{proof} 
$(1) \Rightarrow (2)$
Choose an exact sequence 
$$
0 \to R \overset{\varphi}{\to} \rmK_R \to C \to 0
$$
of $R$-modules
so that $C = \oplus_{i=2}^r S/\fkp_i$, where each $S/\fkp_i~(\fkp_i \in \Spec S)$ is a regular local ring of dimension $d-1$. Let $\fkp_i = (x_{ij} \mid 1 \le j \le \ell)$ with   a part $\{x_{ij}\}_{1 \le j \le \ell}$ of a regular system of parameters of $S$, where $\ell = n + 1~(=\height_S\fkp_i)$. We set $f_1 = \varphi(1) \in \rmK_R$ and consider the $S$-isomorphism 
$$
\rmK_R/Sf_1 \overset{\psi} \longrightarrow S/\fkp_2 \oplus S/\fkp_3 \oplus \cdots \oplus S/\fkp_r,
$$
choosing elements $\{f_i \in \rmK_R\}_{2 \le i \le r}$ so that 
$$\psi(\overline{f_i}) =(0,\ldots, 0, \overset{\underset{\vee}{i}}{1}, 0 \ldots, 0) \in S/\fkp_2 \oplus S/\fkp_3 \oplus \cdots \oplus S/\fkp_r,$$ where $\overline{f_i}$ denotes the image of $f_i$ in $\rmK_R/Sf_1$. Hence $\{f_i\}_{1 \le i \le r}$ is a minimal system of generators of the $S$-module $\rmK_R$. Let $\{\mathbf{e}_i\}_{1\le i \le r}$ denote the standard basis of $S^r$ and let $\varepsilon : S^r \to \rmK_R$ be the homomorphism defined by $\varepsilon(\mathbf{e}_i) = f_i$ for each $1 \le i \le r$. We now look at  the exact sequence
$$
0 \to L \to S^r \overset{\varepsilon}{\longrightarrow} \rmK_R \to 0.
$$
Then because $x_{ij}\overline{f_i} = 0$ in $\rmK_R/Sf_1$, we get  $y_{ij}f_1 + x_{ij}f_i= 0$ in $\rmK_R$ for some $y_{ij} \in \fkn$. Set 
$ \mathbf{a}_{ij} = y_{ij}\mathbf{e}_{1} + x_{ij}\mathbf{e}_i \in L$ for each $2 \le i \le r$ and $1 \le j \le \ell$. Then $\{ \mathbf{a}_{ij} \}_{2\le i \le r, 1 \le j \le \ell}$ is a part of a minimal basis of $L$, because $\{x_{ij}\}_{1 \le j \le \ell}$ is a part of a regular system of parameters of $S$ (use the fact  that $L \subseteq \fkn S^r$). Hence $q \ge (r-1) \ell$.

Let 
$\mathbf{a} \in L$ and write $\mathbf{a} = \sum_{i=1}^ra_i\mathbf{e}_i$ with $a_i \in S$. Then $a_i \in \fkp_i =(x_{ij} \mid 1 \le j \le \ell)$ for every $2 \le i \le r$, because $\sum_{i=2}^r a_i\overline{f_i} = 0$ in $\rmK_R/Sf_1$.  Therefore, writing $a_i = \sum_{j=1}^\ell c_{ij}x_{ij}$ with $c_{ij} \in S$, we get
$\mathbf{a}-\sum_{i=2}^rc_{ij}\mathbf{a}_{ij} = c\mathbf{e}_1$ for some $c \in S$, which shows $L$ is minimally generated by $\{\mathbf{a}_{ij}\}_{2 \le i \le r, 1 \le j \le \ell}$ together with some elements $\{z_k\mathbf{e}_1\}_{1 \le k \le m}$~($z_k \in S$). Thus the $S$-module $\rmK_R=\Ext_S^n(R,S)$ possesses a minimal free resolution
$$ 
\cdots \to S^q \overset{{\Bbb M}}{\longrightarrow} S^r \overset{\varepsilon}{\to} \rmK_R \to 0
$$ 
with $q = m + (r-1)\ell$, 
in which the matrix ${\Bbb M}$ has the required form. Since $R \cong \Ext_S^n(\rmK_R,S)$ (\cite[Satz 6.1]{HK}), the minimal $S$-free resolution of $R$ is obtained, by taking the $S$-dual, from the minimal free resolution of $\rmK_R$, so that  assertion (2) follows.

$(2) \Rightarrow (1)$
We look at the presentation 
$$
S^q \overset{{}^t \Bbb M}{\longrightarrow} S^r \overset{\varepsilon}{\to} \rmK_R \to 0
$$
of $\rmK_R=\Ext_S^n(R,S)$. Let $\{\mathbf{e}_i\}_{1 \le i \le r}$ be the standard basis of $S^r$ and set $f_1 = \varepsilon(\mathbf{e}_1)$. 
Then 
$$
\rmK_R/Rf_1 \cong S^r/[\operatorname{Im}{}^t \Bbb M + S\mathbf{e}_1] \cong \bigoplus_{i=2}^r S/(x_{ij} \mid 1\le j \le \ell),
$$
where each $S/(x_{ij} \mid 1\le j \le \ell)$ is a regular local ring of dimension $d-1$, so that $R$ is a semi-Gorenstein local ring, because the sequence
$$0 \to R \overset{\varphi}{\to} \rmK_R \to \bigoplus_{i=2}^r S/(x_{ij} \mid 1\le j \le \ell) \to 0$$
is exact by Lemma \ref{3.1} (1), where $\varphi(1) = f_1$.
This completes the proof of equivalence $(1) \Leftrightarrow (2)$.

To prove the last equality in Theorem \ref{p}, we need a preliminary step.
Let $S$ be a Noetherian local ring. Let $\xx = x_1, x_2, \ldots, x_{\ell}$ be a regular sequence in $S$ and $\yy = y_1, y_2, \ldots, y_{\ell}$ a sequence of elements in $S$. We denote by $\Bbb K = \Bbb K_\bullet (\xx,S)$ the Koszul complex of $S$ associated to $x_1, x_2, \ldots, x_{\ell}$ and by $\Bbb L = \Bbb K_\bullet(\yy,S)$ the Koszul complex of $S$ associated to $\yy=y_1, y_2, \ldots, y_{\ell}$. We consider the diagram.
\[
\xymatrix{
	\Bbb K_2 \ar[r]^{\partial_2^{\Bbb K}} 
 & \Bbb K_1=\Bbb L_1 \ar[r]^{\partial_1^{\Bbb K}} \ar[d]^{\partial_1^{\Bbb L}} & K_0(\xx;S) = S \\
                 &   L_0(\yy;S) = S & 
}
\]
 and set $L = \partial_1^\Bbb L(\operatorname{Ker}~\partial_1^\Bbb K)$. Then because $\operatorname{Ker}~\partial_1^\Bbb K =\operatorname{Im}~\partial_2^\Bbb K$, we get the following.

\begin{lem}\label{p7.8}
$L = (x_{\alpha}y_{\beta}-x_{\beta}y_{\alpha} \mid 1 \le \alpha < \beta \le \ell )=  {\rm I}_2
\left(\begin{smallmatrix}
y_1 & y_2 & \cdots & y_{\ell}\\
x_{1} & y_{2} & \cdots & x_{\ell}
\end{smallmatrix}\right)$.
\end{lem}

 Let us now check the last assertion in Theorem \ref{p}. We maintain the notation in the proof of implication (1) $\Rightarrow$ (2). Let $a \in S$. Then $a \in \fka$ if and only if $af_1=0$, because $\fka = (0):_S\rmK_R$. The latter condition is equivalent to saying that 
$$
a\mathbf{e_1} = \sum_{2\le i \le r,~1 \le j \le \ell}c_{ij}(y_{ij}\mathbf{e_i}+x_{ij}\mathbf{e_i}) + \sum_{k=1}^m d_k(z_k\mathbf{e_1})
$$
for some $c_{ij},~ d_k \in S$, that is
$$a = \sum_{i=2}^r \left( \sum_{j=1}^{\ell}c_{ij}y_{ij} \right) + \sum_{k=1}^m d_k z_k \ \ \text{and}\ \ \sum_{j=1}^{\ell}c_{ij}x_{ij} = 0\ \ \text{for~all}\ \ 2 \le i \le r.$$
If $\sum_{j=1}^{\ell}c_{ij}x_{ij} = 0$, then by Lemma \ref{p7.8} we get 
$$
\sum_{j=1}^{\ell}c_{ij}y_{ij} \in {\rm I}_2
\left(\begin{smallmatrix}
y_{i1} & y_{i2} & \cdots & y_{i\ell} \\
x_{i1} & y_{i2} & \cdots & x_{i\ell}
\end{smallmatrix}\right),
$$
because $x_{i1}, x_{i2}, \ldots, x_{i\ell}$ is an $S$-sequence. Hence $$
\fka \subseteq (z_1, z_2, \ldots, z_m) + \sum_{i=2}^r {\rm I}_2
\left(\begin{smallmatrix}
y_{i1} & y_{i2} & \cdots & y_{i\ell} \\
x_{i1} & y_{i2} & \cdots & x_{i\ell}
\end{smallmatrix}\right).
$$
To see the reverse inclusion, notice that $z_k\in \fka$ for every $1 \le k \le m$, because $z_k f_1=0$. Let $2 \le i \le r$ and $1 \le \alpha < \beta \le \ell$. Then since
$$(x_{i\alpha} y_{i\beta}-x_{i\beta} y_{i\alpha})\mathbf{e}_1 = x_{i\alpha}(y_{i\beta}\mathbf{e}_1+x_{i\beta}\mathbf{e}_i)-x_{i\beta}(y_{i\alpha}\mathbf{e}_1+x_{i\alpha}\mathbf{e}_i) \in \Ker \varepsilon,$$
we get $(x_{i\alpha} y_{i\beta} - x_{i\beta} x_{i\alpha})f_1=0$, so that $x_{i\alpha} y_{i\beta} - x_{i\beta} y_{i\alpha} \in \fka$. Thus $(z_1, z_2, \ldots, z_m) + \sum_{i=2}^r {\rm I}_2
\left(\begin{smallmatrix}
y_{i1} & y_{i2} & \cdots & y_{i\ell} \\
x_{i1} & y_{i2} & \cdots & x_{i\ell}
\end{smallmatrix}\right) \subseteq \fka$, which completes the proof of Theorem \ref{p}.
\end{proof}

\begin{cor}\label{p7.10}
With the notation of Theorem $\ref{p}$ suppose that assertion $(1)$ holds true. We then have the following.
\begin{enumerate}
\item[$(1)$] If $n = 2$, then $r = 2$ and $q = 3$, so that ${}^t\Bbb M = \left(\begin{smallmatrix}
y_{21} & y_{22} & y_{23} \\
x_{21} & x_{22} & x_{23} 
\end{smallmatrix}\right)$.
\item[$(2)$] Suppose that $\fka \subseteq \fkn^2$. If $R$ has maximal embedding dimension, then $r = n$ and $q= n^2 -1$, so that $m = 0$.
\end{enumerate}
\end{cor}

\begin{proof}
(1) Since $n = 2$, we get $q = r+1 \ge (r-1)\ell$. Hence $2 \ge (r-1)(\ell -1) = 2(r-1)$, as $\ell = 3$. Hence $r = 2$ and $q = 3$.

(2) We set $v = \ell_R(\fkm/\fkm^2)~(=\dim S)$, $e = \rme_\fkm^0(R)$, and $d = \dim R ~(= v-n)$. Since $v = e + d - 1$, we then have $r = v - d=n$, while $q = (e-2){\cdot}\binom{e}{e-1}$ by \cite{S1}. Hence $q = n^2-1=(r-1)\ell$, so that $m = 0$.
\end{proof}

One cannot expect $m = 0$ in general, although assertion (1)  in Theorem \ref{p} holds true. Let us note one example.

\begin{ex} Let $V = k[[t]]$ be the formal power series ring over a field $k$ and set $R = k[[t^5,t^6,t^7,t^9]]$. Let $S=k[[X,Y,Z,W]]$ be the formal power series ring and let $\varphi : S \to R$ be the $k$-algebra map defined by $\varphi (X) = t^5, \varphi(Y)=t^6, \varphi(Z) = t^7$, and $\varphi(W) = t^9$. Then $R$ has a minimal $S$-free resolution
of the form
$$0 \to S^2 \overset{\Bbb M}{\to} S^6 \to S^5 \to S \to R \to 0,$$
where ${}^t\Bbb M = 
\left(\begin{smallmatrix}
W & X^2 & XY & YZ & Y^2 - XZ & Z^2 - XW \\
X & Y & Z & W & 0 & 0\\
\end{smallmatrix}\right)
$. Hence $R$ is semi-Gorenstein with $\rmr (R) = 2$ and $$\Ker \varphi = (Y^2 - XZ, Z^2 - XW) +\rmI_2
\left(\begin{smallmatrix}
W & X^2 & XY & YZ \\
X & Y & Z & W \\
\end{smallmatrix}\right).$$
\end{ex}

\section{Almost Gorenstein graded rings}\label{aggr}

We now explore graded rings. In this section let $R=\bigoplus_{n\ge 0}R_n$ be a Noetherian graded ring such that $k=R_0$ is a local ring. Let $d=\operatorname{dim} R$ and let $\fkM$ be the unique graded maximal ideal of $R$. Assume that $R$ is a Cohen-Macaulay ring, admitting the graded canonical module $\K_R$. The latter condition is equivalent to saying that $k=R_0$ is a homomorphic image of a Gorenstein ring (\cite{GW1, GW2}). We put $a=\rma (R)$. Hence 
$
a=\operatorname{max} \{n \in \Z \mid [\rmH^d_\fkM({R})]_n \ne (0) \} 
=-\operatorname{min} \{ n \in \Z \mid [\K_R]_n \ne (0) \}.
$

\begin{defn}\label{6.1} We say ${R}$ is an almost Gorenstein graded ring, if there exists an exact sequence 
$0\rightarrow{R}\rightarrow \K_{R}(-a)\rightarrow C\rightarrow 0$
of graded ${R}-$modules such that $\mu_{R}(C)=\e^0_\fkM(C)$. Here $\K_R(-a)$ denotes the graded $R$-module whose underlying $R$-module is the same as that of $\K_R$ and the grading is given by $[\K_R(-a)]_n = [\K_R]_{n - a}$ for all $n \in \Bbb Z$. 
\end{defn}

The ring ${R}$ is an almost Gorenstein graded ring, if $R$ is a Gorenstein ring. As $(\K_R)_\fkM = \K_{R_\fkM}$, the ring $R_\fkM$ is an almost Gorenstein local ring, once $R$ is an almost Gorenstein graded ring.


The condition stated in Definition \ref{6.1} is rather strong, as we show in the following. @Firstly we note:


\begin{thm}[\cite{GI}]\label{6.2} Suppose that $A$ is a Gorenstein local ring and let $I~(\ne A)$ be an ideal of $A$. If $\gr_I (A) =\bigoplus_{n \ge 0}I^n/I^{n+1}$ is an almost Gorenstein graded ring, then $\gr_I(A)$ is a Gorenstein ring.
\end{thm}

We secondly explore the almost Gorenstein property in the Rees algebras of parameter ideals. Let $(A,\m)$ be a Gorenstein local ring of dimension $d\ge 3$ and $Q=(a_1,a_2,\ldots,a_d)$  a parameter ideal of $A$. We set $\mathcal R={\mathcal R}(Q)=A[Qt] \subseteq A[t]$, where $t$ is an indeterminate. We then have the following.

\begin{thm}\label{6.3}
The Rees algebra $\mathcal R = {\mathcal R}(Q)$ of $Q$ is an almost Gorenstein graded ring if and only if $Q=\m$ $($and hence $A$ is a regular local ring$)$.
\end{thm}

To prove Theorem \ref{6.3}, we need some preliminary steps. Let $B=A[X_1,X_2,\ldots,X_d]$ be the polynomial ring and let $\varphi: B\rightarrow \mathcal R$ be the homomorphism of $A$-algebras defined by $\varphi(X_i)=a_it$ for all $1\le i\le d$. Then $\varphi$ preserves grading and $\Ker\varphi=I_2\left(\begin{smallmatrix}X_1& X_2&\ldots&X_d\\a_1& a_2&\ldots& a_d\end{smallmatrix}\right)$ is a perfect ideal of $B$, since $\dim{\mathcal R}=d+1$. Therefore $\mathcal R$ is a Cohen-Macaulay ring. Because $\fkM=\sqrt{(X_1, \{X_{i+1}-a_i\}_{1\le i\le d-1},-a_d)R}$ and $$\mathcal R/(X_1,\{X_{i+1}-a_i\}_{1\le i\le d-1},-a_d)\mathcal R \cong A/[(a_1)+(a_2,a_3,\ldots,a_d)^2],$$
we get $\rmr(\mathcal R)=\rmr(A/[(a_1)+(a_2,a_3,\ldots,a_d)^2])=d-1 \ge 2.$ Hence $\mathcal R$ is not a Gorenstein ring.

Let $G=\text{\rm gr}_Q (A)$ and choose the canonical $Q$-filtration $\omega =\{\omega_n\}_{n\in\Z}$ of $A$ which satisfies the following conditions (\cite[HTZ]{GI}).
\begin{enumerate}
\item[$(\rmi)$] $\omega_n=A$, if $n<d$ and $\omega_n=Q^{n-d}\omega_d$, if $n\ge d$. 
\item[$(\mathrm{ii})$] $[{\mathcal R}(\omega)]_+ \cong \K_{\mathcal R}$ and $\gr_\omega(A)(-1)\cong K_G$ as graded ${\mathcal R}-$modules, where ${\mathcal R}(\omega) = \sum_{n\ge 0}\omega_nt^n$ and $\gr_{\omega} (A) =\bigoplus_{n \ge 0}\omega_n/\omega_{n+1}$.
\end{enumerate}
On the other hand, since $G=(A/Q)[T_1,T_2,\ldots,T_d]$ is the polynomial ring, we get $K_G\cong G(-d).$ Therefore $\omega_{d-1}/\omega_d\cong A/Q$ by condition (ii) and hence $\omega_d=Q$, because $\omega_{d-1}=A$ by condition (i). Consequently $\omega_n=Q^{n-d+1}$ for all $n\ge d$. Therefore $\K_R=\sum_{n=1}^{d-2}At^n+{\mathcal R} t^{d-1}$, so that we get the exact sequence $0 \rightarrow \mathcal R \xrightarrow{\psi} \K_{\mathcal R}(1) \rightarrow C \rightarrow 0$ of graded $\mathcal R$-modules, where $\psi(1)= t$. Hence  $\a(\mathcal R)=-1$, because $C_n=(0)$, if $n\le 0$.

Let $f=\overline{t^{d-1}}$ denote the image of $t^{d-1}$ in $C$ and put $D=C/\mathcal R f$. Then it is standard to check that $(0):_{{\mathcal R}} C=(0):_{{\mathcal R}} f=Q^{d-2}{\mathcal R}$.
Hence $D=\sum_{n=0}^{d-3}D_n$ is a finitely graded ${\mathcal R}$-module and $\ell_A(D)<\infty$.

Let $\overline{\mathcal R}={\mathcal R}/Q^{d-2}{\mathcal R}$ and look at the exact sequence
\begin{equation}
0\rightarrow\overline{\mathcal R}(2-d)\xrightarrow{\varphi} C\rightarrow D\rightarrow 0\tag{$\sharp$}
\end{equation}
of graded ${\mathcal R}$-modules, where $\varphi(1)=f$. Then since $\overline{\mathcal R}_0=A/Q^{d-2}$ is an Artinian local ring, the ideal $\fkM \overline{\mathcal R}$ of $\overline{\mathcal R}$ contains $[\overline{\mathcal R}]_+ = (a_1t,a_2t,\ldots,a_dt) \overline{\mathcal R}$ as a reduction. Therefore thanks to exact sequence ($\sharp$), we get $\e_\fkM^0(C)=\e^0_{\fkM\overline{\mathcal R}}(C)=\e^0_{[\overline{\mathcal R}]_+}(C)=\e^0_{[\overline{\mathcal R}]_+}(\overline{\mathcal R})$, because $\ell_A(D)<\infty$ but $\dim_{\overline{\mathcal R}}C=\dim\overline{\mathcal R}=d$. In order to compute $\e^0_{[\overline{\mathcal R}]_+}(\overline{\mathcal R})$, it suffices to see the Hilbert function $\ell_A([\overline{\mathcal R}]_n)$. Since 
$$
\textstyle
\ell_A([\overline{\mathcal R}]_n)
=\ell_A(A/Q^{n+d-2})-\ell_A(A/Q^n)
=\ell_A(A/Q){\cdot}\left[\binom{n+2d-3}{d}-\binom{n+d-1}{d}\right]
$$
for all $n\ge 0$, we readily get $\e^0_\fkM(C)=\e^0_{[\overline{\mathcal R}]_+}(\overline{\mathcal R})=\ell_A(A/Q)(d-2)$.
Summarizing the above observations, we get the following, because $\mu_{\mathcal R}(C)=\rmr (R) -1 = d-2$.

\begin{lem}\label{6.4} $\mu_{\mathcal R}(C)=\e^0_\fkM(C)$ if and only if $\ell_A(A/Q)=1$, i.e., $Q=\m$.
\end{lem}

We are now ready to prove Theorem \ref{6.3}.

\begin{proof}[Proof of Theorem $\ref{6.3}$]
If $Q=\m$, then $\mu_{\mathcal R}(C)=\e^0_\fkM(C)$ by Lemma \ref{6.4}.
Let us show the \textit{only if} part.
Since ${\mathcal R}$ is an almost Gorenstein graded ring, we get an exact sequence $0 \rightarrow R \xrightarrow{\rho} \K_R(1)\rightarrow X\rightarrow 0$ of graded ${\mathcal R}$-modules such that $\mu_{\mathcal R}(X)=\e^0_\fkM(X)$. Let $\xi=\rho(1)\in [\K_{\mathcal R}]_1=At$ and remember that $\xi\notin \fkM \K_{\mathcal R}$ (Corollary 3.9). We then have $\xi=\varepsilon t$ for some $\varepsilon \in \rmU (A)$ and therefore $X \cong C =(\K_{\mathcal R}/{\mathcal R} t)(1)$ as a graded ${\mathcal R}$-module. Hence $Q=\m$ by Lemma \ref{6.4}, because $\mu_R(X)=\e^0_\fkM(X)$. 
\end{proof}

We thirdly explore the almost Gorenstein property in the polynomial extensions.

\begin{thm}\label{13}
Let $(R,\fkm)$ be a Noetherian local ring with infinite residue class field. Let $S = R[X_1, X_2, \ldots, X_n]$ be the polynomial ring and consider $S$ to be a $\Bbb Z$-graded ring such that $S_0 = R$ and $\operatorname{deg}X_i = 1$ for every $1 \le i \le n$. Then the following conditions are equivalent.
\begin{enumerate}
\item[$(1)$] $R$ is an almost Gorenstein local ring.
\item[$(2)$] $S$ is an almost Gorenstein graded ring.
\end{enumerate} 
\end{thm}

\begin{proof} We put  $\fkM = \fkm S + S_+$.

(2) $\Rightarrow$ (1) This follows from Theorem \ref{3.5}, because $S_\fkM$ is an almost Gorenstein local ring and the fiber ring $S_\fkM/\fkm S_\fkM$ is a regular local ring.

(1) $\Rightarrow$ (2) We may assume that $R$ is not Gorenstein. Hence $d = \dim R > 0$. Choose an exact sequence
$$(\sharp_1)\ \ \ 0 \to R \to \rmK_R \to C \to 0$$
of $R$-modules so that $\mu_R(C) = \rme^0_\fkm(C)$ and consider $R$ to be a $\Bbb Z$-graded ring trivially.  Then, tensoring sequence $(\sharp_1)$ by $S$,  we get
the exact sequence
$$(\sharp_2)\ \ \ 0 \to S \to S \otimes_R\rmK_R \to S \otimes_RC \to 0$$
of graded $S$-modules. Let $D = S \otimes_RC$. Then $D$ is a Cohen-Macaulay graded $S$-module of $\dim D = \dim_RC + n = \dim S - 1$. We choose elements $f_1, f_2, \ldots, f_{d-1} \in \fkm$ so that $\fkm C = \fkq C$, where $\fkq = (f_1, f_2, \ldots, f_{d-1})$. Then $$\fkM D = (\fkm S) D + S_+D=(\fkq S) D + S_+D,$$ so that $\mu_S(D) = \rme_\fkM^0(D)$. Therefore, exact sequence $(\sharp_2)$ shows $S$ to be an almost Gorenstein graded ring, because $\rma (S) = -n$ and $S\otimes_R\rmK_R = \rmK_S(n)$.
\end{proof}

\begin{cor}\label{14}
Let $(R, \m)$ be an Artinian local ring and assume that the residue class field $R/\m$ of $R$ is infinite. If the polynomial ring $R[X_1, X_2, \ldots, X_n]$ is an almost Gorenstein graded ring for some $n \ge 1$, then $R$ is a Gorenstein ring.
\end{cor}

The last assertion of the following result is due to S.-i. Iai \cite[Theorem 1.1]{I}.

\begin{cor}\label{15}
Let $(R,\fkm)$ be a Noetherian local ring with $d = \dim R >0$ and infinite residue class field. Assume that $R$ is a homomorphic image of a Gorenstein local ring. We choose a system  $a_1, a_2, \ldots, a_d$ of parameters of $R$. Let $1 \le r \le d$ be an integer and set $\q = (a_1, a_2, \ldots, a_r)$. If $\operatorname{gr}_\q(R)$ is an almost Gorenstein graded ring, then $R$ is an almost Gorenstein local ring. In particular, $R$ is a Gorenstein local ring, if $r = d$ and $\operatorname{gr}_\q(R)$ is an almost Gorenstein graded ring.
\end{cor}

\begin{proof} The ring $R$ is  Cohen-Macaulay, because the associated graded ring $\operatorname{gr}_\q(R)$ of $\q$ is Cohen-Macaulay. Hence   $a_1, a_2, \ldots, a_r$ forms an $R$-regular sequence, so that $\operatorname{gr}_\q(R) = (R/\q)[X_1, X_2, \ldots,X_r]$ is the polynomial ring. Therefore  by Theorem \ref{13}, $R/\q$ is an almost Gorenstein local ring, whence  $R$ is almost Gorenstein. Remember that if $r = d$, then $R$ is Gorenstein by Corollary \ref{14}, because $\dim R/\q = 0$. \end{proof}

Unfortunately, even though $\R_\fkM$ is an almost Gorenstein local ring, $\R$ is not necessarily an almost Gorenstein graded ring. We explore one example.

\begin{ex}\label{6.5} Let $U=k[s,t]$ be the polynomial ring over a field $k$ and look at the  Cohen-Macaulay graded subring $\R=k[s,s^3t,s^3t^2,s^3t^3]$ of $U$. Then $\R_\fkM$ is almost Gorenstein. In fact, let $S=k[X,Y,Z,W]$ be the  weighted polynomial ring such that $\deg X=1$, $\deg Y=4$, $\deg Z=5$, and $\deg W=6$. Let $\psi : S \rightarrow \R$ be the $k$-algebra map defined by $\psi (X)=s$, $\psi(Y)=s^3t$, $\psi(Z)=s^3t^2$, and $\psi(W)=s^3t^3$. Then $\operatorname{Ker}\psi=I_2
\left(
\begin{smallmatrix}
X^3 & Y& Z\\
Y& Z& W
\end{smallmatrix}
\right)$
and the graded $S$-module $\R$ has a graded minimal free resolution
$$
0 \rightarrow S(-13)\oplus S(-14)
\xrightarrow{\left(\begin{smallmatrix}X^3 & Y\\Y&Z\\Z&W\end{smallmatrix}\right)}
  S(-10)
  \oplus 
  S(-9)
  \oplus
  S(-8)
\xrightarrow{(\Delta_1\ \Delta_2\ \Delta_3)}
S
\xrightarrow{\psi}
\R
\rightarrow 0,$$
where $\Delta_1=Z^2-YW$, $\Delta_2=X^3W-YZ,$ and $\Delta_3=Y^2-X^3Z$.
Therefore, because $\K_S \cong S(-16)$, taking the $\K_S$-dual of the resolution, we get the presentation
\begin{equation}
  S(-6)
  \oplus
  S(-7)
  \oplus
  S(-8)
\xrightarrow{\left(\begin{smallmatrix}X^3 & Y & Z\\Y & Z & W\end{smallmatrix}\right)} 
  S(-3)\oplus S(-2)
\xrightarrow{\varepsilon} K_\R\rightarrow 0\tag{$\sharp$}
\end{equation}
of the graded canonical module $K_\R$ of $\R$. Hence $\a(\R)=-2$. Let $\xi=\varepsilon(\binom{1}{0})\in[K_\R]_3$ and we have  the exact sequence
$0\rightarrow\R\xrightarrow{\varphi} K_\R(3)\rightarrow S/(Y,Z,W)(1)\rightarrow 0$
of graded $\R$-modules, where $\varphi(1)=\xi$. Hence $\R_\fkM$ is a semi-Gorenstein local ring. On the other hand, thanks to presentation $(\sharp)$ of $\K_\R$, we know $[\K_\R]_2=k\eta\ne (0)$, where $\eta=\varepsilon(\binom{0}{1})$. Hence if $\R$ is an almost Gorenstein graded ring, we must have $\mu_\R(\K_\R/\R\eta)=\e^0_\fkM(\K_\R/\R\eta)$, which is impossible, because $\K_\R/\R \eta \cong [S/(X^3,Y,Z)](-3)$.
\end{ex}

This example \ref{6.5} seems to suggest a correct definition of almost Gorenstein graded rings might be the following: \textit{there exists an exact sequence $0\rightarrow\R\rightarrow \K_\R(n)\rightarrow C\rightarrow 0$ of graded $\R$-modules for some $n\in\Z$ such that $\mu_\R(C)=\e^0_\fkM(C)$.} We would like to leave further investigations to readers as an open problem.



\section{Almost Gorenstein associated graded rings}\label{assgr}

The purpose of this section is to explore how the almost Gorenstein property of base local rings is inherited from that of the associated graded rings. Our goal  is the following.

\begin{thm}\label{1}
Let $(R, \m)$ be a Cohen-Macaulay local ring with infinite residue class field, possessing the canonical module $\rmK_R$. Let $I$ be an $\m$-primary ideal of $R$ and let $\operatorname{gr}_I(R)= \bigoplus_{n \ge 0}I^n/I^{n+1}$ be the associated graded ring of $I$. If $\operatorname{gr}_I(R)$ is an almost Gorenstein graded ring with $\rmr(\operatorname{gr}_I(R))) = \rmr(R)$, then $R$ is an almost Gorenstein local ring.
\end{thm}

Theorem \ref{1} is reduced, by induction on $\dim R$, to the case where $\dim R = 1$. Let us start from the key result of dimension one. Our setting is the following.

\begin{setting}\label{2} {\rm Let $R$ be a Cohen-Macaulay local ring of dimension one.  We consider a filtration ${\mathcal F}=\{ I_n \}_{n \in \Bbb Z}$ of ideals of $R$. Therefore $\{I_n \}_{n \in \Bbb Z}$ is a family of 
ideals of $R$ which satisfies the following three conditions: (1) $I_n = R$ for all  $n \le 0$ but $I_1\ne R$, (2) $I_n \supseteq I_{n+1}$ for all $n \in \Bbb Z$, and (3) $I_m I_n \subseteq I_{m+n}$ for all $m, n \in \Bbb Z$. Let $t$ be an indeterminate and we set
$$
\mathcal R = \mathcal R(\mathcal F) = \sum_{n \ge 0}I_nt^n \subseteq R[t],
$$
$$
\mathcal R' = \mathcal R'(\mathcal F) = \mathcal R[t^{-1}]= \sum_{n \in \Bbb Z}I_nt^n \subseteq R[t, t^{-1}], \ \ \text{and}\\
$$
$$
G = G(\mathcal F) = {\mathcal R}'(\mathcal F)/t^{-1}{\mathcal R}'(\mathcal F).
$$
We call them respectively the Rees algebra, the extended Rees algebra, and the associated graded ring of ${\mathcal  F}$. We assume the following three conditions are satisfied:

\begin{enumerate}
\item $R$ is a homomorphic image of a Gorenstein ring,
\item $\mathcal R$ is a Noetherian ring, and
\item $G$ is a Cohen-Macaulay ring.
\end{enumerate}
}
\end{setting}

Let $\rma (G)= \max \{n \in \Bbb Z \mid [\rmH_\fkM^1(G)]_n \ne (0)\}
$
(\cite{GW1}), where $\{[\rmH_\fkM^1({G})]_n\}_{n \in \Bbb Z}$ stands for the homogeneous components of the first graded local cohomology module $\rmH_\fkM^1({G})$ of ${G}$ with respect to the graded maximal ideal $\fkM = \fkm G + G_+ $ of $G$. We set  $c= \rma (G)+1$ and $K=\rmK_R$. Then by  \cite[Theorem 1.1]{GI} we have a unique family $\omega =\{\omega_n \}_{n \in \Bbb Z}$ of $R$-submodules of $K$ satisfying the following four conditions:

\begin{enumerate}
\item[$(\mathrm i)$] $\omega_n \supseteq \omega_{n+1}$ for all $n \in \Bbb Z$,
\item[$(\mathrm{ii})$]  $\omega_n = K$ for all $n \le -c$, 
\item[$(\mathrm{iii})$]  $I_m \omega_n \subseteq \omega_{m+n}$ for all $m, n \in \Bbb Z$, and
\item[$(\mathrm{iv})$]  $\rmK_{\mathcal R'} \cong {\mathcal R'}(\omega)$ and $\rmK_G \cong G(\omega)(-1)$ as  graded $\mathcal R'$-modules,
\end{enumerate}
where  $\mathcal R' (\omega) = \sum_{n \in \Bbb Z}{\omega}_nt^n \subset K[t,t^{-1}]$ and $G(\omega) = \mathcal R'(\omega)/t^{-1}\mathcal R'(\omega)$, and $\rmK_{\mathcal R'}$ and $\rmK_G$ denote respectively the graded canonical modules of $\mathcal R'$ and $G$. Notice that $[G(\omega)]_n = (0)$ if $n < -c$ (see condition (ii)).

With this notation we have the following.

\begin{lem}\label{3}
There exist integers $d > 0$ and $k \ge 0$ such that $\omega_{dn-c} = I_d^{n-k}\omega_{dk-c}$ for all $n \ge k$.
\end{lem}

\begin{proof} Let  $L = \mathcal R(\omega)(-c)$, where $\mathcal R (\omega)= \sum_{n \ge 0}\omega_nt^n \subseteq K[t]$. Then $L$ is a finitely generated graded $\mathcal R$-module such that $L_n = (0)$ for $n < 0$. We  choose an integer $d \gg 0$ so that the Veronesean subring $\mathcal R^{(d)} = \sum_{n \ge 0}\mathcal R_{dn}$ of $\mathcal R$ with order $d$ is standard, whence $\mathcal R^{(d)} = R[\mathcal R_d]$. Then, because $L^{(d)} = \sum_{n \ge 0}L_{dn}$ is a finitely generated graded $\mathcal R^{(d)}$-module, we may choose a homogeneous system $\{f_i\}_{1 \le i \le \ell}$ of generators of $L^{(d)}$ so that for each $1 \le i \le \ell$ 
$$f_i \in [L^{(d)}]_{k_i} = [\mathcal R (\omega)]_{dk_i-c}$$ with $k_i \ge \frac{c}{d}$. Setting $k = \max\{k_i \mid 1 \le i \le \ell \}$, for all $n \ge k$ we get $$\omega_{dn-c} \subseteq \sum_{i=1}^\ell I_{d(n-k_i)}\omega_{dk_i - c} \subseteq I_d^{n-k}\omega_{dk-c},$$
as asserted.
\end{proof}

Let us fix an element $f \in K$ and let $\xi= \overline{ft^{-c}} \in G(\omega)(-c)$ denote the image of $ft^{-c} \in \mathcal R'(\omega)$ in $G(\omega)$. Assume $(0):_G \xi = (0)$ and consider  the following short exact sequence 
$$(E) \ \ \ 
0 \to G \overset{\psi}{\to} G(\omega)(-c) \to C \to 0,
$$
of graded $G$-modules, 
where $\psi(1)= \xi$. Then $C_\fkp=(0)$ for all $\fkp \in Ass G$, because  $[G(\omega)]_\fkp \cong [\rmK_G]_\fkp \cong \rmK_{G_\fkp}$ as $G_\fkp$-modules by condition (iv) above and $\ell_{G_\fkp}(G_{\fkp}) = \ell_{G_\fkp}(\rmK_{G_\fkp})$ (\cite[Korollar 6.4]{HK}). Therefore  $\ell_R(C)=\ell_G(C) < \infty$ since $\dim G = 1$, so that $C$ is finitely graded. We now consider  the  exact sequence
$$ 
{\mathcal R'} \overset{\varphi}{\to} {\mathcal R'}(\omega)(-c) \to D \to 0
$$
 of graded $\mathcal R'$-modules defined by $\varphi(1) = ft^{-c}$. Then  $C \cong D/uD$ as a $G$-module, where $u = t^{-1}$. Notice that $\dim \mathcal R'/\fkp = 2$ for all $\fkp \in \Ass \mathcal R'$, because $\mathcal R'$ is a Cohen-Macaulay ring of dimension $2$. We then have $D_\fkp = (0)$ for all $\fkp \in \Ass \mathcal R'$, since $\dim_{\mathcal R'}D \le 1$. Hence the homomorphism $\varphi$ is injective, because $\mathcal R'(\omega) \cong \rmK_{\mathcal R'}$ by condition (iv)  and $\ell_{\mathcal R'_\fkp}([\mathcal R'(\omega)]_\fkp) = \ell_{\mathcal R'\fkp}([\rmK_{\mathcal R'}]_\fkp)= \ell_{\mathcal R'\fkp}(\rmK_{\mathcal R'_\fkp}) = \ell_{\mathcal R'_\fkp}(\mathcal R'_\fkp)$ for all $\fkp \in \Ass \mathcal R'$. The snake lemma shows $u$ acts on $D$ as a non-zerodivisor, since $u$ acts on $\mathcal R'(\omega)$ as a non-zerodivisor.

Let us  suppose that $C \neq (0)$ and set $S=\{ n \in \Bbb Z \mid C_n \neq (0)\}$. We write $S=\{ n_1 < n_2 < \cdots < n_{\ell} \}$, where $\ell = \sharp S > 0$. We then have the following.
\medskip

\begin{lem}\label{4}
$D_n = (0)$ if $n > n_{\ell}$ and $D_n \cong K/Rf$ if $n \le 0$. Consequently, $\ell_R(K/Rf) = \ell_R(C)$.
\end{lem}

\begin{proof}
Let $n > n_{\ell}$. Then  $C_n = (0)$. By exact sequence $(E)$ above, we get $I_n/I_{n+1} \cong \omega_{n-c}/\omega_{n+1-c}$, whence $\omega_{n-c} = I_n f + \omega_{n+1-c}$. Therefore $\omega_n-c \subseteq I_n f + \omega_q$ for all $q \in \Bbb Z$. By Lemma \ref{3} we may choose integers $d\gg 0$ and $k \ge 0$ so that $$\omega_{n-c} \subseteq  I_nf  + \omega_{dm-c}\subseteq  I_n f + I_d^{m-k}f$$ for all $m \ge k$. Consequently, $\omega_{n-c}=I_n f$. Hence $D_n = (0)$ for all $n \ge n_{\ell}$.  If $n \le 0$, then $D_n \cong [\mathcal R'(\omega)(-c)]_n/\mathcal R'_nf \cong  K/Rf$ (see condition (ii) above). 
To see the last assertion, notice that because  $S=\{ n_1 < n_2 < \cdots < n_{\ell} \}$, $D_n =u^{n-n_1}D_{n_1} \cong D_{n_1}$ if $n \le n_1$ and  $D_n = u^{n_{i+1}-n}D_{n_{i+1}}\cong D_{n_{i+1}}$ if $1 \le i < \ell$ and  $n_{i} < n \le n_{i+1}.$ Therefore since $D_n = (0)$ for $n > n_\ell$, we get   
$$
\ell_R(K/Rf) =  \ell_R(D_0)=\ell_R(D_{n_1})=\sum_{i=1}^\ell \ell_R(C_{n_i})= \ell_R(C).
$$
\end{proof}

Exact sequence $(E)$ above now shows the following estimations. Remember that $\rmr(R) \le \rmr(G)$, because $\rmK_G[t]=K[t,t^{-1}]$ so that $\mu_R(K) \le \mu_G(\rmK_G)$.

\begin{prop}\label{5}
$\rmr(R)-1 \le \rmr(G)-1 \le \mu_G(C) \le \ell_G(C) = \ell_R(C) = \ell_R(K/Rf)$.
\end{prop}

We are now back to a general situation of Setting \ref{2}.  

\begin{thm}\label{6} 
Let $G$ be as in Setting $\ref{2}$ and assume that $G$ is an almost Gorenstein graded ring with $\rmr(G) = \rmr(R)$. Then  $R$ is an almost Gorenstein local ring.
\end{thm}

\begin{proof} We may assume  that $G$ is not a Gorenstein ring. We choose an exact sequence
$$
0 \to G \overset{\psi}{\to} G(\omega)(-c) \to C \to 0
$$
of graded $G$-modules so that $C \ne (0)$ and $\fkM C = (0)$. Then  $\mu_G(C) = \ell_G(C)$. We set $\xi = \psi(1)$ and write $\xi = \overline{ft^{-c}}$ with $f \in K$. Hence $(0):_G \xi = (0)$. We now look at the estimations stated in Proposition \ref{5}. If $\rmr (R) - 1 =\ell_G(C)$, then $\ell_R(K/Rf) = \mu_R(K/Rf)$ because $\rmr(R) - 1 = \mu_R(K) - 1 \le \mu_R(K/Rf) \le \ell_R(K/Rf) = \ell_G(C)$, so that $\fkm{\cdot}(K/Rf) = (0)$. Consequently, we get the exact sequence
$$0 \to R \overset{\varphi}{\to} K \to K/Rf \to 0$$
of $R$-modules with $\varphi (1) = f$, 
whence $R$ is an almost Gorenstein local ring. If
$\rmr(R)-1 < \ell_G(C)$, then $\psi(1) \in \fkM{\cdot}[G(\omega)(-c)]$ because $\rmr(G) -1<\mu_G(C)$, so  that $G_\fkM$ is a discrete valuation ring. This is impossible, since $G$ is not a Gorenstein ring.
\end{proof}

The converse of Theorem \ref{6} is also true when $G$ satisfies some additional conditions. To see this, we need the following. 
Recall that our graded ring $G$ is said to be level, if $\K_G = G{\cdot}[\K_G]_{-a}$, where $a = \a(G)$.  Let $\widehat{R}$ denote the $\fkm$-adic completion of $R$.

\begin{lem}\label{7}
Suppose that $\rmQ(\widehat{R})$ is a Gorenstein ring and the field $R/\fkm$ is infinite. Let us choose a canonical ideal $K$ of $R$ so that $R \subseteq K \subseteq \overline{R}$. Let $a \in \fkm$ be a regular element of $R$ such that $I = a K \subsetneq R$. We then have the following.
\begin{enumerate}
\item[$(1)$]  Suppose that $G$ is an integral domain.  Then there is an element  $f \in K \setminus \omega_{1-c}$ so that $af \in I$ generates a minimal reduction of $I$. Hence $(0):_G\xi = (0)$, where $\xi = \overline{ft^{-c}} \in G(\omega)(-c)$.
\item[$(2)$] Suppose that $\rmQ(G)$ is a Gorenstein ring and $G$ is a level ring. Then there is an element $f \in K$ such that $af \in I$ generates a reduction of $I$ and $G_\fkp{\cdot}\frac{\xi}{1} = [G(\omega)(-c)]_\fkp \cong G_\fkp$ for all $\fkp \in \Ass G$, where $\xi = \overline{ft^{-c}} \in G(\omega)(-c)$. Hence $(0):_G\xi = (0)$.
\end{enumerate}
\end{lem}

\begin{proof}
(1) Let $L = \omega_{1-c}$. Then $aL \subsetneq aK =I$, since $L \subsetneq K$. We  write $I = (x_1, x_2, \ldots, x_n)$ such that each $x_i$ generates a minimal reduction of $I$.  Choose $f= x_i$ so that $x_i \not\in L$, which is the  required one.

(2) Let $M = G(\omega)(-c)$. Then since $M = G{\cdot}M_0$ and $M_\fkp \ne (0)$, $M_0 \not\subseteq \fkp M_\fkp \cap M$ for any $\fkp \in \Ass G$. Choose an element $f \in K$ so that $af$ generates a reduction of $I$ and $\xi =\overline{ft^{-c}} \not\in \fkp M_\fkp \cap M$ for any $\fkp \in \Ass G$. Then $G_\fkp {\cdot}\frac{\xi}{1}= M_\fkp$ for all $\fkp \in \Ass G$, because $M_\fkp \cong G_\fkp$. 
\end{proof}

\begin{thm}\label{8}
Suppose that $R$ is an almost Gorenstein local ring and the field $R/\fkm$ is infinite. Assume that one of the following conditions is satisfied${\mathrm :}$ 
\begin{enumerate}
\item[$(1)$] $G$ is an integral domain;
\item[$(2)$] $\rmQ(G)$ is a Gorenstein ring and $G$ is a level ring.
\end{enumerate}
Then $G$ is an almost Gorenstein graded ring with $\rmr(G) = \rmr(R)$.
\end{thm}

\begin{proof}
The ring $\rmQ(\widehat{R})$ is Gorenstein, since $R$ is an almost Gorenstein local ring. Let $K$ be a canonical ideal of $R$ such that $R \subseteq K \subseteq \overline{R}$. We choose an element $f \in K$ and $a \in \fkm$ as in Lemma \ref{7}. Then $\mu_R(K/Rf) = \rmr(R)-1$, since $f$ is a part of a minimal system of generators of $K$ (recall that $af$ generates a minimal reduction of $I = aK$). Therefore by Proposition \ref{5}, $\rmr(G) = \rmr(R)$ and $\fkM{\cdot}C = (0)$, whence $G$ is an almost Gorenstein graded ring. 
\end{proof}

Suppose that $(R, \fkm)$ is a complete local domain of dimension one and let $V = \overline{R}$. Hence $V$ is a discrete valuation ring. Let $\fkn$ be the maximal ideal of $V$ and set $I_n = \fkn^n \cap R$ for each $n \in \Bbb Z$. Then $\mathcal F = \{I_n\}_{n \in \Bbb Z}$ is a filtration of ideals of $R$. We have $I_1 = \fkm$ and $G(\mathcal F) ~(\subseteq \gr_\fkn(V) = \bigoplus_{n \ge 0}\fkn^n/\fkn^{n+1})$ is an integral domain. The ring $\mathcal R(\mathcal F)$ is Noetherian, since  $\fkn^n \subseteq R$ for all $n \gg 0$. Therefore, applying Theorems \ref{6} and  \ref{8} to this setting, we readily get the following, where the implication $(2) \Rightarrow (1)$ is not true in general, unless $\rmr(G) = \rmr(R)$ (see Example \ref{p7.5}).

\begin{cor}[cf. {\cite[Proposition 29]{BF}}] Let $R$ and $\mathcal F$ be as above and consider the following conditions.
\begin{enumerate}
\item[$(1)$] $R$ is an almost Gorenstein local ring;
\item[$(2)$] $G$ is an almost Gorenstein graded ring and $\rmr(G) = \rmr(R)$.
\end{enumerate}
Then the implication $(2) \Rightarrow (1)$ holds true.
If the field $R/\fkm$ is infinite, the converse is also true. 

\end{cor}

To prove Theorem \ref{1}, we need one more result.

\begin{prop}\label{10}
Let $G = G_0[G_1]$ be a Noetherian standard graded ring. Assume that $G_0$ is an Artinian local ring with infinite residue class field. If $G$ is an almost Gorenstein graded ring with $\dim G \ge 2$, then $G/(x)$ is an almost Gorenstein graded ring for some non-zerodivisor $x \in G_1$. 
\end{prop}

\begin{proof} 
We may assume that $G$ is not a Gorenstein ring. Let $\fkm$ be the maximal ideal of $G_0$ and set $\fkM = \fkm G + G_+$. We consider the sequence
$$
0 \to G \to \rmK_G(-a) \to C \to 0
$$
of graded $G$-modules
such that $\mu_G(C) = \e^0_{\mathfrak{M}}(C)$, where $a = \rma(G)$ is the $\rma$-invariant of $G$. Then because the field $G_0/\mathfrak{m}$ is infinite and the ideal $G_{+}=(G_1)$ of $G$ is a reduction of $\mathfrak{M}$, we may choose an element $x \in G_1$ so that $x$ is $G$-regular and superficial for $C$ with respect to $\mathfrak{M}$. We set $\overline{G} = G/(x)$ and remember that  $x$ is $C$-regular, as $\dim_GC = \dim G - 1 > 0$. We then have the exact sequence
$$
0 \to G/(x) \to ({\rmK_G}/x{\rmK_G})(-a) \to C/xC \to 0
$$
of graded $\overline{G}$-modules. We now notice that $\rma (\overline{G}) = a+1$ and that 
$$
({\rmK_G}/x{\rmK_G})(-a) \cong \rmK_{\overline{G}}(-(a + 1))
$$
as a graded $\overline{G}$-module, while we see 
$$
\e^0_{{\fkM}/(x)}(C/xC) = \e^0_{\fkM}(C) = \mu_G(C) = \mu_G(C/xC),
$$ since $x$ is superficial for $C$ with respect to $\fkM$. 
Thus $\overline{G}$ is an almost Gorenstein graded ring.
\end{proof}

We are now ready to prove Theorem \ref{1}.

\begin{proof}[Proof of Theorem $\ref{1}$]
We set $d = \dim R$ and $G = \operatorname{gr}_I(R)$.
We may assume that $G$ is not a Gorenstein ring. Hence $d = \dim G \ge 1$. By Theorem \ref{6} we may also assume that $d > 1$ and that our assertion holds true for $d-1$. Let us consider an exact sequence
$$
0 \to G \to \rmK_G(-a) \to C \to 0
$$
of graded $G$-modules  with $\mu_G(C) = \e^0_{\fkM}(C)$, where $\fkM = \fkm G+G_+$ and $a= \rma(G)$. We choose  an element $a \in I$ so that the initial form $a^* = a + I^2 \in G_1=I/I^2$ of $a$ is $G$-regular and $G/a^*G = \operatorname{gr}_{I/(a)}(R/(a))$ is an almost Gorenstein graded ring (this choice is possible; see  Proposition \ref {10}). Then the hypothesis on $d$ shows $R/(a)$ is an almost Gorenstein local ring. Therefore $R$ is an almost Gorenstein local ring, because $a$ is $R$-regular. 
\end{proof}

In general, the local rings $R_\fkp$~$(\fkp \in \Spec R)$ of an almost Gorenstein local ring $R$ are not necessarily almost Gorenstein, as we will show by Example \ref{p7.5}. To do this, we assume that $R$ is a Cohen-Macaulay local ring of dimension $d \ge 0$, possessing the canonical module $\rmK_R$.  Let $\mathcal{F}=\{ I_n\}_{n \in \Bbb Z}$ be a filtration of 
ideals of $R$ such that $I_0 = R$ but $I_1 \neq R$. Smilarly as Setting \ref{2}, we consider the $R$-algebras
$${\mathcal R} = \sum_{n \ge 0}I_nt^n \subseteq R[t],\ \ {\mathcal R}' = \sum_{n \in \Bbb Z}I_nt^n \subseteq R[t,t^{-1}], \ \ \text{and}\ \ G = {\mathcal R}'/t^{-1}{\mathcal R}'$$  associated to $\mathcal{F}$, where $t$ is an indeterminate. Notice that $\mathcal{R}' = \mathcal{R}[t^{-1}]$ and that $G = \oplus_{n \ge 0}I_n/I_{n+1}$. Let $\fkN$ denote the graded maximal ideal of ${\mathcal R}'$. We then have the following.

\begin{thm}\label{p7.4} Suppose that $R/\fkm$ is infinite and that ${\mathcal R}$  is a Noetherian ring. If $G_{\fkN}$ is a pseudo-Gorenstein local ring, then $R$ is pseudo-Gorenstein.
\end{thm}

\begin{proof} By Theorem \ref{3.9} (1) ${\mathcal R'}_{\fkN}$ is an almost Gorenstein ring  with $\rmr(G_{\fkN}) = \rmr({\mathcal R'}_{\fkN}) \le 2$, as $G_{\fkN} = {\mathcal R'}_{\fkN}/t^{-1}{\mathcal R'}_{\fkN}$ and $t^{-1}$ is $\mathcal{R}'$-regular. Let $\fkp = \fkm {\cdot}R[t, t^{-1}]$ and set $P = \fkp \cap {\mathcal R'}$. Then  $P \subseteq \fkN$, so that by Proposition  \ref{7.0} $R[t,t^{-1}]_{\fkp}$ is an almost  Gorenstein ring, because $R[t,t^{-1}]_{\fkp} = {\mathcal R'}_P = ({\mathcal R'}_{\fkN})_{P{\mathcal R'}_{\fkN}}$. Thus by Theorem \ref{3.5} $R$ is an almost Gorenstein ring with $\rmr(R)\le 2$, because  $R/\fkm$ is infinite and the composite homomorphism $R \to R[t,t^{-1}] \to R[t,t^{-1}]_\fkp$ is flat.
\end{proof}

\begin{rem}\label{3}
The following example \ref{p7.5} is given by V. Barucci, D. E. Dobbs, and M. Fontana \cite[Example II.1.19]{BDF}. Because $R[t,t^{-1}]_{\fkp} = {\mathcal R'}_P = ({\mathcal R'}_{\fkN})_{P{\mathcal R'}_{\fkN}}$ with the notation of the proof of Theorem \ref{p7.4}, Theorem \ref{3.5} and Example \ref{p7.5} show that $({\mathcal R'}_{\fkN})_{P{\mathcal R'}_{\fkN}}$ is \textit{not} an almost Gorenstein local ring (here we assume the field $k$ is infinite). Hence, in general, local rings $R_\fkp~(\fkp \in \Spec R)$ of an almost Gorenstein local ring $R$ are not necessarily almost Gorenstein. Notice that $\mathcal{R}'_\fkN$ is \textit{not} a semi-Gorenstein local ring (remember that the local rings of a semi-Gorenstein local ring are semi-Gorenstein; see Proposition  \ref{7.0}). Therefore the example also shows that a local ring $R$ is \textit{not necessarily} semi-Gorenstein, even if $R/(f)$ is a semi-Gorenstein ring for some non-zerodivisor $f$ of $R$.

\begin{ex}\label{p7.5}
Let $k$ be a field with $\operatorname{ch} k \ne 2$ and let $R = k[[x^4, x^6+x^7, x^{10}]] \subseteq V$, where $V = k[[x]]$ denotes the formal power series ring over $k$. Then $V = \overline{R}$. Let $v$ denote the discrete valuation of $V$ and set $H=\{v(a) \mid 0 \ne a \in R\}$, the value semigroup of $R$. We consider the filtration  $\mathcal{F}= \{(xV)^n\cap R\}_{n \in \Bbb Z}$ of ideals of $R$ and  set $G=\mathcal{R}'/t^{-1}\mathcal{R}'$, where $\mathcal{R}'=\mathcal{R}'(\mathcal{F})$ is the extended Rees algebra of $\mathcal{F}$. We then have:
\begin{enumerate}
\item[$(1)$] $H = \left<4,6,11,13\right>$.
\item[$(2)$] $G \cong k[x^4,x^6,x^{11}, x^{13}] ~(\subseteq k[x])$ as a graded $k$-algebra and $G_\fkN$ is an almost Gorenstein local ring with $\rmr(G_\fkN) = 3$, where $\fkN$ is the graded maximal ideal of $\mathcal{R}'$.
\item[$(3)$] $R$ is not an almost Gorenstein local ring and $\rmr(R) = 2$. 
\end{enumerate} 
\end{ex}

\end{rem}

\section{Almost Gorenstein homogeneous rings}\label{homog}

In this section let $R = k[R_1]$ be a Cohen-Macaulay homogeneous ring over an infinite field $k=R_0$. We assume $d=\dim R>0$. Let $\fkM=R_+$ and $a = \a(R)$. For each finitely generated graded $R$-module $X$, let 
$
$[\![X]\!]$=\sum_{n=0}^{\infty}\dim_kX_n{\cdot}\lambda^n \ \in \ \Z[\lambda]$
be the Hilbert series of $X$, where $X_n$~($n\in\Z$) denotes the homogeneous component of $X$ with degree $n$. Then as it is well-known, writing $[\![R]\!]
 = \frac{F(\lambda)}{(1-\lambda)^d}$ with $F(\lambda)\in\Z[\lambda]$, we have $\deg F(\lambda) = a+d \geq 0$ and $[\![\K_R(-a)]\!]= \frac{F(\frac{1}{\lambda}){\cdot}\lambda^{a+d}}{(1-\lambda)^d}$.
Let $f_1,f_2,\ldots,f_d$ be a linear system of parameters of $R$. Then, because $\a(R/(f_1,f_2,\ldots,f_d)) = a+d$, we have $a=1-d$ if and only if $\fkM^2 = (f_1,f_2,\ldots,f_d)\fkM$ and $\fkM \ne (f_1,f_2,\ldots,f_d)$. Conversely, we get the following. Remember that the graded ring $R$ is said to be level, if $\K_R = R{\cdot}[\K_R]_{-a}$. 

\begin{prop}\label{7.1} Suppose that $a=1-d$. Then $R$ is a level ring with 
$[\![R[\!]= \frac{1+c\lambda}{(1-\lambda)^d}$ and 
$[\![\K_R(-a)[\!] = \frac{c+\lambda}{(1-\lambda)^d},$
where $c=\dim_k R_1 -d$.
\end{prop}

The following lemma shows that the converse of Proposition \ref{7.1}  is also true, if $R$ is an almost Gorenstein graded ring.

\begin{lem}\label{7.2}
Suppose that $R$ is an almost Gorenstein graded ring and assume that $R$ is not a Gorenstein ring. If $R$ is a level ring, then $a=1-d$.
\end{lem}

\begin{proof}
Suppose that $R$ is a level ring and take an exact sequence 
$0 \to R \to \K_R(-a) \to C \to 0$
of graded $R$-modules so that $\mu_R(C) = \e^0_\fkM(C)$. Then $C\ne(0)$ and hence $C$ is a Cohen-Macaulay $R$-module of dimension $d-1$ (Lemma \ref{3.1} (2)). We have $C = R C_0$ and $\mu_R(C) = \mathrm{r}(R) -1$ (Corollary \ref{3.8}).  Remember that $\fkM C=(f_2,f_3,\ldots,f_d)C$ for some $f_2,f_3,\ldots,f_d \in R_1$ (see Proposition \ref{2.2} (2)) and we have 
$[\![C]\!]= \frac{r-1}{(1-\lambda)^{d-1}},$
where $r=\mathrm{r}(R)$. Consequently, 
$[\![\K_R(-a)]\!] - [\![R]\!] = \frac{r-1}{(1-\lambda)^{d-1}},$
so that 
$F(\frac{1}{\lambda}){\cdot}\lambda^{a+d} - F(\lambda) = (r-1)(1-\lambda).$
Let us write $F(\lambda) = \sum_{i=0}^{a+d}c_i\lambda^i$ with $c_i\in\Z$. Then the equality 
$\sum_{i=0}^{a+d}c_{a+d-i}\lambda^i = \sum_{i=0}^{a+d}c_i\lambda^i + (r-1)(1-\lambda)$
forces that if $a+d\geq 2$, then $c_0\lambda^{a+d} = c_{a+d}\lambda^{a+d}$ and $c_{a+d} = c_0 +(r-1)$, which is impossible, since $c_0 = 1$ and $r>1$. Therefore we have $a+d=1$, because $0\leq a+d$ and $R$ is not a Gorenstein ring.
\end{proof}

The following is a key in our argument.

\begin{lem}\label{7.3}
Suppose that $R$ is a level ring and $\mathrm{Q}(R)$ is a Gorenstein ring. Then there exists an exact sequence 
$0 \to R \to \K_R(-a) \to C \to 0$
of graded $R$-modules with $\dim_R C <d$.
\end{lem}

\begin{proof}
Let $V = [\K_R]_{-a}$. Then $\K_R=R V$. For each $\p \in \Ass R$, let $L(\p)$ be the kernel of the composite of two canonical homomorphism $h(\p): \K_R \to (\K_R)_{\p} \to (\K_R)_{\p}/\p(\K_R)_{\p}$. Then $V \not\subseteq L(\p)$, because $\K_R = R V$ and $(\K_R)_{\p} \cong R_{\p}$. Let us choose $\xi \in V \setminus \bigcup_{\p\in\Ass R}L(\p)$ and let $\varphi: R \to \K_R(-a)$ be the homomorphism of graded $R$-modules defined by $\varphi(1)=\xi$. Look now at the exact sequence 
$R \overset{\varphi}{\to} \K_R(-a) \to C \to 0$
of graded $R$-modules. Then, because $(\K_R)_{\p}=R_{\p}\frac{\xi}{1}$ for all $\p \in \Ass R$ (remember that $(\K_R)_{\p} \cong R_{\p}$), we have $C_{\p} = (0)$ for all $\p \in \Ass R$. Hence $\dim_R C < d$ and therefore $\varphi$ is injective (see Lemma \ref{3.1} (1)).
\end{proof}

We now come to the main result of this section.

\begin{thm}\label{7.4}
Suppose that $\mathrm{Q}(R)$ is a Gorenstein ring. If $a=1-d$, then $R$ is an almost Gorenstein graded ring.
\end{thm}

\begin{proof}
We may assume that $R$ is not a Gorenstein ring. Thanks to Lemma \ref{7.3}, we can choose an exact sequence 
$0 \to R \to \K_R(-a) \to C \to 0$
of graded $R$-modules so that $C$ is a Cohen-Macaulay $R$-module of dimension $d-1$ (see Lemma \ref{3.1} (2) also). 
Then Proposition \ref{7.1} implies $[\![C]\!]=[\![\K_R(-a)]\!]- [\![R]\!] = \frac{c-1}{(1-\lambda)^{d-1}}$, where $c=\dim_k R_1 - d$.
Let $f_1,f_2,\ldots,f_{d-1}~(\in R_1)$ be a linear system of parameters for $C$. 
Then because $[\![C/(f_2,\ldots,f_{d})C]\!] = (1-\lambda)^{d-1}[\![C]\!]= c-1$, we get 
$C/(f_1,f_2,\ldots,f_{d-1})C = [C/(f_1,f_2,\ldots,f_{d-1})C]_0,$
which shows that $\fkM C = (f_1,f_2,\ldots,f_{d-1})C$. Thus $\mu_R C = \e^0_{\fkM}(C)$ and hence $R$ is an almost Gorenstein graded ring.
\end{proof}

\begin{ex}\label{7.5}
Let $S = k[X_1, X_2, \ldots , X_n, Y_1, Y_2, \ldots, Y_n]$ ($n\geq 2$) be the polynomial ring over an infinite field $k$ and let $R = S/\rmI_2
\left(\begin{smallmatrix}   
X_1 & X_2 & \ldots &X_n \\
Y_1 & Y_2 & \ldots &Y_n
\end{smallmatrix}\right)$. Then $R$ is a Cohen-Macaulay normal ring with $\dim R = n+1$. We have $\a(R) = 1-\dim R$, because $\fkM^2 = (X_1, \{ X_{i+1}-Y_i \}_{1\leq i\leq n-1}, Y_n)\fkM$. Hence by Theorem \ref{7.4}, $R$ is an almost Gorenstein graded ring. 
\end{ex}

We explore almost Gorenstein Veronesean subrings.

\begin{cor}\label{7.6}
Suppose that $d=2$, R is reduced, and $\a(R)<0$. Then the Veronesean subrings $R^{(n)} = k[R_n]$ of $R$ are almost Gorenstein graded rings for all $n\geq 1$.
\end{cor}

\begin{proof}
Let $S=R^{(n)}$. Then $S=k[S_1]$ is a Cohen-Macaulay reduced ring and $\operatorname{dim} S =2$. Since $\H^2_{\fkN}(S) = [\H^2_{\fkM}(R)]^{(n)}$ (here $\fkN=S_+$), we get $\a(S)<0$. Hence it suffices to show that $R$ is an almost Gorenstein graded ring, which readily follows from Theorem \ref{7.5}, because $\mathrm{Q}(R)$ is a Gorenstein ring and $\a(R) \leq 1-\dim R$ (remember that $R$ is the polynomial ring, if $\a(R) = -\dim R$).
\end{proof}

Let us explore a few concrete examples.

\begin{ex}\label{7.7}
Let $R=k[X, Y, Z]/(Z^2-XY)$, where $k[X, Y, Z]$ is the polynomial ring over an infinite field $k$. Then $R^{(n)}$ is an almost Gorenstein graded ring for all $n\geq 1$
\end{ex}

\begin{proof}
The assertion follows from Corollary \ref{7.6}, since $R$ is normal with $\operatorname{dim} R = 2$ and $\rma (R) = -1$.
\end{proof}

\begin{ex}\label{7.8}
Let $R=k[X_1, X_2,\ldots, X_d]$ ($d\geq 1$) be the polynomial ring over an infinite field $k$. Let $n\geq 1$ be an integer and look at the Veronesean subring $R^{(n)} = k[R_n]$ of $R$. Then the following hold.
\begin{enumerate}[(1)]
\item
$R^{(n)}$ is an almost Gorenstein graded ring, if $d\leq 2$.
\item
Suppose that $d\geq 3$. Then $R^{(n)}$ is an almost Gorenstein graded ring if and only if either $n \mid d$, or $d=3$ and $n=2$.
\end{enumerate}  
\end{ex}

\begin{proof}
Assertion (1) follows from Corollary \ref{7.6}. Suppose that $d\geq 3$ and consider assertion (2). The ring $R^{(n)}$ is a Gorenstein ring if and only if $n \mid d$ (\cite{Mat}). Assume that $n \nmid d$ and put $S=R^{(n)}$. If $d=3$ and $n=2$, then $S = k[X_1^2, X_2^2, X_3^2, X_1X_2, X_2X_3, X_3X_1]$ with $[S_+]^2 = (X_1^2, X_2^2, X_3^2)S_+$. Hence $S$ is an almost Gorenstein graded ring by Theorem \ref{7.4}. Conversely, suppose that $S$ is an almost Gorenstein graded ring. Let $L=\{(\alpha_1,\alpha_2,\ldots,\alpha_d) \mid 0 \leq \alpha_i \in \Z\}$. We put $|\alpha| = \sum_{i=1}^{d} \alpha_i$ and  $X^\alpha = \prod_{i=1}^dX_i^{\alpha_i}$ for each $\alpha = (\alpha_1,\alpha_2,\ldots ,\alpha_d)\in L$. Let $s= \operatorname{min} \{s\in\Z \mid s n\geq q\}$ where $q=(n-1)(d-1)$ and put $\fka = (X_1^n, X_2^n,\ldots ,X_d^n)$. We then have
$$\fka:S_+ = \fka + (X^{\alpha}X^{\beta} \mid \alpha,~\beta\in L \text{ such that } |\alpha|=q,~|\beta| = sn-q),$$
which shows the homogeneous Cohen-Macaulay ring $S$ is a level ring with $\a(S) = 1-d$. Therefore, because $\H^d_{S_+}(S) \cong [\H^d_{R_+}(R)]^{(n)}$ and because $[\H^d_{R_+}(R)]_i\ne(0)$ if and only if $i\leq -d$, we have $-n(d-1)<-d<-n(d-2)$, which forces $n=2$ and $d=3$, because $n\geq 2$ (remember that $n \nmid d$).
\end{proof} 

\begin{ex}\label{7.9}
Let $n \geq 1$ be an integer and let $\Delta$ be a simplicial complex with vertex set $[n] = \{1,2,\ldots , n\}$. Let $R = k[\Delta]$ denote the Stanley-Reisner ring of $\Delta$ over an infinite field $k$. Then $\e_{R_+}^0(R)$ is equal to the number of facets of $\Delta$. If $R$ is Cohen-Macaulay and $n = \e_{R_+}^0(R) + \operatorname{dim} R -1$, then $R$ is an almost Gorenstein graded ring. For example, look at the simplicial complex 
$\Delta: $
\begin{center}
\includegraphics[width=5cm,clip]{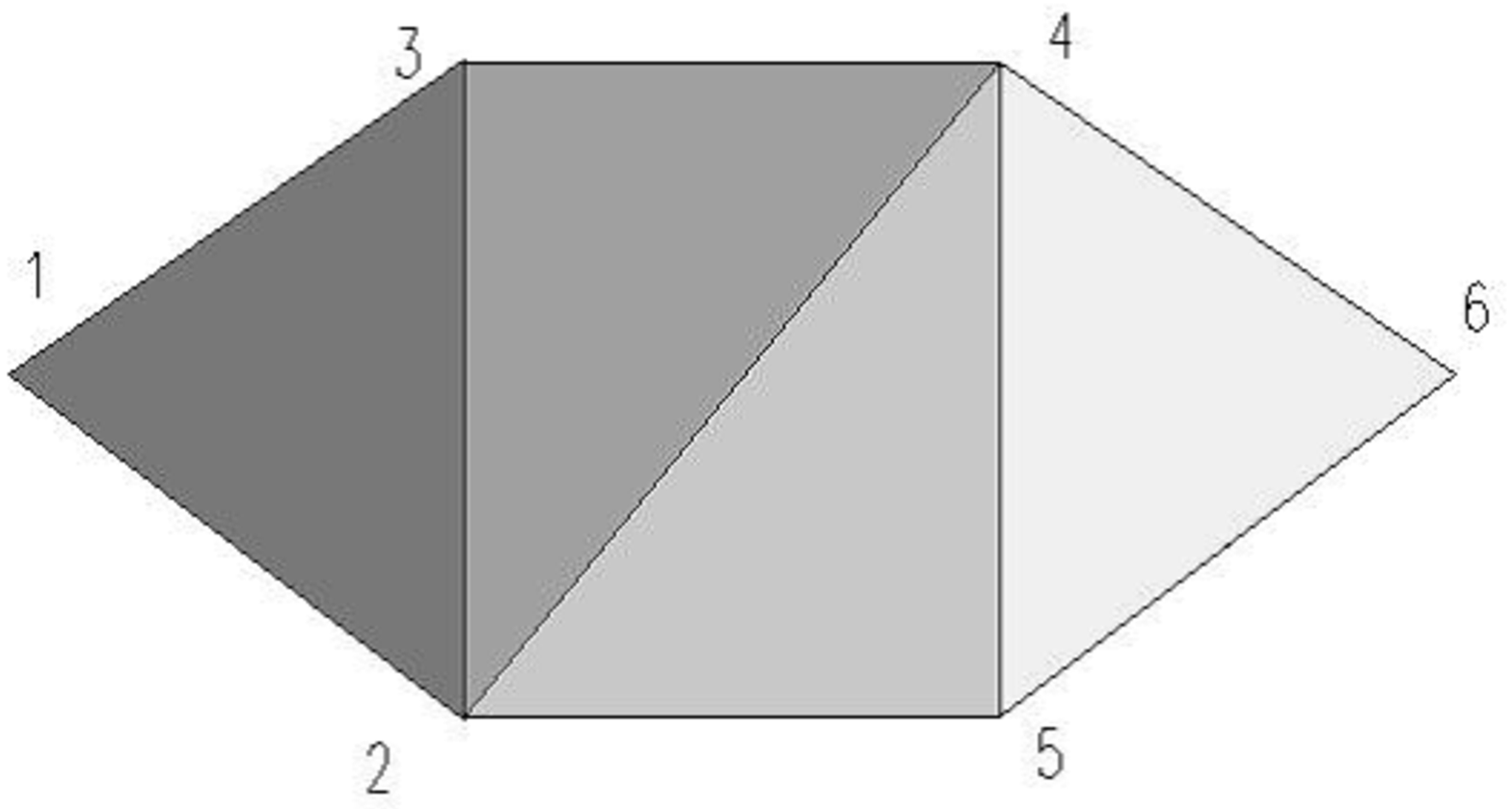}
\end{center}
with $n = 6$. Then $R = k[\Delta]$ is an almost Gorenstein graded ring of dimension $3$. Note that $R_P$ is not a Gorenstein local ring for $P = (X_1, X_3, X_4, X_5, X_6)R$.
\end{ex}


\section{Two-dimensional rational singularities are almost Gorenstein}\label{tancon}

Throughout this section let $(R,\m)$ denote a Cohen-Macaulay local ring of dimension $d \geq 0$, admitting  the canonical module $\rmK_R$. We assume that $R/\m$ is infinite. Let $\rmv(R)=\mu_R(\m)$ and $\e(R)=\e_{\m}^0(R)$. We set $G=\gr_{\m}(R)=\bigoplus{\m}^n/{\m}^{n+1}$ be the associated graded ring of $\m$ and put $\fkM=G_+$.
The purpose of this section is to study the question of when $G$ is an almost Gorenstein graded ring, in connection with the almost Gorensteinness of the base local ring $R$. Remember that, thanks to \cite{S1}, $\rmv (R) = \rme (R) + d - 1$ if and only if $\fkm^2 = Q\fkm$ for some (and hence any) minimal reduction $Q$ of $\fkm$. When this is the case, $G$ is a Cohen-Macaulay ring and $\rma (G) = 1 - d$, provided $R$ is not a regular local ring.

Our result is stated as follows.

\begin{thm}\label{8.1}
The following assertions hold true.
\begin{enumerate}[\rm(1)]
\item
Suppose that $R$ is an almost Gorenstein local ring with $\rmv(R)=\e(R)+d-1$.
Then $G$ is an almost Gorenstein level graded ring.
\item
Suppose that $G$ is an almost Gorenstein level graded ring. Then $R$ is an almost Gorenstein local ring.
\end{enumerate}
\end{thm}

\begin{proof}
(1) If $R$ is a Gorenstein ring, then $\e(R) \leq 2$ and $R$ is an abstract hypersurface, since  $\m^2=Q\m$ for some minimal reduction $Q$ of $\m$.
Therefore $G$ is also a hypersurface and hence a Gorenstein ring.

Assume that $R$ is not a Gorenstein local ring.
Hence $d > 0$ and $\rma(G)=1-d$.
We show that $G$ is an almost Gorenstein graded ring by induction on $d$.
First we consider the case $d=1$.
Let $\overline{R}$ denote the integral closure of $R$ in $\rmQ(R)$.
Choose an $R$-submodule $K$ of $\overline{R}$ so that $R \subseteq K \subseteq \overline{R}$ and $K \cong \rmK_R$ as an $R$-module (this choice is possible; see \cite[Corollary 2.8]{GMP}).
We have $\m K \subseteq R$ by \cite[Theorem 3.11]{GMP} as $R$ is an almost Gorenstein local ring.
Hence $\m K = \m$ (see Corollary \ref{3.8}), and $\m^n K=\m^n$ for all $n \geq 1$.
Let $C=K/R$ and consider the $\m$-adic filtrations of $R$, $K$, and $C$.
We then have the exact sequence
\begin{equation}
0 \to G \to \gr_{\m}(K) \to \gr_{\m}(C) \to 0 \tag{$\sharp$}
\end{equation}
of graded $G$-modules induced from the canonical exact sequence
$ 0 \to R \to K \to C \to 0$
of filtered $R$-modules.
Notice that $\gr_{\m}(C)=[\gr_{\m}(C)]_0$, since $\m C=(0)$. Thanks to exact sequence $(\sharp)$ above, the following claim yields  that $G$ is an almost Gorenstein graded ring.

\begin{claim}\label{8.2}
$\gr_{\m}(K) \cong \rmK_G$ as a graded $G$-module.
\end{claim}

\begin{proof}[Proof of Claim {\rm \ref{8.2}}]
Since $\rma(G)=0$ and $G$ is a graded submodule of $\gr_{\m}(K)$, it suffices to show that $\depth_G \gr_{\m}(K) > 0$ and $\rmr_G(\gr_{\m}(K))=1$.
Choose $a \in \m$ so that $\m^2=a\m$ and let $f=\overline{a} ~(\in \m/\m^2)$ denotes the image of $a$ in $G$.
Let $x \in \m^n K$ $(n \geq 0)$ and assume that $ax \in \m^{n+2} K=\m^{n+2}$.
Then, since $\m^{n+2}=a\m^{n+1}$, we readily get $x \in \m^{n+1}=\m^{n+1} K$.
Thus $f$ is $\gr_{\m}(K)$-regular, $[\gr_{\m}(K) / f \gr_{\m}(K)]_0 = K/\m$, $[\gr_{\m}(K) / f \gr_{\m}(K)]_1 = \m/aK$, and $\gr_{\m}(K) / f \gr_{\m}(K) = K/\m \oplus \m/aK$.
Let $x \in K \setminus \m$ and assume that $\fkM \overline{x} \subseteq f \gr_{\m}(K)$, where $\overline{x} ~(\in K/\m)$ denotes the image of $x$ in $\gr_{\m}(K)$.
Then, since $\m x \subseteq aK$ and $\rmr_R(K)=1$ (remember that $K \cong \K_R$ as an $R$-module), we get $aK:_K \m=aK+Rx$.
Therefore $\m (K/aK)=(0)$; otherwise $\m (K/aK) \supseteq [(0):_{K/aK}\m]$ and hence   $x \in \m K=\m$.
Consequently, because $K/aK=\rmE_{R/(a)}(R/\m)$ (the injective envelope of the $R/(a)$-module  $R/\fkm$; see \cite[Korollar 6.4]{HK}) is a faithful $R/(a)$-module, we have $\m=(a)$.
This is, however, impossible, because $R$ is not a discrete valuation ring.
Let $x \in \m \setminus aK$ and assume that $\fkM \overline{x} \subseteq f \gr_{\m}(K)$, where $\overline{x} ~(\in \m K/\m^2 K)$ denotes the image of $x$ in $\gr_{\m}(K)$.
Then $\m x \subseteq aK$ and hence $aK:_K \fkm=aK+Rx$, which proves $\ell_{R/\m}((0):_{{\gr_{\m}(K)}/f \gr_{\m}(K)}\fkM)=1$.
Thus $\rmr_G(\gr_{\m}(K))=1$, so that $\gr_{\m}(K) \cong \rmK_G$ as a graded $G$-module.
\end{proof}

Assume now that $d>1$ and that our assertion holds true for $d-1$.
Let $0 \to R \to \rmK_R \to C \to 0$ be an exact sequence of $R$-modules such that $\mu_R(C)=\e_{\m}^0(C)$.
Choose $a \in \m$ so that $a$ is a part of a minimal reduction of $\m$ and $a$ is superficial for $C$ with respect to $\m$.
Let $f=\overline{a} ~(\in \m/\m^2)$ denote the image of $a$ in $G=\gr_{\m}(R)$.
We then have $G/fG=\gr_{\m}(R/(a))$ and $\rmv(R/(a))=\e(R/(a))+d-2$.
By the hypothesis of induction, $G/fG$ is an almost Gorenstein graded ring, because $R/(a)$ is an almost Gorenstein local ring (see Proof of Theorem \ref{3.5}).
Choose an exact sequence
$ 0 \to G/fG \to \rmK_{G/fG}(d-2) \to X \to 0 $
of graded $G/fG$-modules so that $\mu_{G/fG}(X)=\e_{[G/fG]_+}^0(X)$.
Recall that $\rmK_{G/fG}(d-2) \cong \K_{G/f \rmK_G}(d-1)$ as a graded $G$-module and we get an exact sequence 
$ 0 \to G \to \rmK_G(d-1) \to Y \to 0 $
of graded $G$-modules, similarly as in Proof of Theorem \ref{3.9}, such that $\mu_G(Y)=\e_\fkM^0(Y)$.
Hence $G$ is an almost Gorenstein graded ring.

(2)
We may assume $G$ is not a Gorenstein ring.
Hence $d>0$ and $a = 1-d$ (Lemma \ref{7.2}).
Suppose that $d=1$ and choose an exact sequence
\begin{equation}
0 \to G \to \rmK_G(-a) \to C \to 0 \tag{$\sharp$}
\end{equation}
of graded $G$-modules so that $\fkM C=(0)$, where $a=\rma(G)$.

We now take the canonical $\m$-filtration $\omega=\{\omega_n\}_{n \in \Bbb Z}$ of $\rmK_R$ so that $\omega_n=\rmK_R$ if $n \leq -a$ and $\rmK_G=\gr_{\omega}(\rmK_R)$ (see, e.g.,  \cite{GI} for the existence of canonical filtrations).
Let $f \in \rmK_R=\omega_{-a}$ such that $\varphi(1)=\overline{f}$, where $\overline{f} ~(\in \rmK_R/\omega_{-a+1})$ denotes the image of $f$ in $\rmK_G(-a)$.
Let $\alpha$ be the $R$-linear map $R \to \rmK_R$ defined by $\alpha(1)=f$ and consider $R$ to be filtered with filtration ${\Bbb F}=\{\m^{n+a}\}_{n \in \Z}$.
Then the homomorphism $\alpha:R \to \rmK_R$ preserves filtrations and $G(a) \cong \gr_{{\Bbb F}}(R)$ as an graded $G$-module.
Consequently, exact sequence $(\sharp)$ turns into
$ 0 \to \gr_{{\Bbb F}}(R)(-a) \to \gr_{\omega}(\rmK_R)(-a) \to C \to 0 $
with  $\varphi(1)=\overline{f}$.
We now notice that $C=C_0$, because $G$ is a level ring.
Hence $\omega_n=\m^{n+a}f+\omega_{n+1}$ for $n > -a$.
Therefore
$ \omega_{-a+1} \subseteq \m f+\omega_{\ell} $
for all $\ell \in \Z$ and hence $\omega_{-a+1}=\m f$.
Thus $\omega_{-a+1}=\m f=\m \rmK_R$, because $\m \rmK_R=\m \omega_{-a} \subseteq \omega_{-a+1}$.
Consequently, in the exact sequence
$ R \overset{\alpha}{\to} \rmK_R \to X \to 0 $, we have $\m X=(0)$.
Thus $R$ is an almost Gorenstein local ring (see Lemma \ref{3.1}). 

Now suppose that $d>1$ and that our assertion holds true for $d-1$.
Look at the exact sequence
$ 0 \to G \to \rmK_G(d-1) \to C \to 0  $
of graded $G$-modules with $\mu_G(C)=\e_\fkM^0(C)$.
Choose $f \in G_1$ so that $f$ is $G$-regular and 
$\fkM C=(f,f_2,\ldots,f_{d-1})C$
for some $f_2,f_3,\ldots,f_{d-1} \in G_1$ (Proposition \ref{2.2} (2)).
We then have the exact sequence
$ 0 \to G/fG \to (\rmK_G/f \rmK_G)(d-1) \to C/fC \to 0 $
of graded $G/fG$-modules, which guarantees that $G/fG$ is an almost Gorenstein graded ring (remember that $(\rmK_G/f \rmK_G)(d-1) \cong \rmK_{G/fG}(d-2)$).
Consequently, thanks to the hypothesis of induction, the local ring $R/(a)$ (here $a \in \m$ such that $f=\overline{a}$ in $\m/\m^2 = [\gr_{\m}(R)]_1$) is an almost Gorenstein local ring and therefore by Theorem \ref{3.9}, $R$ is an almost Gorenstein local ring, because $a$ is $R$-regular.
\end{proof}

When $\rmv (R) = \e(R) + d - 1$, the almost Gorensteinness of $R$ is equivalent to the Gorensteinness of $\rmQ (G)$, as we show in the following.

\begin{cor}\label{8.3} Suppose that $\rmv(R)=\e(R)+d-1$.
Then the following are equivalent.
\begin{enumerate}[\rm(1)]
\item
$R$ is an almost Gorenstein local ring,
\item
$G$ is an almost Gorenstein graded ring,
\item
$\rmQ (G)$ is a Gorenstein ring.
\end{enumerate}
\end{cor}

\begin{proof}
Since $G$ is a Cohen-Macaulay level graded ring (Proposition \ref{7.1} and \cite{S1}), the equivalence $(1) \Leftrightarrow (2)$ follows from Theorem \ref{8.1}.
See Theorem \ref{7.4} (resp. Lemma \ref{3.1} (1)) for the implication $(3) \Rightarrow (2)$ (resp. $(2) \Rightarrow (3)$).
\end{proof}

We say that $\fkm$ is a normal ideal, if $\fkm^n$ is an integrally closed ideal for all $ n \ge 1$.

\begin{cor}\label{8.4} Suppose that $\rmv(R)=\e(R)+d-1$ and that $R$ is a normal ring. If $\m$ is a normal ideal, then $R$ is an almost Gorenstein local ring.
\end{cor}

\begin{proof}
Let $\calR'=\calR'(\m)=R[{\m}t, t^{-1}]$ be the extended Rees algebra of $\fkm$, where $t$ is an indeterminate.
Then $\calR'$ is a normal ring, because $R$ is a normal local ring and $\m$ is a normal ideal. Hence the total ring of fractions of $G=\calR'/t^{-1}\calR'$ is a Gorenstein ring, so that $R$ is almost Gorenstein by Corollary \ref{8.3}.
\end{proof}

We now reach the goal of this section.

\begin{cor}\label{8.5}
Every $2$-dimensional rational singularity is an almost Gorenstein local ring.
\end{cor}

Auslander's theorem \cite{A} says that every two-dimensional Cohen-Macaulay complete local ring $R$ of finite Cohen-Macaulay representation type is a rational singularity, provided $R$ contains a field of characteristic $0$. Hence by Corollary \ref{8.5} we get the following.

\begin{cor}\label{8.6} Every two-dimensional Cohen-Macaulay complete local ring $R$ of finite Cohen-Macaulay representation type is an almost Gorenstein local ring, provided $R$ contains a field of characteristic $0$. 
\end{cor}


\section{One-dimensional Cohen-Macaulay local rings of finite CM-representation type are almost Gorenstein}\label{fCM}
The purpose of this section is to prove the following.

\begin{thm}\label{12.1} 
Let $A$ be a Cohen-Macaulay complete local ring of dimension one and assume that the residue class field of $A$ is algebraically closed of characterisic $0$. 
If $A$ has finite Cohen-Macaulay representation type, then $A$ is an almost Gorenstein local ring.
\end{thm}

Let $A$ be a Cohen-Macaulay complete local ring of dimension one  with  algebraically closed residue class field $k$ of characterisic $0$. Assume that $A$ has finite Cohen-Macaulay representation type. Then by \cite[(9.2)]{Y} one obtains  a \textit{simple singularity} $R$ so that $R \subseteq A \subseteq \overline{R}$, where $\overline{R}$ denotes the integral closure of $R$ in the total ring $\rmQ (R)$ of fractions. We remember that 
$R=S/(F),$ where $S = k[[X,Y]]$ is the formal power series ring over $k$  and $F$ is one of the following polynomials (\cite[(8.5)]{Y}).
\begin{enumerate}
\item[$(\rmA_n)$\ ]
$X^2-Y^{n+1} \quad(n\ge1)$
\item[$(\rmD_n)$\ ]
$X^2Y-Y^{n-1} \quad (n\ge4)$
\item[$(\rmE_6)$\ ]
$X^3-Y^4$
\item[$(\rmE_7)$\ ]
$X^3-XY^3$
\item[$(\rmE_8)$\ ]
$X^3-Y^5$.
\end{enumerate}
Our purpose is, therefore, to show that all the intermediate local rings $R \subseteq A \subseteq \overline{R}$ are almost Gorenstein.

Let us begin with the analysis of overrings of local rings with multiplicity $2$. 

\begin{lem}\label{12.2}
Let $(R,\fkm)$ be a Cohen-Macaulay local ring with $\dim R = 1$ and $\e_\fkm^0(R) = 2$. Let $R \subseteq S \subseteq \rmQ(R)$ be an intermediate ring such that  $S$ is a finitely generated $R$-module. Then the following assertions hold true.\begin{enumerate}
\item[$(1)$] If $S$ is a local ring with maximal ideal $\fkn$, then $\e_\fkn^0(S) \le 2$.  We have $\e^0_\fkn(S) = 2$ and $R/\fkm \cong S/\fkn$, if $S \subsetneq  \overline{R}$.
\item[$(2)$] Suppose that $S$ is not  a local ring. Then $\widehat{R}\otimes_RS \cong V_1 \times V_2$ with discrete valuation rings $V_1$ and $V_2$, where $\widehat{R}$ denotes the $\fkm$-adic completion of $R$. Hence $S$ is a regular ring and $S = \overline{R}$.
\item[$(3)$] Let $A = R:\fkm$ in $\rmQ(R)$. Then $A = \fkm:\fkm$ and if $R \subsetneq S$, then $A \subseteq S$.
\end{enumerate}
\end{lem}

\begin{proof} (1), (2) We have $\fkm^2 = f\fkm$ for some $f \in \fkm$, since $\e^0_\fkm(R) = 2$ (see \cite[Theorem 3.4]{SV2}). Let $Q = (f)$. Then 
$$2 = \e_\fkm^0(R) = \e_0(Q,R) = \e_0(Q,S) = \ell_R(S/fS) \ge \ell_S(S/fS) \ge \mu_R(S),$$
because $\ell_R(S/R) < \infty$. Therefore assertion (1) follows, since $\ell_S(S/fS) \ge \e^0_\fkn(S)$. We have $R/\fkm = S/\fkn$, if $\e^0_\fkn(S) = 2$ (remember that $\ell_R(S/fS) = [S/\fkn : R/\fkm]{\cdot}\ell_S(S/fS)$).  Assume now that $S$ is not a local ring. To see assertion (2), passing to $\widehat{R}$, we may assume that $R$ is complete. Then $\mu_R(S)= 2$. Hence $S$ contains exactly two maximal ideals $\fkn_1$ and $\fkn_2$, so that $S \cong S_{\fkn_1} \times S_{\fkn_2}$ with $\mu_R(S_{\fkn_i}) = 1$ for $i = 1,2$. Because $\ell_R(S/fS) = 2$, we get $\ell_R(S_{\fkn_i}/fS_{\fkn_i}) = 1$, whence the local rings $S_{\fkn_i}$ are discrete valuation rings. Therefore $S$ is  regular.  

(3) Because $R$ is not a discrete valuation ring, we get $A = \fkm:\fkm$. We also have $\ell_R(A/R) = 1$ (\cite[Satz 1.46]{HK}), since $R$ is a Gorenstein ring. Therefore $A \subseteq S$, if $R \subsetneq S$.
\end{proof}

\begin{prop}\label{12.3}
Let $(R,\fkm)$ be a Cohen-Macaulay local ring with $\dim R = 1$ and $\e_\fkm^0(R) = 2$. Let $R \subsetneq S \subseteq \rmQ(R)$ be an intermediate ring such that  $S$ is a finitely generated $R$-module. Let $n = \ell_R(S/R)$.  We then have the following.
\begin{enumerate}
\item[$(1)$] There exists a unique chain of intermediate rings$$R = A_0 \subsetneq A_1 \subsetneq \ldots \subsetneq A_n = S.$$
\item[$(2)$] Every intermediate ring $R \subseteq A \subseteq S$ appears as one of $\{A_i\}_{0 \le i \le n}$.
\end{enumerate}
\end{prop}

\begin{proof}
Let $A_1 = R:\fkm$. Then by Lemma \ref{12.2} (3) $A_1 \subseteq A$ for every intermediate ring $R \subsetneq A \subseteq S$, which enables us to assume  that $n>1$ and that the assertion holds true for $n-1$. As $R \subsetneq A_1 \subsetneq S$, by Lemma \ref{12.2} $A_1$ is a local ring with $\e_{\fkn_1}^0(A_1) = 2$ and $R/\fkm \cong A_1/\n_1$, where $\fkn_1$ is the maximal ideal of $A_1$. Hence $\ell_{A_1}(S/A_1) = \ell_{R}(S/A_1)= n-1$, so that the assertion follows from the hypothesis of induction on $n$. 
\end{proof}

\begin{cor}\label{12.4}
Let $(R,\fkm)$ be a Cohen-Macaulay local ring with $\dim R = 1$ and $\e_\fkm^0(R) = 2$ and let $R \subseteq A, B \subseteq \rmQ(R)$ be intermediate rings such that $A$ and $B$ are finitely generated $R$-modules. Then $A \subseteq B$ or $B \subseteq A$. 
\end{cor}

\begin{proof}
Let $S = R[A,B]$. Then $R \subseteq S \subseteq \rmQ(R)$ and $S$ is a finitely generated $R$-module. Hence assertion follows from Proposition \ref{12.3}.
\end{proof}

For a Cohen-Macaulay local ring $R$ of dimension one, let $\calX_R$ denote the set of intermediate rings $R \subseteq A \subseteq \rmQ(R)$ such that $A$ is a finitely generated $R$-module.

\begin{cor}\label{12.5} The following assertions hold true.
\begin{enumerate}
\item[$(1)$] Let $V=k[[t]]$ be the formal power series ring over a field $k$ and  $R = k[[t^2, t^{2\ell + 1}]]$ where $\ell >0$. Then
$\calX _R = \{k[[t^2, t^{2q + 1}]] \mid 0 \le q \le \ell\}$. 
\item[$(2)$]  Let $S = k[[X,Y]]$ be the formal power series ring over a field $k$ of $\mathrm{ch} k$ $\ne 2$ and $R = S/(X^2 - Y^{2\ell})$ with $\ell >0$. Let $x, y$ denote the images of $X,Y$ in $R$, respectively. Then $\calX_R = \{R[\frac{x}{y^q}] \mid 0 \le q \le \ell\}$.
\end{enumerate}
\end{cor}

\begin{proof}
(1) Let $R_q = k[[t^2, t^{2q + 1}]]$ for $0 \le q \le \ell$. Then we have the tower 
$$R = R_\ell \subsetneq R_{\ell -1} \subsetneq \ldots \subsetneq R_0 = V$$
of intermediate rings. Hence the result follows from Corollary \ref{12.3}.

(2) Let $R_q = R[\frac{x}{y^q}]$ for $0 \le q \le \ell$. We then have a tower  
$$R = R_0 \subsetneq R_{1} \subsetneq \ldots \subsetneq R_\ell \subseteq  \overline{R}$$
of intermediate rings. Since $\ell_R(\overline{R}/R) = \ell$ (remember that $\overline{R} = S/(X+Y^\ell) \oplus S/(X-Y^\ell)$, because $\mathrm{ch} k \ne 2$), the assertion follows Corollary \ref{12.3}.
\end{proof}

We need one more result.

\begin{prop}\label{12.6}
Let $(R,\fkm)$ be a Gorenstein local ring with $\dim R = 1$ and assume that  there is an element $f \in \fkm$ such that $fR$ is a reduction of $\fkm$. Let $R \subseteq A \subseteq \rmQ(R)$ be an intermediate ring and assume that $\ell_R(A/R) = 1$ and that  $A$ is a local ring with maximal ideal $\fkn$. 
Then the following assertions hold.
\begin{enumerate}
\item[$(1)$] If $A$ has maximal embedding dimension and $fA$ is a reduction of $\fkn$, then $A$ is an almost Gorenstein local ring.
\item[$(2)$] If $\e_\fkm^0(R) = 3$, then $A$ is an almost Gorenstein local ring.
\end{enumerate}
\end{prop}

\begin{proof}
(1) We get $A = R:\fkm$, since $R \subsetneq A \subseteq R:\fkm$ and $\ell_R([R:\fkm]/R) = 1$. Hence $\rmK_A = R:A = R:(R:\fkm) = \fkm$ (\cite[5.19, Definition 2.4
]{HK}). Since $fA$ is a reduction of $\fkn$ and $A$ has maximal embedding dimension, we get $\fkn^2 = f\fkn$, so that $\fkn{\cdot}(\fkm/fA) = (0)$, because $\fkm/fA \subseteq \fkn/fA$. Hence $A$ is an almost Gorenstein local ring (Lemma \ref{3.1} (1)).

(2) We may assume $\rmv(A)\ge 3$, where $\rmv(A)$ denotes the embedding dimension of $A$. Remember
$$3 = \e_\fkm^0(R) = \e_0(fR,A) = \ell_R(A/fA) \ge \ell_A(A/fA) = \e_\fkn^0(A) \ge \rmv(A)\ge 3,$$
so that $\ell_A(A/fA) = \e_\fkn^0(A)= \rmv (A) = 3$.  Hence $fA$ is a reduction of $\fkn$ by a theorem of Rees \cite{Rees} and the assertion follows from assertion (1), because $A$ has maximal embedding dimension.
\end{proof}

We shall now check, for the five cases from ($\rmA_n$) to ($\rmD_n$) separately, that all the intermediate local rings $R \subseteq A \subseteq \overline{R}$ are almost Gorenstein.

{\bf (1) The case  $(\rmA_n)$.} Let $R \subseteq A \subseteq \overline{R}$ be an intermediate local ring such that $A$ is a finitely generated $R$-module. Let $\fkn$ be the maximal ideal of $A$. Then  $\e_\fkn^0(A) \le 2$ by Lemma \ref{12.2} (2), so that $A$ is Gorenstein.

{\bf (2) The cases  $(\rmE_6)$ and $(\rmE_8)$.} Let $V = k[[t]]$ be the formal power series ring over $k$. We then have $S/(X^3 -Y^4) \cong k[[t^3,t^4]]$ and $S/(X^3 - Y^5) \cong k[[t^3,t^5]]$. We begin with the following.

\begin{prop}\label{1.4}
The following assertions hold true.
\begin{enumerate}
\item[$(1)$] $\calX_{k[[t^3,t^4,t^5]]} = \left\{k[[t^3,t^4,t^5]], k[[t^2, t^3]], V\right\}$.
\item[$(2)$] $\calX_{k[[t^3,t^4]]} = \left\{k[[t^3,t^4]], k[[t^3,t^4,t^5]], k[[t^2, t^3]], V\right\}$.
\item[$(3)$] $\calX_{k[[t^3,t^5]]} = \left\{k[[t^3,t^5]], k[[t^3,t^5,t^7]], k[[t^3,t^4,t^5]], k[[t^2, t^3]], V\right\}$.
\end{enumerate}
\end{prop}

\begin{proof}
(1) 
Let $A = k[[t^3,t^4,t^5]]$ and let $B \in \calX_A$ such that $B \ne A$. We choose $f \in B \setminus A$ and write $$f = c_1t + c_2t^2 + g$$ with $c_1, c_2 \in k$ and $g \in A$. If $c_1 \ne 0$, then $fV = tV$, so that $V = k[[f]] \subseteq B$. Suppose $c_1 = 0$. Then $f = c_2t^2 + g$ and $c_2 \ne 0$, so that $t^2 \in B$. Hence $k[[t^2, t^3]] \subseteq B$, which shows $B = k[[t^2, t^3]]$ or $B = V$, because $\ell_{k[[t^2, t^3]]}(V/k[[t^2, t^3]]) = 1$.

(2)
Let $A = k[[t^3,t^4]]$ and let $B \in \calX_{A}$ such that $B \ne A$. We choose $f \in B \setminus A$ and write
$$f = c_1t + c_2t^2 + c_5t^5 + g$$
with $c_i \in k$ and $g \in A$. If $c_1 \ne 0$, then $V = k[[f]] \subseteq B$. Suppose $c_1 = 0$ but $c_2 \ne 0$. Then, rechoosing $f$ so that $c_2 = 1$,we get $t^2 + c_5t^5 \in B$. Hence $t^3(t^2 + c_5t^5) = t^5 + c_5t^8 \in B$. Therefore $t^5 \in B$, because $\left<3,4\right> \ni n$ for all $n \ge 8$ (here $\left<3,4\right>$ denotes the numerical semigroup generated by $3, 4$). Thus $k[[t^3, t^4,t^5]] \subsetneq B$, whence $B =k[[t^2,t^3]]$ or $B = V$ by Lemma \ref{1.3}. Suppose that $c_1 = c_2 = 0$ for any choice of $f \in B \setminus A$. We then have $t^5 \in B$ since $c_5 \ne 0$, so that  $B =k[[t^3,t^4,t^5]]$.   

(3)
Let $A = k[[t^3,t^5]]$ and let $B \in \calX_{A}$ such that $B \ne A$. We choose $f \in B \setminus A$ and write
$$f = c_1t + c_2t^2 + c_4t^4 + c_7t^7 + g$$
with $c_i \in k$ and $g \in A$. If $c_1 \ne 0$, then $B = V$. Suppose $c_1 = 0$ but $c_2 \ne 0$, say $c_2 = 1$. Then, since $t^5f = t^7 + c_4t^9 + c_7t^{12} + t^5g \in B$, we get $t^7 \in B$ (remember that $\left<3,5\right> \ni n$ for all $n \ge 8$). Hence $(t^2 + c_4t^4)^2 \in B$, so that $t^4 \in B$. Therefore $k[[t^2, t^3]] \subseteq B$. If $c_1 = c_2 = 0$ but $c_4 = 1$, then $t^4 + c_7t^7 \in B$, so that $t^7 \in B$ because $t^3(t^4 + c_7t^7) \in B$. Hence $k[[t^3,t^4]] \subseteq B$. Suppose that  $c_1 = c_2 = c_7 = 0$ for any choice of $f \in B \setminus A$. We then have $t^7 \in B$, whence $B =k[[t^3,t^5,t^7]]$.   
\end{proof}

It is standard to check that  $k[[t^3,t^4,t^5]]$ and $k[[t^3,t^5,t^7]]$ are almost Gorenstein local rings, which proves the case   $(\rmE_6)$ or $(\rmE_8)$.

{\bf (3) The case  $(\rmE_7)$.} We consider $F = X^3 - XY^3$.
Let $f=X^2-Y^3$. Then $X, f$ is a system of parameters of $S=k[[X,Y]]$. Therefore $(F)=(X)\cap (f)$. Let $k[[t]]$ be the formal power series ring and we get a tower $$R = S/(F) \subseteq S/(X) \oplus S/(f) = k[[Y]] \oplus k[[t^2,t^3]] \subseteq k[[Y]] \oplus k[[t]] = \overline{R}$$
of rings, 
where we naturally identify $S/(X) = k[[Y]]$ and $S/(f) = k[[t^2, t^3]] \subseteq k[[t]]$. Let $R \subsetneq A \subseteq \overline{R} = k[[Y]] \oplus k[[t]]$ be an intermediate local ring. 
Let $p_2:k[[Y]] \oplus k[[t]] \to k[[t]], ~(a,b) \mapsto b$ be the projection and set $C = p_2(A)$.  Then $k[[t^2,t^3]] \subseteq C \subseteq V$, whence $C=k[[t^2,t^3]]$ or $C = V$ (Corollary \ref{12.5}).

We firstly  consider the case where $C=V$. Let $\fkn$ denote the maximal ideal of $A$.

\begin{claim}
There exists an element $z \in A$ such that $z=(0, t)$
\end{claim} 
\begin{proof}[Proof of Claim 2]
Since $t \in C$, there exists $z \in A$ such that $z=(g,t)$ with $g \in k[[Y]]$. Then $z \in \n$. Suppose $g \neq 0$ and write $g=Y^n\varepsilon$, where $n>0$ and $\varepsilon \in \rmU(k[[Y]])$. Let $\overline{g}$ denotes the image of $g \in S = k[[X,Y]]$ in $A$. Then since $\overline{g} = (g, g(t^2))$ in $S/(X)\oplus k[[t]]$,  we have $z-\overline{g} = (0, t-g(t^2))$ and $t-g(t^2) = t-t^{2n}{\cdot}\varepsilon(t^2) = tu_2$ with $u_2 = 1-t^{2n-1}{\cdot}\varepsilon(t^2)$. Hence because $u_2$ is a unit of $C = k[[t]]$,  we may choose a unit $u \in A$ so that $p_2(u)=u_2$. We then have 
$
(0,t)= u^{-1}(z-\overline{g}) \in A.
$
Thus $z'=u^{-1}(z-\overline{g})$ is a required element of $A$.
\end{proof}

Let $z = (0,t)$. Let $x=\overline{X}$ and $y=\overline{Y}$ denote the images of $X,Y$ in $A$, respectively. Hence $x=(0, t^3), ~y=(Y,t^2)$, and therefore $x=z^3$ and $z(y-z^2)=0$. We consider the $k$-algebra map $\psi: k[[Y,Z]] \to A$ defined by  $\psi(Y)=y$, $\psi(Z)=z$. Then $Z(Y-Z^2) \in \Ker\psi$. We now consider  the following commutative diagram
\[
\xymatrix{
0 \ar[r] & A \ar[r] & \overline{R}=k[[Y]]\oplus V \ar[r] & \overline{R}/A \ar[r] & 0 \\
0 \ar[r] & \frac{k[[Y,Z]]}{(Z(Y-Z^2))} \ar[u]_{\bar{\psi}} \ar[r] & \frac{k[[Y,Z]]}{(Z)}\oplus \frac{k[[Y,Z]]}{(Y-Z^2)} \ar[u]_{\cong} \ar[r] & \frac{k[[Y, Z]]}{(Y,Z)} \ar[r] 
& 0, \\
}
\]
where the rows are canonical exact sequences and $\bar{\psi} : k[[Y,Z]] \to A$ is the homomorphism derived from $\psi$. Then the induced homomorphism $\rho : k[[Y,Z]]/(Y,Z) \to \overline{R}/A$ has to be bijective, because $\overline{R}/A \ne (0)$ (remember that $A$ is a \textit{local} ring) and $\rho$ is surjective. Consequently, $\bar{\psi} : k[[Y,Z]]/(Z(Y - Z^2)) \to  A$ is an isomorphism, so that $A$ is a Gorenstein ring.

Next we consider the case where $C=k[[t^2, t^3]]$. Hence $R \subsetneq A \subsetneq k[[Y]]\oplus k[[t^2, t^3]]$. We set $B = k[[t^2, t^3]]$ and $T = k[[Y]]\oplus B$. Remember that $\ell_R(T/R) = 3$, whence $\ell_R(A/R) = 1$ or $2$. If $\ell_R(A/R) = 1$, then by Proposition \ref{12.6} (2) $A$ is an almost Gorenstein local ring.

Suppose that $\ell_R(A/R)=2$. Hence $\ell_R(T/A)=1$. Therefore  as  $$\ell_R(T/A) = [A/\n : R/\m]{\cdot}\ell_A(T/A),$$ we get $R/\fkm \cong A/\fkn$ and $\ell_A(T/A)=1$, whence $\fkn = (0):_AT/A$ is an ideal of $T$.  Let $J$ denote the Jacobson radical of $T$ and consider the exact sequence
$$0 \to A/\fkn \to T/\fkn \to T/A \to 0$$
of $A$-modules. We then have $\ell_A(T/\fkn) = 2$, so that $\fkn = J$, because $\fkn \subseteq J$ and $\ell_A(T/J) = \ell_R(T/J)= 2$. Hence $A = k + J$ and $\fkn =((0,t^3), (Y,0), (0,t^2))$.
Let $\psi: k[[X, Y, Z]] \to A$ be the $k$-algebra map defined by $\psi(X)=(0, t^3)$, $\psi(Y)=(Y, 0)$, $\psi(Z)=(0, t^2)$. Then $X^2-Z^3, XY, YZ\in \Ker\psi$ and we get  the following commutative diagram
\[
\xymatrix{
0 \ar[r] & A \ar[r] & T= k[[Y]]\oplus k[[t^2, t^3]] \ar[r] & T/A \ar[r] & 0 \\
0 \ar[r] & \frac{k[[X, Y, Z]]}{(X, Z)\cap (X^2-Z^3, Y)} \ar[u]_{\bar{\psi}} \ar[r] & \frac{k[[X, Y, Z]]}{(X, Z)}\oplus \frac{k[[X, Y, Z]]}{(X^2-Z^3, Y)} \ar[u]_{\cong} \ar[r] & \frac{k[[X, Y, Z]]}{(X, Y, Z)} \ar[r] 
& 0. \\
}
\]
For the same reason as above, the induced homomorphism $k[[X, Y,Z]]/(X, Y,Z) \to T/A$ has to be bijective, so that 
$
A \cong k[[X, Y, Z]]/{(X, Z)\cap (X^2-Z^3, Y)}.
$
Notice now that 
$$(X, Z)\cap (X^2-Z^3, Y) = \rmI_2\left(\begin{smallmatrix}
Z^2 & X & Y\\
X&Z&0
\end{smallmatrix}\right) =\rmI_2\left(\begin{smallmatrix}
Z^2 & X & Y\\
X + Z^2 & X + Z & Y
\end{smallmatrix}\right).$$
Then thanks to the theorem of Hilbert-Burch and Theorem \ref{p}, $A$ is an almost Gorenstein local ring, because $X + Z^2, X + Z, Y$ is a regular system of parameters of $k[[X, Y, Z]]$. This completes the proof of the case $(\rmE_7)$.


{\bf (4) The case  $(\rmD_n)$.} 

{\bf (i) The case where $n = 2\ell +1$ with $\ell \ge 1$.} 
We consider $F = Y(X^2-Y^{2\ell +1})$.
Let $f=X^2-Y^{2\ell +1 }$. Then $X, f$ is a system of parameters of $S=k[[X,Y]]$. Therefore $(F)=(Y)\cap (f)$ and we get a tower $$R = S/(F) \subseteq S/(Y) \oplus S/(f) = k[[X]] \oplus k[[t^2,t^{2\ell + 1}]] \subseteq k[[X]] \oplus k[[t]] = \overline{R}$$
of rings, 
where we naturally identify $S/(Y) = k[[X]]$ and $S/(f) = k[[t^2, t^{2\ell + 1}]] \subseteq k[[t]]$. 
Let $R \subsetneq A \subseteq \overline{R}$ be an intermediate ring and assume that $(A, \n)$ is a local ring. Let $p_2: \overline{R} \to V$ be the projection and set $B= p_2(A)$. Then since $k[[t^2, t^{2\ell +1}]] \subseteq B \subseteq V$, by Corollary \ref{12.5} (1)  $B = k[[t^2, t^{2q+1}]]$ for some $0 \le q \le \ell$. We choose an element $z \in A$ so that $z=(g, t^{2q+1})$ in $\overline{R}= k[[X]] \oplus k[[t]]$ with $g \in k[[X]]$. Suppose $g \neq 0$ and write $g=X^n \varepsilon~(n>0, \varepsilon \in \rmU(k[[X]])$. Denote by $\overline{g}$ the image of $g \in S=k[[X,Y]]$ in $A$. We have 
\begin{eqnarray*}
z-\overline{g} &=& z - (g, (t^{2\ell +1})^n {\cdot}\varepsilon(t^{2\ell + 1})) \\
 &=& (0, t^{2q+1}(1-t^{(2\ell+1)n-(2q+1)}{\cdot}\varepsilon(t^{2\ell + 1}))).
\end{eqnarray*}
Here we notice that $(2\ell+1)n-(2q+1) \ge 0$ and that $(2\ell+1)n-(2q+1)= 0$ if and only if $n=1$ and $\ell = q$.

If $(2\ell+1)n-(2q+1) > 0$, we set $u_2 = 1-t^{(2\ell+1)n-(2q+1)}{\cdot}\varepsilon(t^{2\ell + 1})$. Then $u_2 \in \rmU(B)$. We choose an element $u \in \rmU(A)$ so that $u_2=p_2(u)$. Then 
$$
z' = u^{-1}(z-\overline{g})=(0, t^{2q+1}).
$$
Therefore, replacing $z$ with $z'$, we may assume without loss of generality that $z \in A$ such that $z=(0, t^{2q+1})$. Let $x=\overline{X}$ and $y=\overline{Y}$, where $\overline{X}, \overline{Y}$ respectively denote the images of $X, Y \in S = k[[X,Y]]$ in $A$. Then $x=(X, t^{2q+1})$ and $y=(0, t^2)$, so that  
$$
y^{2q+1}=z^2, ~~ y(x-y^{\ell-q}z)=0, ~~~\text{and}~~~z(x-y^{\ell-q}z)=0.
$$
Let $\psi: k[[X,Y,Z]] \to A$ be the $k$-algebra map defined by $\psi(X)=x$, $\psi(Y)=y$, $\psi(Z)=z$. Then $\Ker\psi \supseteq (Y,Z)\cap (Z^2-Y^{2q+1}, X-Y^{\ell-q}Z)$. Therefore, considering the commutative diagram
\[
\resizebox{\hsize}{!}{
\xymatrix{
0 \ar[r] & A \ar[r] & T=k[[X]]\oplus k[[t^2, t^{2q + 1}]] \ar[r] & T/A \ar[r] & 0 \\
0 \ar[r] & \frac{k[[X,Y,Z]]}{(Y,Z)\cap (Z^2-Y^{2\ell +1},X-Y^{\ell -q}Z)} \ar[u]_{\bar{\psi}} \ar[r] & \frac{k[[X,Y,Z]]}{(Y, Z)}\oplus \frac{k[[X,Y,Z]]}{(Z^2-Y^{2\ell +1},X-Y^{\ell -q}Z)} \ar[u]_{\cong} \ar[r] & \frac{k[[X,Y,Z]]}{(X,Y,Z)} \ar[r] 
& 0, \\
}}
\]
we see that 
$$
A \cong k[[X,Y,Z]]/{(Y,Z)\cap (X-Y^{\ell-q}Z, Z^2-Y^{2q+1})},
$$
because $T/A \neq (0)$. Notice now that 
$$(X, Z)\cap (X^2-Z^3, Y) = \rmI_2\left(\begin{smallmatrix}
Y^{2q}&Z&X-Y^{\ell - q}\\
Z &Y& 0
\end{smallmatrix}\right) =\rmI_2\left(\begin{smallmatrix}
Y^{2q} &Z &X-Y^{\ell-q}Z\\
Z-Y^{2q} & Y-Z & Y^{\ell-q}Z -X
\end{smallmatrix}\right).$$
Then  by Theorem \ref{p} $A$ is an almost Gorenstein local ring, because $Z-Y^{2q}, Y-Z,Y^{\ell-q}Z -X$ is a regular system of parameters of $k[[X, Y, Z]]$.

If $n=1$ and $\ell=q$, then  $R \subsetneq A \subsetneq k[[X]]\oplus k[[t^2, t^{2\ell+1}]]$, so that $\ell_R(A/R)=1$ (remember that $\ell_R((k[[X]]\oplus k[[t^2, t^{2\ell+1}]])/R) = 2$). Hence $A$ is an almost Gorenstein local ring by Proposition \ref{12.6} (2).

{\bf (ii) The case where $n = 2\ell$ with $\ell \ge 1$.} 
Let $f=X^2-Y^{2\ell}=(X+Y^{\ell})(X-Y^{\ell})$ and $T = S/(f)$. Since $\operatorname{ch} k \ne 2$, $X+Y^\ell$, $X-Y^\ell$ is a system of parameters of $S=k[[X,Y]]$, so we get the exact sequence
$$
0 \to T \stackrel{\alpha} \longrightarrow S/(X+Y^\ell)\oplus S/(X-Y^\ell) \stackrel{\beta} \longrightarrow S/(X, Y^\ell) \to 0.
$$ Hence $\ell_T(\overline{T}/T)=\ell$.
We look at the tower 
$$R \subseteq k[[X]] \oplus T \subseteq k[[X]] \oplus \overline{T} = \overline{R} $$
of rings and consider an intermediate ring $R \subsetneq A \subsetneq \overline{R}$ such that $(A, \n)$ is a local ring. 
Let $p_2 : k[[X]] \oplus \overline{T} \to \overline{T}$ be the projection and set  $B= p_2(A)$.  We denote by $x, y$ the images of $X,Y$ in $T=S/(f)$, respectively. Then by Corollary \ref{12.5} (2) $B = T[\frac{x}{y^{q}}]$ for some $0 \le q \le \ell$. Here we notice that $q < \ell$, since $A \ne \overline{R}$. We choose an element $z \in A$ so that $z=(g, \frac{x}{y^q})$, where $g \in k[[X]]$. If $g \neq 0$, then we write $g=X^n \varepsilon~(n>0, \varepsilon \in \rmU(k[[X]])$. Let $\overline{g}$ be the image of $g \in S$ in $A$. We then have 
$$
z-\overline{g} = (g, \frac{x}{y^q}) - (g, x^n{\cdot}\varepsilon(x)) = (0, \frac{x}{y^q}{\cdot}(1-(\frac{x}{y^q})^{n-1}){\cdot}y^{nq}{\cdot}\varepsilon(x)).
$$

Suppose now that 
$
1-(\frac{x}{y^q})^{n-1})y^{nq}\varepsilon(x) \in \rmU(B)
$
(this is the case if $n>1$ or if $n=1$ and $q>0$). We then have $(0, \frac{x}{y^q}) \in A$ for the same reason as above. Let $x_1, y_1$ be the images of $X,Y \in S$ in $A$, respectively. Hence $x_1=(X, x)$ and $y_1=(0, y)$, so that  
$$
z^2-y_1^{2(\ell-q)}, ~~ y_1(x_1-y_1^qz)=0, \ \ \text{and}\ \  z(x_1-y_1^qz)=0.
$$
Let $\psi: k[[X,Y,Z]] \to A$ be the $k$-algebra map defined by $\psi(X)=x_1$, $\psi(Y)=y_1$, $\psi(Z)=z$. Then $\Ker\psi \supseteq (Y,Z)\cap (Z^2-Y^{2(\ell-q)}, X-Y^qZ)$, and by the commutative diagram
\[
\resizebox{\hsize}{!}{
\xymatrix{
0 \ar[r] & A \ar[r] & k[[X]]\oplus B \ar[r] & D \ar[r] & 0 \\
0 \ar[r] & \frac{k[[X,Y,Z]]}{(Y,Z)\cap (Z^2-Y^{2(\ell-q)},X-Y^qZ)} \ar[u]_{\bar{\psi}} \ar[r] & \frac{k[[X,Y,Z]]}{(Y, Z)}\oplus \frac{k[[X,Y,Z]]}{(Z^2-Y^{2(\ell-q)},X-Y^qZ)} \ar[u]_{\cong} \ar[r] & \frac{k[[X,Y,Z]]}{(X,Y,Z)} \ar[r] 
 & 0 \\
}}
\]
we get
$$
A \cong k[[X,Y,Z]]/{(Y,Z)\cap (X-Y^{q}Z, Z^2-Y^{2(\ell-q)})}.
$$
Notice that
{\footnotesize 
$$(Y, Z)\cap (Z^2-Y^{2\ell - q},X- Y^qZ) = \rmI_2\left(\begin{smallmatrix}
Y &Z& 0\\
Z&Y^{2\ell - q)-1}&X - Y^qZ
\end{smallmatrix}\right) =\rmI_2\left(\begin{smallmatrix}
Y & Z & 0\\
Z + Y&Z + Y^{2(\ell - q) - 1}&X-Y^{q}Z
\end{smallmatrix}\right).$$
} If  $\ell - q =  1$, then $Z, Y^{2\ell - q)-1}, X-Y^{q}Z$ is a regular system of parameters of $k[[X,Y,Z]]$ and if $\ell - q \ge 2$, then $Z+Y, Y^{2(\ell-q)-1}+Z, X-Y^{q}Z$ is a regular system of parameters of $k[[X,Y,Z]]$, so that $A$ is an almost Gorenstein local ring in any case.

If $n=1$ and $q=0$, then $R \subsetneq A \subsetneq k[[X]] \oplus T$, so that $\ell_R(A/R)=1$, because  $\ell_R(\left(k[[X]] \oplus T\right)/R) = 2$. Therefore $A$ is an almost Gorenstein local ring by Proposition \ref{12.6} (2). This completes the proof of Theorem \ref{12.1} as well as the proof of the case $(\rmD_n)$.



\end{document}